\documentclass{amsart}

\usepackage{bbm} 
\usepackage{nicefrac}
\usepackage{hyperref}

\newcommand*{\mailto}[1]{\href{mailto:#1}{\nolinkurl{#1}}}
\newcommand{\arxiv}[1]{\href{http://arxiv.org/abs/#1}{arXiv:\,#1}}

\newtheorem{theorem}{Theorem}[section]

\newtheorem{lemma}[theorem]{Lemma}
\newtheorem{proposition}[theorem]{Proposition}
\newtheorem{corollary}[theorem]{Corollary}
\newtheorem{remark}[theorem]{Remark}
\newtheorem{hypothesis}[theorem]{Hypothesis}

\newcommand{\R}{{\mathbb R}}
\newcommand{\N}{{\mathbb N}}
\newcommand{\Z}{{\mathbb Z}}
\newcommand{\C}{{\mathbb C}}

\newcommand{\be}{\begin{equation}}
\newcommand{\ee}{\end{equation}}

\newcommand{\spr}[2]{\langle #1 , #2 \rangle}

\newcommand{\E}{\mathrm{e}}
\newcommand{\I}{\mathrm{i}}

\newcommand{\tr}{\mathrm{tr}}
\newcommand{\im}{\mathrm{Im}}
\newcommand{\re}{\mathrm{Re}}
\newcommand{\dom}[1]{\mathrm{dom}(#1)}
\newcommand{\mul}[1]{\mathrm{mul}(#1)}

\DeclareMathOperator{\ran}{ran}

\DeclareMathOperator{\rang}{rank}

\DeclareMathOperator{\reg}{r}
\newcommand{\LRel}{\text{LR}}

\newcommand{\sigdis}{\sig_d}
\newcommand{\sigess}{\sig_e}

\DeclareMathOperator{\linspan}{span}
\DeclareMathOperator{\supp}{supp}

\newcommand{\floor}[1]{\lfloor#1 \rfloor}

\newcommand{\qd}{{[1]}}
\newcommand{\Tpre}{T_0}
\newcommand{\Tmin}{T_{\mathrm{min}}}
\newcommand{\Tmax}{T_{\mathrm{max}}}
\newcommand{\Tloc}{T_{\mathrm{loc}}}
\newcommand{\Llocp}{L^1_{\mathrm{loc}}((a,b);|\varsigma|)}
\newcommand{\Llocr}{L^1_{\mathrm{loc}}((a,b);\varrho)}
\newcommand{\AClocp}{AC_{\mathrm{loc}}((a,b);\varsigma)}
\newcommand{\AClocr}{AC_{\mathrm{loc}}((a,b);\varrho)}
\newcommand{\BCa}{BC_{a}}
\newcommand{\BCb}{BC_{b}}

\newcommand{\Lr}{L^2((a,b);\varrho)}
\newcommand{\Lrmu}{L^2(\R;\mu)}
\newcommand{\Deftau}{\D_\tau}
\newcommand{\indik}{\mathbbm{1}}
\newcommand{\D}{\mathfrak{D}}
\newcommand{\M}{\mathrm{M}}

\newcommand{\eps}{\varepsilon}
\newcommand{\vphi}{\varphi}
\newcommand{\sig}{\sigma}
\newcommand{\lam}{\lambda}

\numberwithin{equation}{section}

\begin{document}

\title{Sturm--Liouville operators with measure-valued coefficients}

\author[J.\ Eckhardt]{Jonathan Eckhardt}
\address{Faculty of Mathematics\\ University of Vienna\\
Oskar-Morgenstern-Platz 1\\ 1090 Wien\\ Austria}
\email{\mailto{jonathan.eckhardt@univie.ac.at}}

\author[G.\ Teschl]{Gerald Teschl}
\address{Faculty of Mathematics\\ University of Vienna\\
Oskar-Morgenstern-Platz 1\\ 1090 Wien\\ Austria\\ and International Erwin Schr\"odinger
Institute for Mathematical Physics\\ Boltzmanngasse 9\\ 1090 Wien\\ Austria}
\email{\mailto{Gerald.Teschl@univie.ac.at}}
\urladdr{\url{http://www.mat.univie.ac.at/~gerald/}}

\thanks{J. d'Analyse Math. {\bf 120}, 151--224 (2013)}
\thanks{{\it Research supported by the Austrian Science Fund (FWF) under Grant No.\ Y330}}


\begin{abstract}
We give a comprehensive treatment of Sturm--Liouville operators whose coefficients are measures
including a full discussion of self-adjoint extensions and boundary conditions, resolvents, 
and Weyl--Titchmarsh--Kodaira theory. We avoid previous technical restrictions and, at the same time,
extend all results to a larger class of operators. Our operators include classical Sturm--Liouville operators,
Sturm--Liouville operators with (local and non-local) $\delta$ and $\delta'$ interactions or transmission conditions as well as
eigenparameter dependent boundary conditions,
Krein string operators, Lax operators arising in the treatment of the Camassa--Holm equation, Jacobi operators, and
Sturm--Liouville operators on time scales as special cases.
\end{abstract}

\maketitle

\section{Introduction}
\label{sec:int}

Sturm--Liouville problems
\be
 -\frac{d}{dx} \left(p(x) \frac{dy}{dx}(x)\right) +q(x) y(x) = zr(x) y(x)
\ee
have a long tradition (see e.g.\ the textbooks \cite{tschroe}, \cite{we80}, \cite{zet} and the references therein)
and so have their generalizations to measure-valued coefficients. In fact, extensions to the case
\be
\frac{d}{d\varrho(x)} \left( - \frac{dy}{d\varsigma(x)}(x) + \int^x y(t) d\chi(t) \right) = z y(x)
\ee
date back at least to Feller \cite{feller} and were also advocated in the fundamental monograph
by Atkinson \cite{at}. Here the derivatives on the left-hand side have to be understood as Radon--Nikod\'{y}m derivatives.
We refer to the book by Mingarelli \cite{ming} for a more detailed historical discussion. In fact, in those references, 
the measure $\varsigma$ was always assumed to be absolutely continuous, $d\varsigma(x) = p(x)^{-1} dx$,
such that $y$ will at least be continuous. We will not make this restriction here since it would exclude (e.g.) the case
of $\delta'$ interactions which constitute a popular physical model.

However, while the generalization to measure-valued coefficients has been very successful on the level of differential equations (see e.g.\
\cite{at}, \cite{ming}, \cite{volkmer2005} and the references therein), much less is known about the associated operators in an
appropriate Hilbert space. First attempts were made by Feller and later complemented by Kac \cite{kac} (cf.\ also Langer \cite{lan}
and Bennewitz~\cite{bennewitz1989}).
Again, a survey of these results and further information can be found in the book of Mingarelli \cite{ming}.

The case where only the potential is allowed to be a measure is fairly well treated since it allows to include the case
of point interactions which is an important model in physics (see e.g.\ the monographs \cite{aghh}, \cite{ak}
as well as the recent results in \cite{bare} and the references therein).
More recently, Savchuk and Shkalikov \cite{ss1}--\cite{ss5}, Goriunov and Mikhailets \cite{gm1}, \cite{gm2} as well as
Mikhailets and Molyboga \cite{mm1}, \cite{mm2} were even able to cover the case where the potential is the derivative
of an arbitrary $L^2$ function. However, note that while this covers the case of $\delta$ interactions, it does
not cover the case of $\delta'$ interactions which are included in the present approach. 
 Moreover, since we allow all three coefficients to be measures, our approach even includes Schr\"odinger operators with non-local $\delta'$ interactions on arbitrary sets of Lebesgue measure zero as studied recently by Albeverio and Nizhnik \cite{albniz} and Brasche and Nizhnik \cite{braniz}.
 This connection will be discussed in detail and exploited in a forthcoming paper \cite{DeltaPrime}. 
 Finally, the case where the weight coefficient is a measure is known as Krein string and
has also attracted considerable interest recently \cite{vs1}--\cite{vs3} (see also the monograph \cite{dymck}). 

However, while the theory developed by Kac and extended by Mingarelli is quite general, it still does exclude
several cases of interest. More precisely, the basic assumptions in Chapter~3 of Mingarelli \cite{ming} require that
the corresponding measures have no weight at a finite boundary point. Unfortunately, this assumption excludes for
example classical cases like Jacobi operators on a half-line. The reason for this assumption is the fact that otherwise
the corresponding maximal operator will be multi-valued and one has to work within the framework of multi-valued
operators. This problem is already visible in the case of half-line Jacobi operators where the underlying Hilbert space has
to be artificially expanded in order to be able to formulate appropriate boundary conditions \cite{tjac}. In our case there is
no natural way of extending the Hilbert space and the intrinsic approach via multi-valued operators is more natural.
Nevertheless, this multi-valuedness is not too severe and corresponds to an at most two dimensional space which
can be removed to obtain a single-valued operator, again, a fact well-known from Jacobi operators with finite endpoints.
Finally, this general approach will also allow us to include a large variety of boundary conditions, including the case
of eigenparameter dependent boundary conditions.

Moreover, the fact that our differential equation is defined on a larger set than the support of the measure $\varrho$
(which determines the underlying Hilbert space) is also motivated by requirements from the applications we have in
mind. The most drastic example in this respect is the Sturm--Liouville problem
\be
 \frac{d}{d\varrho(x)} \left( - \frac{dy}{dx}(x) + \frac{1}{4} \int^x y(t) dt \right) = z y(x)
\ee
on $\R$ which arises in the Lax pair of the dispersionless Camassa--Holm equation \cite{ch}, \cite{chh}. 
In case of the well-known one-peakon solution, $\varrho$ is a single Dirac measure and the underlying Hilbert space is one-dimensional. However,
the corresponding differential equation has to be investigated on all of $\R$, where the Camassa--Holm equation
is defined. An appropriate spectral theory for this operator in the case where $\varrho$ is an arbitrary measure seems to be missing and is one of the main motivations
for the present paper.

Furthermore, there is of course another reason why Sturm--Liouville equations with measure-valued coefficients
are of  interest, namely, the unification of the continuous with the discrete case. While such a unification already
 was one of the main motivations in Atkinson \cite{at} and Mingarelli \cite{ming}, it has recently attracted
 enormous attention via the introduction of the calculus on time scales \cite{bope}.
In fact, given a time scale $\mathbb{T}\subseteq\R$, the so-called associated Hilger (or delta) derivative is
nothing but the Radon--Nikod\'{y}m derivative with respect to the measure $\varrho$, which corresponds to the distribution function $R(x) = \inf \{ y\in \mathbb{T} \,|\, y > x\}$.
We refer to \cite{ekte} for further details and to a follow-up publication \cite{etslts}, where we will provide further details
on this connection.

Finally, our approach also includes a number of generalizations for Sturm--Liouville problems which have been attracting significant interest in the past. 
 In particular, our approach covers boundary conditions depending polynomially on the eigenparameter as well as internal discontinuities (also known as transmission conditions) as introduced by Hald
\cite{hal} in his study of the inverse problem for the torsional modes of the earth (cf.\ \cite{sjt}).

As one of our central results we will develop singular Weyl--Titchmarsh--Kodaira theory for these operators extending
the recent work of Kostenko, Sakhnovich, and Teschl  \cite{kst} (see also Kodaira \cite{kod}, Kac \cite{kac}, and
Gesztesy and Zinchenko \cite{gz}). In particular, we will
cover singular settings where the Weyl--Titchmarsh--Kodaira function is no longer a Herglotz--Nevanlinna function (or a
generalized Nevanlinna function). Again this general approach is motivated by applications to the dispersionless
Camassa--Holm equation, where the associated spectral measure can exhibit arbitrary growth. 

\section{Notation}

Let $(a,b)$ be an arbitrary interval and $\mu$ be a locally finite complex Borel measure on $(a,b)$.
By $AC_{\mathrm{loc}}((a,b);\mu)$ we denote the set of left-continuous functions, which are locally absolutely continuous with respect to $\mu$. 
These are precisely the functions $f$ which can be written in the form
\begin{align*}
 f(x) = f(c) + \int_c^x h(s) d\mu(s), \quad x\in(a,b),
\end{align*}
where $h\in L^1_{\mathrm{loc}}((a,b); \mu)$ and the integral has to be read as
\be
 \int_c^x h(s) d\mu(s) = \begin{cases}
                           \int_{[c,x)} h(s)d\mu(s),             & \text{if }x>c, \\
                           0, & \text{if }x=c, \\
                           - \int_{[x,c)} h(s)d\mu(s),             & \text{if }x<c.
                         \end{cases}
\ee
The function $h$ is the Radon--Nikod\'{y}m derivative of $f$ with respect to $\mu$. It is uniquely defined in $L^1_{\mathrm{loc}}((a,b); \mu)$ and we write
\begin{align*}
 \frac{d f}{d\mu} = h.
\end{align*}
Every function $f$ which is locally absolutely continuous with respect to $\mu$ is locally of bounded variation and hence also the right-hand limits 
\begin{align*}
 f(x+)= \lim_{\eps\downarrow 0} f(x+\eps), \quad x\in(a,b)
\end{align*}
of $f$ exist everywhere. 
Also note that some function $f\in AC_{\mathrm{loc}}((a,b);\mu)$ can only be discontinuous in some point, if $\mu$ has mass in this point. 

In this respect we also recall the integration by parts formula (\cite[Theorem~21.67]{hesr}) for two locally finite complex Borel measures $\mu$, $\nu$ on $(a,b)$ 
\begin{align}\label{eqnLDpartint}
\int_{\alpha}^\beta F(x)d\nu(x) = \left. FG\right|_\alpha^\beta - \int_{\alpha}^\beta G(x+) d\mu(x), \quad \alpha,\,\beta\in(a,b),
\end{align}
where $F$, $G$ are left-continuous distribution functions of $\mu$, $\nu$ respectively.

\section{Sturm--Liouville equations with measure-valued coefficients}

Let $(a,b)$ be an arbitrary interval and $\varrho$, $\varsigma$ and $\chi$ be locally finite complex Borel measures 
 on $(a,b)$. We want to define a linear differential expression $\tau$ which is informally given by
\begin{align*}
\tau f =  \frac{d}{d\varrho} \left(-\frac{df}{d\varsigma} + \int fd\chi \right).
\end{align*}
 In this section, we will successively add assumptions on our measure coefficients as soon as they are needed.
 All of them are included in Hypothesis~\ref{genhyp} below, which will then be in force throughout the rest of this paper. 
 However, up to now the only additional assumptions we impose on our measures is that $\varsigma$ is supported on the whole interval,
 i.e., $\supp(\varsigma)=(a,b)$.

The maximal domain $\Deftau$ of functions such that $\tau f$ makes sense consists of all functions 
$f\in \AClocp$ for which the function
\begin{align}\label{eqn:tauqdptheta}
 -\frac{df}{d\varsigma}(x) + \int_c^x fd\chi, \quad x\in(a,b)
\end{align}
is locally absolutely continuous with respect to $\varrho$, i.e., there is some representative of this function 
lying in $\AClocr$. As a consequence of the assumption $\supp(\varsigma)=(a,b)$, this representative is unique. 
We then set $\tau f\in\Llocr$ to be the Radon--Nikod\'{y}m derivative of this function with respect to $\varrho$.
One easily sees that this definition is independent of $c\in(a,b)$ since the corresponding 
functions~\eqref{eqn:tauqdptheta} as well as their unique representatives only differ by an additive constant.
As usual, we denote the Radon--Nikod\'{y}m derivative with respect to $\varsigma$ of some function $f\in\Deftau$ by
\begin{align*}
 f^\qd = \frac{df}{d\varsigma} \in \Llocp.
\end{align*}
The function $f^\qd$ is called the first quasi-derivative of $f$. 

We note that the definition of $\tau$ is consistent with classical theory. 
Indeed, let $\varrho$, $\varsigma$ and $\chi$ be locally absolutely continuous with respect to Lebesgue measure, 
and denote by $r$, $p^{-1}$ and $q$ the respective densities i.e.,
\begin{align*}
 \varrho(B) = \int_B r(x)dx, \quad \varsigma(B) = \int_B \frac{1}{p(x)}dx \quad\text{and}\quad \chi(B) = \int_B q(x)dx
\end{align*}
for each Borel set $B$.
Then some function $f$ lies in $\Deftau$ if and only if $f$ as well as its quasi-derivative $f^\qd = pf'$ are 
 locally absolutely continuous (with respect to Lebesgue measure).
In this case
\begin{align*}
 \tau f(x) = \frac{1}{r(x)} \left( - \frac{d}{dx} \left(p(x)\frac{df}{dx}(x)\right) + q(x)f(x) \right), \quad x\in(a,b)
\end{align*}
is the usual Sturm--Liouville differential expression. Also note that if we add a single point mass $\delta_c$ located at $c$ to $\varsigma$
\[
\varrho(B) = \int_B r(x) dx, \quad \varsigma(B) = \int_B \frac{1}{p(x)} dx + \alpha\delta_c(B), \quad\text{and}\quad \chi(B) = \int_B q(x)dx,
\]
then we obtain the following jump condition
\[
f(c+)-f(c)= \alpha f^\qd(c)
\]
 at $c$, and hence this corresponds to a $\delta'$ interaction of strength $\alpha$. Similarly, adding a point mass $\alpha \delta_c$ to $\chi$ we obtain
a $\delta$ interaction associated with the jump condition 
\[
f^\qd(c+)-f^\qd(c)= \alpha f(c).
\]
Considering a piecewise continuous weight function $r(x)=p(x)$ with a jump $r(c+)= \beta^{-1} r(c)$ and adding a point mass $r(c)\frac{\alpha}{\beta} \delta_c$ to $\chi$ we obtain a transmission condition
\[
f(c+)= f(c), \qquad f'(c+)- \beta f'(c)= \alpha f(c).
\]
The more general case where $f$ is also allowed to jump can be reduce to this one by virtue of a Liouville-type transformation (cf.\ Remark~2.4 in \cite{sjt}).
Moreover, our approach also incorporates quite general boundary conditions including cases where the boundary conditions depend polynomially on
the eigenvalue parameter (cf.\ Remark~\ref{rem:bcev}).

As another special case, choosing the measures
\begin{align*}
 \varrho(B) = \sum_{n\in\Z} \delta_n(B), \quad \varsigma(B) = \int_B \frac{1}{p_{\floor{x}}} dx \quad\text{and}\quad \chi(B) = \sum_{n\in\Z} q_n \delta_n(B),
\end{align*}
where $p_n\not=0$, $q_n\in\R$ and $\delta_n$ is the Dirac measure in $n\in\Z$, we obtain the usual Jacobi difference expression.
In fact, $\tau f(n)$ at some point $n\in\Z$ is equal to the jump of the function
\begin{align*}
 -p_{\floor{x}} f'(x) + \sum_{ 0 \leq n < x} q_n f(n), \quad x\in\R
\end{align*}
in that point and hence
\begin{align*}
 \tau f(n) = p_{n-1}(f(n)-f(n-1)) - p_n(f(n+1)-f(n)) +q_nf(n).
\end{align*}

Now from the theory of linear measure differential equations (see Appendix~\ref{appMODE} for the required results) we get an existence and uniqueness theorem for differential equations associated with $\tau$.

\begin{theorem}\label{thm:exisuniq}
 Fix some arbitrary function $g\in \Llocr$. Then there is a unique solution  $f\in\Deftau$ of the initial value problem
\begin{align}
 (\tau - z) f = g \quad\text{with}\quad f(c)=d_1 \quad\text{and}\quad f^\qd(c)=d_2
\end{align}
for each $z\in\C$, $c\in(a,b)$ and $d_1$, $d_2\in\C$ if and only if 
\begin{align}\label{eqntauEEcond}
 \varrho(\lbrace x\rbrace) \varsigma(\lbrace x\rbrace) = 0 \quad\text{and}\quad \chi(\lbrace x\rbrace) \varsigma(\lbrace x\rbrace) \not=1
\end{align} 
for all $x\in(a,b)$.
If in addition $g$, $d_1$, $d_2$ and $z$ are real, then the solution is real.
\end{theorem}

\begin{proof}
 Some function $f\in\Deftau$ is a solution of $(\tau-z)f=g$ 
   with $f(c)=d_1$ and $f^\qd(c)=d_2$ if and only if for each $x\in(a,b)$ 
\begin{align*}
f(x) & = d_1 + \int_c^x f^\qd d\varsigma,             \\
 f^\qd(x) & = d_2 + \int_c^x f d\chi - \int_c^x \left(zf + g\right) d\varrho.
\end{align*}
Now set $\omega=|\varsigma|+|\chi|+|\varrho|$ and let $m_{12}$, $m_{21}$ and $f_2$ be the Radon--Nikod\'{y}m 
 derivatives of $\varsigma$, $\chi-z\varrho$ and $g\varrho$ with respect to $\omega$.
Then these equations can for each $x\in(a,b)$ be written as
\begin{align*}
 \begin{pmatrix} f(x) \\ f^\qd(x) \end{pmatrix} = \begin{pmatrix} d_1 \\ d_2 \end{pmatrix}
   + \int_c^x \begin{pmatrix}  0 & m_{12} \\ m_{21} & 0 \end{pmatrix} \begin{pmatrix} f \\ f^\qd \end{pmatrix}d\omega
   + \int_c^x \begin{pmatrix} 0 \\ f_2 \end{pmatrix}d\omega.
\end{align*}
Hence the claim follows from Theorem~\ref{thm:appMIEEE}, since~\eqref{eqntauEEcond} holds for all $x\in(a,b)$ if
 and only if
 \begin{align*}
 I + \omega(\lbrace x\rbrace) \begin{pmatrix}
      0 & m_{12}(x) \\ m_{21}(x) & 0
     \end{pmatrix} =
  \begin{pmatrix}
   1 & \varsigma(\lbrace x\rbrace) \\ \chi(\lbrace x\rbrace) - z\varrho(\lbrace x\rbrace) & 1
  \end{pmatrix}
 \end{align*}
 is invertible for all $z\in\C$ and $x\in(a,b)$.
\end{proof}

Note that if $g\in\Llocr$ and~\eqref{eqntauEEcond} holds for each
 $x\in(a,b)$, then there is also a unique solution of the initial value problem
\begin{align*}
(\tau-z)f=g \quad\text{with}\quad f(c+)=d_1 \quad\text{and}\quad f^\qd(c+)=d_2
\end{align*}
for every $z\in\C$, $c\in(a,b)$, $d_1$, $d_2\in\C$ by Corollary~\ref{corappIniValProRC}.

Because of Theorem~\ref{thm:exisuniq}, in the following we will always assume that the measure $\varsigma$ 
has no point masses in common with $\varrho$ or $\chi$, i.e., 
\begin{align}\label{eqntauEEcondPMincommon}
  \varsigma(\lbrace x\rbrace) \varrho(\lbrace x\rbrace) = \varsigma(\lbrace x\rbrace) \chi(\lbrace x\rbrace) = 0
\end{align}
for all $x\in(a,b)$. This assumption is stronger than the one needed in Theorem~\ref{thm:exisuniq} but we will
need it for the Lagrange identity below.

For $f$, $g\in\Deftau$ we define the Wronski determinant
\begin{align}
 W(f,g)(x) = f(x)g^\qd(x) - f^\qd(x)g(x), \quad x\in (a,b).
\end{align}
This function is locally absolutely continuous with respect to $\varrho$ with 
\begin{align*}
 \frac{d\, W(f,g)}{d\varrho} = g\, \tau f - f\, \tau g. 
\end{align*}
Indeed, this is a simple consequence of the following Lagrange identity.

\begin{proposition}\label{propLagrange}
 For each $f$, $g\in\Deftau$ and $\alpha$, $\beta\in(a,b)$ we have 
\begin{align}\label{eqn:lagrange}
 \int_\alpha^\beta \left(g(x) \tau f(x)  - f(x) \tau g(x)\right) d\varrho(x) = W(f,g)(\beta) - W(f,g)(\alpha).
\end{align}
\end{proposition}

\begin{proof}
By definition $g$ is a distribution function of the measure $g^\qd \varsigma$. Furthermore, the function
\begin{align*}
 f_1(x) = -f^\qd(x) + \int_\alpha^x fd\chi, \quad x\in(a,b)
\end{align*}
is a distribution function of $\tau f \varrho$. Hence one gets from integration by parts
\begin{align*}
 \int_\alpha^\beta g(t) \tau f(t)  d\varrho(t) & = \left[ f_1(t) g(t) \right]_{t=\alpha}^\beta - \int_\alpha^\beta f_1(t+)g^\qd(t)d\varsigma(t).
\end{align*}
We can drop the right-hand limit in the integral since the discontinuities of $f_1$ are a null set with 
 respect to $\varsigma$ by~\eqref{eqntauEEcondPMincommon}. Hence the integral becomes
\begin{align*}
 \int_\alpha^\beta f_1(t) g^\qd(t) d\varsigma(t) & = \int_\alpha^\beta \int_\alpha^t fd\chi\, g^\qd(t)d\varsigma(t) - \int_\alpha^\beta f^\qd(t) g^\qd(t) d\varsigma(t) \\
        & = g(\beta) \int_\alpha^\beta fd\chi  - \int_\alpha^\beta g f d\chi - \int_\alpha^\beta f^\qd(t) g^\qd(t) d\varsigma(t),
\end{align*}
where we performed another integration by parts (and used again~\eqref{eqntauEEcondPMincommon}).
Now verifying the identity is an easy calculation.
\end{proof}

As a consequence of the Lagrange identity, the Wronskian $W(u_1,u_2)$ of two 
solutions $u_1$, $u_2\in\Deftau$ of $(\tau-z)u=0$ is constant.
Furthermore, we have
\begin{align*}
 W(u_1,u_2)  \not=0 \quad\Leftrightarrow\quad u_1, \: u_2\text{ linearly independent}.
\end{align*}
Indeed, the Wronskian of two linearly dependent solutions vanishes obviously. 
Conversely, $W(u_1,u_2) = 0$ means that the vectors 
\begin{align*}
\begin{pmatrix} u_1(x) \\ u_1^\qd(x) \end{pmatrix} \quad\text{and}\quad \begin{pmatrix} u_2(x) \\ u_2^\qd(x) \end{pmatrix}
\end{align*}
are linearly dependent for each $x\in(a,b)$.
But because of uniqueness of solutions this implies the linear dependence of $u_1$ and $u_2$.

For every $z\in\C$ we call two linearly independent solutions of $(\tau-z)u=0$ a fundamental system of $(\tau-z)u=0$.
From Theorem~\ref{thm:exisuniq} and the properties of the Wronskian, one sees that fundamental systems always exist.

\begin{proposition}\label{prop:repsol}
 Let $z\in\C$ and $u_1$, $u_2$ be a fundamental system of the equation $(\tau-z)u =0$. 
 Furthermore, let $c\in(a,b)$, $d_1, d_2\in\C$, $g\in \Llocr$.
 Then there exist $c_1$, $c_2\in\C$ such that the solution $f$ of 
\begin{align*}
 (\tau-z)f = g \quad\text{with}\quad f(c)=d_1 \quad\text{and}\quad f^\qd(c)=d_2
\end{align*}
is given by
\begin{align*}
 f(x) & = c_1 u_1(x) + c_2 u_2(x) + \int_c^x \frac{u_1(x) u_2(t) - u_1(t) u_2(x)}{W(u_1,u_2)} g(t) d\varrho(t), \\
 f^\qd(x) & = c_1 u_1^\qd(x) + c_2 u_2^\qd(x) + \int_c^x \frac{u_1^\qd(x) u_2(t) - u_1(t) u_2^\qd(x)}{W(u_1,u_2)} g(t) d\varrho(t),
\end{align*}
 for each $x\in(a,b)$. 
If $u_1$, $u_2$ is the fundamental system with
\begin{align*}
 u_1(c)=u_2^\qd(c) = 1 \quad\text{and}\quad u_1^\qd(c)=u_2(c)=0, 
\end{align*}
then $c_1=d_1$ and $c_2=d_2$.
\end{proposition}

\begin{proof}
 We set
\begin{align*}
 h(x) = u_1(x)\int_c^x u_2 g\, d\varrho - u_2(x) \int_c^x u_1 g\, d\varrho, \quad x\in(a,b).
\end{align*}
Integration by parts shows that 
\begin{align*}
 \int_\alpha^\beta u_1^\qd(x) & \int_c^x u_2 g\, d\varrho - u_2^\qd(x)\int_c^x u_1 g\, d\varrho\, d\varsigma(x) = \\
       & \qquad\qquad= \left[ u_1(x)\int_c^x u_2 g\, d\varrho - u_2(x) \int_c^x u_1 g\, d\varrho \right]_{x=\alpha}^\beta
\end{align*}
for all $\alpha$, $\beta\in(a,b)$ with $\alpha<\beta$, hence
\begin{align*}
 h^\qd(x) = u_1^\qd(x) \int_c^x u_2 g\, d\varrho - u_2^\qd(x)\int_c^x u_1 g\, d\varrho, \quad x\in(a,b).
\end{align*}
Using again integration by parts we get
\begin{align*}
 \int_\alpha^\beta u_1(x) \int_c^x & u_2 g\, d\varrho\, d\chi(x) - z \int_\alpha^\beta u_1(x) \int_c^x u_2 g\, d\varrho\, d\varrho(x) = \\
      & = \left[ \int_c^x u_2 g\, d\varrho \left( \int_c^x u_1\, d\chi - z\int_c^x u_1\, d\varrho\right) \right]_{x=\alpha}^\beta \\
      & \qquad\qquad  - \int_\alpha^\beta \left( \int_c^x u_1\, d\chi - z \int_c^x u_1\, d\varrho\right) u_2(x)g(x) d\varrho(x) \\
      & = \left[ \int_c^x u_2 g\, d\varrho \left(u_1^\qd(x) - u_1^\qd(c)\right) \right]_{x=\alpha}^\beta \\
      & \qquad\qquad  - \int_\alpha^\beta \left(u_1^\qd(x)-u_1^\qd(c)\right) u_2(x)g(x) d\varrho(x) \\
      & = u_1^\qd(\beta)\int_c^\beta u_2 g\, d\varrho - u_1^\qd(\alpha) \int_c^\alpha u_2 g\, d\varrho - \int_\alpha^\beta u_2 u_1^\qd g\, d\varrho
\end{align*}
for all $\alpha, \beta\in(a,b)$ with $\alpha<\beta$. Now an easy calculation shows that
\begin{align*}
 \int_\alpha^\beta h\, d\chi - \int_\alpha^\beta z h + W(u_1,u_2) g\, d\varrho = h^\qd(\beta) - h^\qd(\alpha).
\end{align*}
Hence $h$ is a solution of $(\tau - z) h=W(u_1,u_2) g$ and therefore the function $f$ given in the claim 
is a solution of $(\tau - z)f = g$. Now if we choose 
\begin{align*}
c_1= \frac{W(f,u_2)(c)}{W(u_1,u_2)(c)} \quad\text{and}\quad c_2=\frac{W(u_1,f)(c)}{W(u_1,u_2)(c)},
\end{align*}
 then $f$ satisfies the initial conditions at $c$.
\end{proof}

Another important identity for the Wronskian is the following Pl\"{u}cker identity.

\begin{proposition}
 For each functions $f_1$, $f_2$, $f_3$, $f_4\in\Deftau$ we have
 \begin{align*}
  0 & = W(f_1,f_2)W(f_3,f_4) + W(f_1,f_3)W(f_4,f_2) + W(f_1,f_4)W(f_2,f_3).
 \end{align*}
\end{proposition}

\begin{proof}
 The right-hand side is equal to the determinant
\begin{align*}
\frac{1}{2} \left| \begin{matrix} f_1 & f_2 & f_3 & f_4 \\
                             f_1^\qd & f_2^\qd & f_3^\qd & f_4^\qd \\
                             f_1 & f_2 & f_3 & f_4 \\
                             f_1^\qd & f_2^\qd & f_3^\qd & f_4^\qd \end{matrix} \right|.
\end{align*}
\end{proof}

We say $\tau$ is regular at $a$, if $|\varrho|((a,c])$, $|\varsigma|((a,c])$ and $|\chi|((a,c])$ are finite for one (and hence for all) $c\in(a,b)$.
Similarly one defines regularity for the right endpoint $b$.
Finally, we say $\tau$ is regular if $\tau$ is regular at both endpoints, i.e., if $|\varrho|$, $|\varsigma|$ and $|\chi|$ are finite.

\begin{theorem}\label{thm:EEreg}
 Let $\tau$ be regular at $a$, $z\in\C$ and $g\in L^1((a,c);\varrho)$ for each $c\in(a,b)$. 
 Then for every solution $f$ of $(\tau-z)f=g$ the limits
\begin{align*}
 f(a) := \lim_{x\rightarrow a} f(x) \quad\text{and}\quad f^\qd(a):=\lim_{x\rightarrow a} f^\qd(x)
\end{align*}
exist and are finite.
For each $d_1$, $d_2\in\C$ there is a unique solution of
\begin{align*}
 (\tau-z)f=g \quad\text{with}\quad f(a)=d_1 \quad\text{and}\quad f^\qd(a)=d_2.
\end{align*}
Furthermore, if $g$, $d_1$, $d_2$ and $z$ are real, then the solution is real.
Similar results hold for the right endpoint $b$.
\end{theorem}

\begin{proof} 
The first part of the theorem is an immediate consequence of Theorem~\ref{thm:appregep}.
 From Proposition~\ref{prop:repsol} we infer that all solutions of $(\tau-z)f=g$ are given by 
 \begin{align*}
  f(x) = c_1 u_1(x) + c_2 u_2(x) + f_0(x), \quad x\in(a,b),
 \end{align*}
 where $c_1$, $c_2\in\C$, 
 $u_1$, $u_2$ are a fundamental system of $(\tau-z)u=0$ and $f_0$ is some solution of $(\tau-z)f=g$.
 Now since 
 \begin{align*}
 W(u_1,u_2)(a)=u_1(a)u_2^\qd(a) - u_1^\qd(a)u_2(a)\not=0,
 \end{align*}
 there is exactly one choice for the coefficients $c_1, c_2\in\C$ such that
 the solution $f$ satisfies the initial values at $a$. If $g$, $d_1$, $d_2$ and $z$ are real then $u_1$, $u_2$, and $f_0$ can be chosen
 real and hence also $c_1$ and $c_2$ are real.
\end{proof}

Under the assumptions of Theorem~\ref{thm:EEreg} one sees that Proposition~\ref{prop:repsol} remains 
 valid even in the case when $c=a$ (respectively $c=b$) with essentially the same proof.

We now turn to analytic dependence of solutions on the spectral parameter $z\in\C$.
 These results will be needed in Section~\ref{secweyltitchm}.

\begin{theorem}\label{thmMSLEanaly}
  Let $g\in \Llocr$, $c\in(a,b)$, $d_1, d_2\in\C$ and for each $z\in\C$ let $f_z$ be the unique solution of
  \begin{align*}
   (\tau - z) f = g \quad\text{with}\quad f(c)=d_1 \quad\text{and}\quad f^\qd(c)=d_2.
  \end{align*}
  Then  $f_z(x)$ and $f_z^\qd(x)$ are entire
  functions of order at most $\nicefrac{1}{2}$ in $z$  for every point $x\in(a,b)$.
  Moreover, for each $\alpha$, $\beta\in(a,b)$ with $\alpha<\beta$ we have
  \begin{align*}
   |f_z(x)| + |f_z^\qd(x) | \leq C \E^{B\sqrt{|z|}}, \quad x\in[\alpha,\beta],~z\in\C
  \end{align*}
  for some constants $C$, $B\in\R$.
\end{theorem}

\begin{proof}
 The analyticity part follows by applying Theorem~\ref{thmappLMDEAnaly} to the equivalent system from the
 proof of Theorem~\ref{thm:exisuniq}. For the remaining part note that because of Proposition~\ref{prop:repsol} it suffices to consider the case when $g$ vanishes identically. If we set for each $z\in\C$ with $|z|\geq 1$
 \begin{align*} 
  v_z(x) = |z| |f_z(x)|^2 + |f_z^\qd(x)|^2, \quad x\in(a,b),
 \end{align*}
 an integration by parts shows that for each $x\in(a,b)$
 \begin{align*}
  v_z(x) = v_z(c) & + 
     |z| \int_c^x \left( f_z f_z^{\qd\ast} +  f_z^\qd f_z^\ast \right) d\varsigma \\
         & + \int_c^x \left( f_z f_z^{\qd\ast} +  f_z^\qd f_z^\ast \right) d\chi - 
         \int_c^x \left(z f_z f_z^{\qd\ast} + z^\ast f_z^\qd f_z^\ast\right) d\varrho.
 \end{align*}
 Because of the elementary estimate
 \begin{align*}
  2| f_z(x) f_z^\qd(x)| \leq \frac{ |z| |f_z(x)|^2 + |f_z^\qd(x)|^2}{\sqrt{|z|}} = \frac{v_z(x)}{\sqrt{|z|}}, \quad x\in(a,b),
 \end{align*}
 we get an upper bound for $v_z$
 \begin{align*}
  v_z(x) & \leq v_z(c) + \int_c^x v_z(t) \sqrt{|z|} d\omega(t), \quad x\in[c,b), 
 \end{align*}
 where $\omega=|\varsigma|+ |\chi| + |\varrho|$, as in the proof of Theorem~\ref{thm:exisuniq}.
 Now an application of Lemma~\ref{lemappGronwall} yields
 \begin{align*}
  v_z(x) \leq v_z(c) \E^{\int_c^x \sqrt{|z|} d\omega}, \quad x\in[c,b).
 \end{align*}
 To the left-hand side of $c$ we have
 \begin{align*}
  v_z(x+) & \leq v_z(c) + \int_{x+}^c v_z(t) \sqrt{|z|} d\omega(t), \quad x\in(a,c)
 \end{align*}
 and hence again by the Gronwall Lemma~\ref{lemappGronwall}
 \begin{align*}
  v_z(x+) \leq v_z(c) \E^{\int_{x+}^c \sqrt{|z|} d\omega}, \quad x\in(a,c),
 \end{align*}
 which is the required bound.
\end{proof}

Under the assumptions of Theorem~\ref{thmMSLEanaly} also the right-hand limits of $f_z$ and their quasi-derivatives are entire functions in $z$ of order at most $\nicefrac{1}{2}$ with corresponding bounds.
Moreover, the same analytic properties are true for the solutions $f_z$ of the initial value problem
\begin{align*}
 (\tau-z)f=g \quad\text{with}\quad f_z(c+)=d_1 \quad\text{and}\quad f_z^\qd(c+)=d_2.
\end{align*}
Indeed, this fact follows for example from the remark after the proof of Theorem~\ref{thmappLMDEAnaly} in Appendix~\ref{appMODE}.

Furthermore, if, in addition to the assumptions of Theorem~\ref{thmMSLEanaly},
 $\tau$ is regular at $a$ and $g$ is integrable near $a$, then the functions 
 \begin{align*}
 z\mapsto f_z(a) \quad\text{and}\quad z\mapsto f_z^\qd(a)
 \end{align*}
 are entire of order at most $\nicefrac{1}{2}$ as well and the bound in Theorem~\ref{thmMSLEanaly} holds for all $x\in[a,\beta]$. 
 Indeed, this follows since the entire functions 
 \begin{align*}
 z\mapsto f_z(x) \quad\text{and}\quad z\mapsto f_z^\qd(x)
 \end{align*}
  are locally bounded, uniformly in $x\in(a,c)$. Moreover, in this case the assertions of Theorem~\ref{thmMSLEanaly} are valid even if we take $c=a$ and/or $\alpha=a$, as the construction of the solution in the proof of Theorem~\ref{thm:EEreg} shows. 
  The required bound is proven as in the general case (hereby note that $\omega$ is finite near $a$ since $\tau$ is regular there).

We gather the assumptions made on our measure coefficients so far and add some new, all of which will be in force in the rest of this paper (except for Lemma~\ref{lemcondforhyp}).
Therefore, we say that some interval $(\alpha,\beta)$ is a gap of $\supp(\varrho)$ if it is contained in the complement of $\supp(\varrho)$ but the endpoints $\alpha$ and $\beta$ lie in $\supp(\varrho)$.

\begin{hypothesis}   The following assumptions on our measure coefficients will be in force throughout the rest of this paper:
 \begin{enumerate}\label{genhyp}
 \item\label{genhypfirst} The measure $\varrho$ is positive. 
 \item The measure $\chi$ is real-valued.
 \item The measure $\varsigma$ is real-valued and supported on the whole interval; 
      \begin{align*}\supp(\varsigma)=(a,b).\end{align*}
 \item\label{genhypatoms} The measure $\varsigma$ has no point masses in common with $\varrho$ or $\chi$, i.e.,
         \begin{align*}
              \varsigma(\lbrace x\rbrace) \chi(\lbrace x\rbrace) = \varsigma(\lbrace x\rbrace) \varrho(\lbrace x\rbrace)=0.
         \end{align*}
   \item \label{genhypgap} For each gap $(\alpha,\beta)$ of $\supp(\varrho)$ and every function $f\in\Deftau$ 
           with the outer limits $f(\alpha-)=f(\beta+)=0$ we have $f(x)=0$, $x\in(\alpha,\beta)$.
   \item \label{genhyponepoint} The measure $\varrho$ is supported on more than one point.
\end{enumerate}
\end{hypothesis}

As a consequence of the real-valuedness of the measures, $\tau$ is a real differential expression, 
i.e., $f\in\Deftau$ if and only if $f^\ast\in\Deftau$ and $\tau f^\ast = (\tau f)^\ast$ in this case.
Moreover, $\varrho$ has to be positive in order to obtain a definite inner product later.
Furthermore, condition~\eqref{genhypgap} in Hypothesis~\ref{genhyp} is crucial for Proposition~\ref{prop:identDeftauTloc}
and Proposition~\ref{propTlocmulval} to hold. For example, if $(a,b)=\R$, $\varrho$ is supported on $\pi\Z$ and we choose
$\varsigma$ and $-\chi$ to be equal to the Lebesgue measure, then the function $f(x)=\sin(x)$ belongs to $\Deftau$ with $\tau f=0$, in contradistinction to Proposition~\ref{prop:identDeftauTloc}. 
Upon cutting this function off outside of the interval $(0,\pi)$, one also sees that Proposition~\ref{propTlocmulval} ceases to hold in this case. 
However, condition~\eqref{genhypgap} is satisfied by a large class of measures as the next lemma shows.

\begin{lemma}\label{lemcondforhyp}
 If for each gap $(\alpha,\beta)$ of $\supp(\varrho)$ 
  the measures $\varsigma|_{(\alpha,\beta)}$ and $\chi|_{(\alpha,\beta)}$ are of one and the same sign, then~\eqref{genhypgap} in Hypothesis~\eqref{genhyp} holds.
\end{lemma}

\begin{proof}
 Let $(\alpha,\beta)$ be a gap of $\supp(\varrho)$ and $f\in\Deftau$ with $f(\alpha-)=f(\beta+)=0$.
 As in the proof of Proposition~\ref{propLagrange}, integration by parts yields 
 \begin{align*}
  f(\beta)^\ast \tau f(\beta) \varrho(\lbrace\beta\rbrace) & = \int_\alpha^{\beta+} \tau f(x) f(x)^\ast d\varrho(x) \\
    & = \int_\alpha^{\beta+} |f^\qd(x)|^2 d\varsigma(x)  
   + \int_\alpha^{\beta+}  \left|f(x)\right|^2  d\chi(x). 
 \end{align*}
 Now the left-hand side vanishes since either $\varrho(\lbrace\beta\rbrace)=0$ or $f$ is continuous in $\beta$ and hence
 $f(\beta)=f(\beta+)=0$.
 Thus $f^\qd$ vanishes almost everywhere with respect to $\varsigma$, i.e., $f^\qd$ vanishes in $(\alpha,\beta)$
  and $f$ is constant in $(\alpha,\beta)$. Now since $f(\beta+) = f(\beta) + f^\qd(\beta)\varsigma(\lbrace \beta\rbrace)$,
 we infer that $f$ vanishes in $(\alpha,\beta)$
\end{proof}

The theory we are going to develop from now on is not applicable if the support of $\varrho$ consists of not more than one point, since
 in this case $\Llocr$ is only one-dimensional (and hence all solutions of $(\tau -z)u=0$ are linearly dependent). 
 In particular, the essential Proposition~\ref{prop:identDeftauTloc} does not hold in this case, which is the main reason for assumption \eqref{genhyponepoint} in Hypothesis~\ref{genhyp}.
 Nevertheless, this case is important, in particular for applications to the isospectral problem of the 
 Camassa--Holm equation. Hence we will treat the case when $\supp(\varrho)$ consists of only one point separately in Appendix~\ref{app:onedim}.

Our aim is to introduce linear operators in the Hilbert space $\Lr$, induced by the differential expression $\tau$.
As a first step we define a linear relation $\Tloc$ of $\Llocr$ into $\Llocr$ by
\begin{align*}
 \Tloc = \left\lbrace (f,\tau f) \,|\, f\in\Deftau \right\rbrace \subseteq \Llocr\times\Llocr.
\end{align*}
For a brief introduction to the theory of linear relations we refer to Appendix~\ref{appLR} and the references cited there. 
Now, in contrast to the classical case, in general $\Deftau$ is not embedded in $\Llocr$, i.e., $\Tloc$ is multi-valued.
Instead we have the following result, which is important for our approach. For later use, we introduce the abbreviations
\begin{align*}
 \alpha_\varrho = \inf\supp(\varrho) \quad\text{and}\quad  \beta_\varrho = \sup\supp(\varrho)
\end{align*}
for the endpoints of the convex hull of the support of $\varrho$.

\begin{proposition}\label{prop:identDeftauTloc}
The linear map 
\begin{align*} 
\begin{matrix}
  \Deftau & \rightarrow & \Tloc \\
   f      & \mapsto     & (f,\tau f)
\end{matrix}
\end{align*}
 is bijective.
\end{proposition}

\begin{proof}
 Clearly this mapping is linear and onto $\Tloc$ by definition.
 Now let $f\in\Deftau$ such that $f=0$ almost everywhere with respect to $\varrho$.
 We will show that $f$ is of the form
 \begin{align}\label{eqn:mulvalfform}
  f(x) = \begin{cases}
          c_a u_a(x), & \text{if }x\in(a,\alpha_\varrho], \\
          0,          & \text{if }x\in(\alpha_\varrho,\beta_\varrho], \\
          c_b u_b(x), & \text{if }x\in(\beta_\varrho,b),
         \end{cases}
 \end{align}
 where $c_a$, $c_b\in\C$ and $u_a$, $u_b$ are the solutions of $\tau u=0$ with 
 \begin{align*}
  u_a(\alpha_\varrho-)=u_b(\beta_\varrho+)=0 \quad\text{and}\quad u_a^\qd(\alpha_\varrho-)= u_b^\qd(\beta_\varrho+)=1.
 \end{align*}
 Obviously we have $f(x)=0$ for all $x$ in the interior of $\supp(\varrho)$ and points of mass of $\varrho$.
 Now if $(\alpha,\beta)$ is a gap of $\supp(\varrho)$, then since $\alpha$, $\beta\in\supp(\varrho)$ 
  we have $f(\alpha-)=f(\beta+)=0$ and hence $f(x)=0$, $x\in[\alpha,\beta]$ by Hypothesis~\ref{genhyp}.
  Hence all points $x\in(\alpha_\varrho,\beta_\varrho)$ for which possibly $f(x)\not=0$, lie on the boundary 
  of $\supp(\varrho)$ such that there are monotone sequences $x_{+,n}$, $x_{-,n}\in\supp(\varrho)$ with 
 $x_{+,n}\downarrow x$ and $x_{-,n}\uparrow x$. 
 Then for each $n\in\N$, we either have $f(x_{-,n}+)=0$ or $f(x_{-,n}-)=0$, hence 
 \begin{align*}
 f(x-)=\lim_{n\rightarrow\infty} f(x_{-,n}-) = \lim_{n\rightarrow\infty} f(x_{-,n}+)=0.
 \end{align*}
 Similarly one shows that also $f(x+)=0$.
 Now since $f$ is a solution of $\tau u=0$ outside of $[\alpha,\beta]$, it remains to show that $f(\alpha_\varrho)=f(\beta_\varrho)=0$.
  Therefore, assume that $f$ is not continuous in $\alpha_\varrho$, i.e., $\varsigma(\lbrace\alpha_\varrho\rbrace)\not=0$.
  Then $f^\qd$ is continuous in $\alpha_\varrho$ and hence $f^\qd(\alpha_\varrho)=0$. But this yields 
  \begin{align*}
  f(\alpha_\varrho-)=f(\alpha_\varrho+) - f^\qd(\alpha_\varrho)\varsigma(\lbrace \alpha_\varrho\rbrace)=0.
  \end{align*}
 Similarly, one shows that $f(\beta_\varrho)=0$ and hence $f$ is of the claimed form. Furthermore, a simple calculation yields
 \begin{align}\label{eqn:mulvaltauf}
   \tau f = c_a \indik_{\lbrace\alpha_\varrho\rbrace} - c_b \indik_{\lbrace\beta_\varrho\rbrace}.
 \end{align}

 Now in order to prove that our mapping is one-to-one let $f\in\Deftau$ be such that $f=0$ 
 and $\tau f = 0$ almost everywhere with respect to $\varrho$.
 By Theorem~\ref{thm:exisuniq} it suffices to prove that $f(c)=f^\qd(c) = 0$ at some point $c\in(a,b)$.
 But this is valid for all points between $\alpha_\varrho$ and $\beta_\varrho$ by the first part of the proof.
\end{proof}

In the following we will always identify the elements of the linear relation $\Tloc$ with functions in $\Deftau$.
Hence some element $f\in\Tloc$ is always identified with some function $f\in\Deftau$, which is an $\AClocp$ representative of 
the first component of $f$ (as an element of $\Tloc$) and $\tau f\in\Llocr$ is the
second component of $f$ (again as an element of $\Tloc$). 
 In general the relation $\Tloc$ is multi-valued, i.e.,
\begin{align*}
 \mul{\Tloc} = \left\lbrace g\in\Llocr \,|\, (0,g)\in\Tloc \right\rbrace \not= \lbrace 0\rbrace.
\end{align*}
In the formulation of the next result, we consider the condition that $\varrho$ has no mass in $\alpha_\varrho$ as trivially
satisfied if $\alpha_\varrho=a$ (since $\varrho$ lives on $(a,b)$ by definition) and similarly at the other endpoint.

\begin{proposition}\label{propTlocmulval}
 The multi-valued part of $\Tloc$ is given by
\begin{align*}
 \mul{\Tloc} = \linspan\left\lbrace \indik_{\lbrace\alpha_\varrho\rbrace}, \indik_{\lbrace\beta_\varrho\rbrace}\right\rbrace.
\end{align*}
 In particular,
\begin{align*}
 \dim\mul{\Tloc} = \begin{cases}
                  0, & \text{if }\varrho\text{ has neither mass in }\alpha_\varrho \text{ nor in }\beta_\varrho, \\
                  1, & \text{if }\varrho\text{ has either mass in }\alpha_\varrho \text{ or in }\beta_\varrho, \\
                  2, & \text{if }\varrho\text{ has mass in }\alpha_\varrho\text{ and in }\beta_\varrho.
                \end{cases}
\end{align*}
 Hence $\Tloc$ is an operator if and only if $\varrho$ has neither mass in $\alpha_\varrho$ nor in $\beta_\varrho$.
\end{proposition}

\begin{proof}
 Let $(f,\tau f)\in\Tloc$ with $f=0$ almost everywhere with respect to $\varrho$.
 In the proof of Proposition~\ref{prop:identDeftauTloc} we saw that such an $f$ is of the 
  form~\eqref{eqn:mulvalfform} and $\tau f$ is a linear combination of $\indik_{\lbrace\alpha_\varrho\rbrace}$ 
 and $\indik_{\lbrace\beta_\varrho\rbrace}$ by~\eqref{eqn:mulvaltauf}.
 It remains to prove that $\mul{\Tloc}$ indeed contains $\indik_{\lbrace\alpha_\varrho\rbrace}$ if
   $\varrho$ has mass in $\alpha_\varrho$.
 Therefore, consider the function 
 \begin{align*}
  f(x) = \begin{cases}
          u_a(x), & \text{if }x\in(a,\alpha_\varrho], \\
          0,      & \text{if }x\in(\alpha_\varrho,b).
         \end{cases}
 \end{align*}
 One easily checks that $f$ lies in $\Deftau$ and hence $(0,\indik_{\lbrace\alpha_\varrho\rbrace})=(f,\tau f)\in\Tloc$.
 Similarly one shows that $\indik_{\lbrace\beta_\varrho\rbrace}$ indeed lies in $\mul{\Tloc}$ if $\varrho$ has mass in $\beta_\varrho$. 
 Furthermore, note that
  $\indik_{\lbrace\alpha_\varrho\rbrace}=0$ (respectively $\indik_{\lbrace\beta_\varrho\rbrace}=0$) as functions in $\Llocr$ provided that $\varrho$
 has no mass in $\alpha_\varrho$ (respectively in $\beta_\varrho$).
\end{proof}

In contrast to the classical case one can not define a proper Wronskian for elements in $\dom{\Tloc}$, 
instead we define the Wronskian of two elements $f$, $g$ of the linear relation $\Tloc$ as
\begin{align*}
 W(f,g)(x) = f(x)g^\qd(x) - f^\qd(x)g(x), \quad x\in(a,b).
\end{align*}
The Lagrange identity then takes the form
\begin{align*}
 W(f,g)(\beta) - W(f,g)(\alpha) =
            \int_\alpha^\beta \left(g(x) \tau f(x)  - f(x) \tau g(x)\right) d\varrho(x).
\end{align*}
Furthermore, note that by Theorem~\ref{thm:exisuniq} we have
\begin{align}\label{eqnSLPdimkerTloc}
 \ran(\Tloc-z) = \Llocr \quad\text{and}\quad \dim\ker(\Tloc-z) = 2
\end{align}
for each $z\in\C$.

\section{Sturm--Liouville relations}

In this section we will restrict the differential relation $\Tloc$ in order to obtain a linear relation
in the Hilbert space $\Lr$ with scalar product
\begin{align*}
 \spr{f}{g} = \int_{a}^b  f(x) g(x)^\ast  d\varrho(x).
\end{align*}
First we define the maximal relation $\Tmax$ in $\Lr$ by
\begin{align}
 \Tmax 
       & = \lbrace (f, \tau f)\in\Tloc \,|\, f\in\Lr, \tau f\in\Lr \rbrace.
\end{align}
In general $\Tmax$ is not an operator. Indeed we have
\begin{align*}
 \mul{\Tmax} = \mul{\Tloc},
\end{align*}
since all elements of $\mul{\Tloc}$ are square integrable with respect to $\varrho$.
In order to obtain a symmetric relation we restrict the maximal relation $\Tmax$ to functions with compact support
\begin{align}
 \Tpre = \lbrace (f,\tau f) \,|\, f\in\Deftau,~ \supp(f) \text{ compact in }(a,b) \rbrace.
\end{align}
Indeed, this relation $\Tpre$ is an operator as we will see later.

Since $\tau$ is a real differential expression, the relations $\Tpre$ and $\Tmax$ are real with respect 
 to the natural conjugation in $\Lr$, i.e., if $f\in\Tmax$ (respectively $f\in\Tpre$), then also $f^\ast\in\Tmax$ (respectively $f^\ast\in\Tpre$), where 
 the conjugation is defined componentwise.

We say some measurable function $f$ lies in $\Lr$ near $a$ (respectively near $b$) if $f$ lies in $L^2((a,c);\varrho)$ (respectively in $L^2((c,b);\varrho)$)
 for all $c\in(a,b)$. Furthermore, we say some $f\in\Tloc$ lies in $\Tmax$ near $a$ (respectively near $b$) if $f$ and $\tau f$ both 
 lie in $\Lr$ near $a$ (respectively near $b$).
One easily sees that some $f\in\Tloc$ lies in $\Tmax$ near $a$ (respectively $b$) if and only if
$f^\ast$ lies in $\Tmax$ near $a$ (respectively $b$).

\begin{proposition}
 Let $\tau$ be regular at $a$ and $f$ lie in $\Tmax$ near $a$. Then both limits
\begin{align*}
 f(a) := \lim_{x\rightarrow a} f(x) \quad\text{and}\quad f^\qd(a) := \lim_{x\rightarrow a} f^\qd(x)
\end{align*}
exist and are finite. A similar result holds at $b$.
\end{proposition}

\begin{proof}
 Under this assumptions $\tau f$ lies in $\Lr$ near $a$ and since $\varrho$ is
 a finite measure near $a$ we have $\tau f\in L^1((a,c);\varrho)$ for each $c\in(a,b)$. 
 Hence the claim follows from Theorem~\ref{thm:EEreg}.
\end{proof}

From the Lagrange identity we now get the following lemma.

\begin{lemma}\label{lem:l2lagrange}
If $f$ and $g$ lie in $\Tmax$ near $a$, then the limit
\begin{align*}
 W(f,g^\ast)(a) := \lim_{\alpha\rightarrow a} W(f,g^\ast)(\alpha)
\end{align*}
exists and is finite. A similar result holds at  $b$.
If $f$, $g\in\Tmax$, then
\begin{align}
 \spr{\tau f}{g} - \spr{f}{\tau g} = W(f,g^\ast)(b) - W(f,g^\ast)(a) =: W_a^b(f,g^\ast).
\end{align}
\end{lemma}

\begin{proof}
 If $f$ and $g$ lie in $\Tmax$ near $a$, then the limit $\alpha\rightarrow a$ of the left-hand side in equation~\eqref{eqn:lagrange} exists.
 Hence also the limit in the claim exists.
 Now the remaining part follows by taking the limits $\alpha\rightarrow a$ and $\beta\rightarrow b$.
\end{proof}

If $\tau$ is regular at $a$ and $f$ and $g$ lie in $\Tmax$ near $a$, then we clearly have
\begin{align*}
 W(f,g^\ast)(a) = f(a) g^\qd(a)^\ast - f^\qd(a) g(a)^\ast.
\end{align*}
In order to determine the adjoint of $\Tpre$
\begin{align*}
 \Tpre^\ast = \lbrace (f,g)\in\Lr\times\Lr \,|\, \forall (u,v)\in\Tpre: \spr{f}{v} = \spr{g}{u} \rbrace,
\end{align*}
as in the classical theory, we need the following lemma (see~\cite[Lemma~9.3]{tschroe}).

\begin{lemma}\label{lem:hilfslemmaadjoint}
 Let $V$ be a vector space over $\C$ and $F_1,\ldots,F_n,F\in V^\ast$, then
\begin{align*}
 F\in\linspan\left\lbrace F_1,\ldots,F_n\right\rbrace \quad\Leftrightarrow\quad \bigcap_{i=1}^n \ker F_i \subseteq \ker F.
\end{align*}
\end{lemma}

\begin{theorem}
 The adjoint of $\Tpre$ is $\Tmax$.
\end{theorem}

\begin{proof}
 From Lemma~\ref{lem:l2lagrange} one immediately gets $\Tmax\subseteq\Tpre^\ast$. Indeed, for each 
  $f\in\Tpre$ and $g\in\Tmax$ we have
  \begin{align*}
    \spr{\tau f}{g} - \spr{f}{\tau g} & = \lim_{\beta\rightarrow b} W(f,g^\ast)(\beta) - \lim_{\alpha\rightarrow a}W(f,g^\ast)(\alpha) = 0,  
  \end{align*}
  since $W(f,g^\ast)$ has compact support.
 Conversely, let $(f,f_2)\in\Tpre^\ast$ and $\tilde{f}$ be a solution of $\tau \tilde{f} = f_2$.
 We expect that $(f-\tilde{f},0)\in\Tloc$. 
 To prove this we will invoke Lemma~\ref{lem:hilfslemmaadjoint}.
 Therefore, we consider linear functionals 
\begin{align*}
 l(g)   & = \int_{a}^b \left(f(x)-\tilde{f}(x)\right)^\ast g(x) d\varrho(x),   & g\in L^2_c((a,b);\varrho), \\
 l_j(g) & = \int_{a}^b u_j(x)^\ast g(x) d\varrho(x),                           & g\in L^2_c((a,b);\varrho) & ,~j=1,2,
\end{align*}
where $u_j$ are two solutions of $\tau u=0$ with $W(u_1,u_2)=1$ and $L^2_c((a,b);\varrho)$ is the space of square integrable functions with compact support.
For these functionals we have $\ker l_1 \cap \ker l_2 \subseteq \ker l$.
Indeed let $g\in\ker l_1 \cap \ker l_2$, then the function
\begin{align*}
 u(x) = u_1(x)\int_a^x u_2(t) g(t) d\varrho(t) + u_2(x) \int_x^b u_1(t) g(t) d\varrho(t), \quad x\in(a,b)
\end{align*}
is a solution of $\tau u = g$ by Proposition~\ref{prop:repsol} and has compact support since $g$ lies in 
the kernel of $l_1$ and $l_2$, hence $u\in\Tpre$. 
Then the Lagrange identity and the definition of the adjoint yields
\begin{align*}
 \int_a^b \left(f(x) -\tilde{f}(x)\right)^\ast \tau u(x) d\varrho(x) & = \spr{\tau u}{f} - \int_a^b \tilde{f}(x)^\ast \tau u(x) d\varrho(x) \\
   & = \spr{u}{f_2} - \int_a^b \tau \tilde{f}(x)^\ast u(x) d\varrho(x) = 0
\end{align*}
and hence $g=\tau u\in\ker l$. 
Now applying Lemma~\ref{lem:hilfslemmaadjoint} there are $c_1$, $c_2\in\C$ such that
\begin{align}\tag{$*$}\label{eqn:ftildfsol}
 \int_a^b \left( f(x)-\tilde{f}(x) + c_1 u_1(x) + c_2 u_2(x)\right)^\ast g(x) d\varrho(x) = 0 
\end{align}
for each function $g\in L^2_c((a,b);\varrho)$. 
By definition of $\Tloc$ we obviously have $(\tilde{f}+c_1 u_1 +c_2 u_2,f_2)\in\Tloc$.
But the first component of this pair is equal to $f$, almost everywhere with respect 
 to $\varrho$ because of~\eqref{eqn:ftildfsol}.
 Hence we also have $(f,f_2)\in\Tloc$ and therefore $(f,f_2)\in\Tmax$.
\end{proof}

By the preceding theorem $\Tpre$ is symmetric.
The closure $\Tmin$ of $\Tpre$ is called the minimal relation,
\begin{align*}
 \Tmin = \overline{\Tpre} = \Tpre^{\ast\ast} = \Tmax^\ast.
\end{align*}
In order to determine $\Tmin$ we need the following lemma on functions in the maximal relation $\Tmax$.

\begin{lemma}\label{lem:funcdomtmax}
 If $f_a$ lies in $\Tmax$ near $a$ and $f_b$ lies in $\Tmax$ near $b$, 
 then there exists an $f\in\Tmax$ such that $f=f_a$ near $a$ and $f=f_b$ near $b$ (regarded as functions in $\Deftau$).
\end{lemma}

\begin{proof}
Let $u_1$, $u_2$ be a fundamental system of $\tau u = 0$ with $W(u_1,u_2)=1$ and let $\alpha$, $\beta\in(a,b)$ with $\alpha<\beta$ such that the functionals
\begin{align*}
 F_j(g) = \int_\alpha^\beta u_j g d\varrho, \quad g\in\Lr,~j=1,2
\end{align*}
are linearly independent. This is possible since otherwise $u_1$ and $u_2$ would be linearly dependent in $\Lr$ and hence also in $\Deftau$ by 
 the identification in Lemma~\ref{prop:identDeftauTloc}.
First we show that there is some $u\in\Deftau$ such that
\begin{align*}
 u(\alpha) = f_a(\alpha), \quad u^\qd(\alpha) = f_a^\qd(\alpha), \quad u(\beta) = f_b(\beta) \quad\text{and}\quad u^\qd(\beta) = f_b^\qd(\beta).
\end{align*}
Indeed, let $g\in\Lr$ and consider the solution $u$ of $\tau u=g$ with the initial conditions
\begin{align*}
 u(\alpha) = f_a(\alpha) \quad\text{and}\quad u^\qd(\alpha)=f_a^\qd(\alpha).
\end{align*}
With Proposition~\ref{prop:repsol} one sees that $u$ has the desired properties if 
\begin{align*}
      \begin{pmatrix} F_2(g) \\ F_1(g) \end{pmatrix} 
     =  \begin{pmatrix} u_1(\beta) & -u_2(\beta) \\ u_1^\qd(\beta) & -u_2^\qd(\beta) \end{pmatrix}^{-1}
       \begin{pmatrix}
        f_b(\beta) - c_1 u_1(\beta) - c_2 u_2(\beta) \\ f_b^\qd(\beta) - c_1u_1^\qd(\beta) - c_2 u_2^\qd(\beta)
       \end{pmatrix},
\end{align*}
where $c_1$, $c_2\in\C$ are the constants appearing in Proposition~\ref{prop:repsol}.
But since the functionals $F_1$, $F_2$ are linearly independent, it is possible to choose a function $g\in\Lr$ such that this equation is valid. 
Now the function $f$ defined by
\begin{align*}
 f(x) = \begin{cases}
          f_a(x),   & \text{if }x\in(a,\alpha], \\
          u(x),     & \text{if }x\in(\alpha,\beta], \\
          f_b(x),   & \text{if }x\in(\beta,b),
        \end{cases}
\end{align*}
has the claimed properties.
\end{proof}

\begin{theorem}\label{thm:Tmin}
 The minimal relation $\Tmin$ is given by
\begin{align}
 \Tmin = \lbrace f\in \Tmax \,|\, \forall g\in\Tmax: W(f,g)(a) = W(f,g)(b) = 0 \rbrace.
\end{align}
 Furthermore, $\Tmin$ is an operator, i.e., $\dim\mul{\Tmin} = 0$.
\end{theorem}

\begin{proof}
 If $f\in\Tmin=\Tmax^\ast\subseteq\Tmax$ we have
\begin{align*}
 0 = \spr{\tau f}{g} - \spr{f}{\tau g} = W(f,g^\ast)(b) - W(f,g^\ast)(a)
\end{align*}
 for each $g\in\Tmax$. Given some $g\in\Tmax$, there is a $g_a\in\Tmax$ such that $g_a^\ast=g$ in a 
 vicinity of $a$ and $g_a=0$ in a vicinity of $b$. Therefore,  
 \begin{align*}
 W(f,g)(a) = W(f,g_a^\ast)(a) - W(f,g_a^\ast)(a) = 0.
 \end{align*}
 Similarly one sees that $W(f,g)(b)=0$ for each $g\in\Tmax$.
 Conversely, if $f\in\Tmax$ such that for each  $g\in\Tmax$, $W(f,g)(a)=W(f,g)(b)=0$, then
\begin{align*}
 \spr{\tau f}{g} - \spr{f}{\tau g} = W(f,g^\ast)(b) - W(f,g^\ast)(a) = 0,
\end{align*}
 hence $f\in \Tmax^\ast = \Tmin$.

 In order to show that $\Tmin$ is an operator, let $f\in\Tmin$ with $f = 0$ almost everywhere with respect to $\varrho$.
 If $\alpha_\varrho>a$ and $\varrho(\lbrace\alpha_\varrho\rbrace)\not=0$, then $f$ is of the form~\eqref{eqn:mulvalfform}. 
 From what we already proved we know that 
 \begin{align*}
 W(f,u_1)(a)=W(f,u_2)(a)=0
 \end{align*}
  for each fundamental system $u_1$, $u_2$ of $\tau u=0$.
 But $W(f,u_j)(x)$ is constant on $(a,\alpha_\varrho)$ and hence we infer $f(\alpha_\varrho)=f^\qd(\alpha_\varrho) = 0$.
 From this we see that $f$ vanishes on $(a,\alpha_\varrho)$. Similarly one proves that $f$ also vanishes on $(\beta_\varrho,b)$, hence $f=0$.
\end{proof}

For regular differential expressions we may characterize the minimal operator in terms of the boundary values of functions $f\in\Tmax$.

\begin{corollary}
 If $\tau$ is regular at $a$ and $f\in\Tmax$, then we have
\begin{align*}
 f(a) = f^\qd(a) = 0 \quad\Leftrightarrow\quad \forall g\in\Tmax: W(f,g)(a) = 0.
\end{align*}
 A similar result holds at $b$.
\end{corollary}

\begin{proof}
 The claim follows from $W(f,g)(a) = f(a)g^\qd(a) - f^\qd(a)g(a)$ and the fact that one finds $g\in\Tmax$ with prescribed initial values at $a$.
 Indeed, one can take $g$ to coincide with some solution of $\tau u = 0$ near $a$.
\end{proof}

If the measure $\varrho$ has no weight near some endpoint, we get another characterization for functions in $\Tmin$ in terms of their 
 left-hand (respectively right-hand) limit at $\alpha_\varrho$ (respectively at $\beta_\varrho$).

\begin{corollary}
 If $\alpha_\varrho>a$ and $f\in\Tmax$, then we have
\begin{align*}
 f(\alpha_\varrho-) = f^\qd(\alpha_\varrho-) = 0 \quad\Leftrightarrow\quad \forall g\in\Tmax: W(f,g)(a) = 0.
\end{align*}
A similar result holds at $b$.
\end{corollary}

\begin{proof}
 The Wronskian of two functions $f$, $g$ which lie in $\Tmax$ near $a$ is constant on $(a,\alpha_\varrho)$ by the Lagrange identity. Hence we have
 \begin{align*}
  W(f,g)(a) = \lim_{x\uparrow\alpha_\varrho} f(x)g^\qd(x) - f^\qd(x)g(x). 
 \end{align*}
 Now the claim follows since we may find some $g$ which lies in $\Tmax$ near $a$, with prescribed left-hand limits at $\alpha_\varrho$.
 Indeed, one may take $g$ to be a suitable solution of $\tau u=0$.
\end{proof}

Note that all functions in $\Tmin$ vanish outside of $(\alpha_\varrho,\beta_\varrho)$.
In general the operator $\Tmin$ is, because of
\begin{align*}
 \dom{\Tmin}^\bot = \mul{\Tmin^\ast} = \mul{\Tmax},
\end{align*}
not densely defined.
On the other side, $\dom{\Tmax}$ is always dense in the Hilbert space $\Lr$ since
\begin{align*}
 \dom{\Tmax}^\bot = \mul{\Tmax^\ast} = \mul{\Tmin} = \lbrace 0 \rbrace.
\end{align*}
Next we will show that $\Tmin$ always has self-adjoint extensions.

\begin{theorem}
 The deficiency indices of the minimal relation $\Tmin$ are equal and at most two, i.e.,
\begin{align}
  n(\Tmin) := \dim \ran(\Tmin - \I)^\bot = \dim \ran \left(\Tmin + \I\right)^\bot \leq 2.
\end{align}
\end{theorem}

\begin{proof}
 The fact that the dimensions are less than two, is a consequence of the inclusion
 \begin{align*}
  \ran(\Tmin \pm \I)^\bot = \ker(\Tmax \mp \I) \subseteq \ker(\Tloc \mp \I).
 \end{align*}
 Now since $\Tmin$ is real with respect to the natural conjugation in $\Lr$, we see that the natural conjugation is a conjugate-linear isometry
  from the kernel of $\Tmax+\I$ onto the kernel of $\Tmax-\I$ and hence their dimensions are equal.
\end{proof}

\section{Weyl's alternative}

We say $\tau$ is in the limit-circle (l.c.) case at $a$, if for each $z\in\C$ all solutions of $(\tau-z)u=0$ 
lie in $\Lr$ near $a$. Furthermore, we say $\tau$ is in the limit-point (l.p.) case at $a$, if for each $z\in\C$ 
there is some solution of $(\tau-z)u=0$ which does not lie in $\Lr$ near $a$.
Similarly one defines the l.c.~and l.p.~cases for the endpoint $b$.
It is clear that $\tau$ is only either in the l.c.\ or in the l.p.\ case at some boundary point.
The next lemma shows that $\tau$ indeed is in one of these cases at each endpoint.

\begin{lemma}\label{lemweylaltLC}
 If there is a $z_0\in\C$ such that all solutions of $(\tau-z_0) u = 0$ lie in $\Lr$ near $a$, 
 then $\tau$ is in the l.c.~case at $a$.
 A similar result holds at the endpoint $b$.
\end{lemma}

\begin{proof}
 Let $z\in\C$ and $u$ be a solution of $(\tau-z)u = 0$. 
 If $u_1$, $u_2$ are a fundamental system of $(\tau-z_0)u=0$ with $W(u_1,u_2)=1$, 
 then $u_1$ and $u_2$ lie in $\Lr$ near $a$ by assumption.
 Therefore, there is some $c\in(a,b)$ such that the function $v=|u_1| + |u_2|$ satisfies
\begin{align*}
 |z-z_0| \int_a^c v^2 d\varrho \leq \frac{1}{2}.
\end{align*}
 Since $u$ is a solution of $(\tau-z_0)u = (z-z_0)u$, we have for each $x\in(a,b)$
\begin{align*}
 u(x) = c_1u_1(x) + c_2u_2(x) + (z-z_0) \int_c^x \left( u_1(x) u_2(t) - u_1(t)u_2(x)\right) u(t)d\varrho(t)
\end{align*}
for some constants $c_1$, $c_2\in\C$ by Proposition~\ref{prop:repsol}. Therefore, we have
\begin{align*}
 |u(x)| \leq C v(x) + |z-z_0| v(x) \int_x^c v(t) |u(t)| d\varrho(t), \quad x\in(a,c),
\end{align*}
where $C= \max(|c_1|,|c_2|)$ and furthermore, using Cauchy--Schwarz
\begin{align*}
 |u(x)|^2 \leq 2 C^2 v(x)^2 + 2 |z-z_0|^2 v(x)^2 \int_x^c v(t)^2 d\varrho(t) \int_x^c |u(t)|^2 d\varrho(t).
\end{align*}
Now an integration yields for each $s\in(a,c)$
\begin{align*}
 \int_s^c |u|^2 d\varrho & \leq 2C^2 \int_a^c v^2 d\varrho + 2|z-z_0|^2 \left(\int_a^c v^2 d\varrho\right)^2 \int_s^c |u|^2 d\varrho \\
               & \leq 2C^2 \int_a^c v^2 d\varrho + \frac{1}{2} \int_s^c |u|^2d\varrho
\end{align*}
 and therefore
\begin{align*}
 \int_s^c |u|^2 d\varrho \leq 4C^2 \int_a^c v^2 d\varrho < \infty.
\end{align*}
Since $s\in(a,c)$ was arbitrary, this yields the claim.
\end{proof}

As an immediate consequence of Lemma~\ref{lemweylaltLC} we obtain: 

\begin{theorem}[Weyl's alternative]
 Each boundary point is either in the l.c.~case or in the l.p.~case. 
\end{theorem}

\begin{proposition}
 If $\tau$ is regular at $a$ or if $\varrho$ has no weight near $a$, then $\tau$ is in the l.c.~case at $a$.
 A similar result holds at the endpoint $b$.
\end{proposition}

\begin{proof}
 If $\tau$ is regular at $a$ each solution of $(\tau-z)u=0$ can be continuously extended to $a$.
 Hence $u$ is in $\Lr$ near $a$, since $\varrho$ is a finite measure near $a$.
 If $\varrho$ has no weight near $a$, each solution lies in $\Lr$ near $a$, since every solution is locally bounded.
\end{proof}

The set $\reg(\Tmin)$ of points of regular type of $\Tmin$ consists of all complex numbers $z\in\C$ such that
$(\Tmin-z)^{-1}$ is a bounded operator (not necessarily everywhere defined). Recall that $\dim \ran(\Tmin - z)^\bot$
is constant on every connected component of $\reg(\Tmin)$ (\cite[Theorem~8.1]{we80}) and thus
\begin{align*}
\dim \ran(\Tmin - z)^\bot = \dim \ker(\Tmax  - z^*) = n(\Tmin)
\end{align*}
 for every $z\in \reg(\Tmin)$.

\begin{lemma}\label{lemWeylRegType}
 For each $z\in \reg(\Tmin)$ there is a non-trivial solution of the equation $(\tau-z)u=0$ which lies in $\Lr$ near $a$.
 A similar result holds for the endpoint $b$.
\end{lemma}

\begin{proof}
Let $z\in \reg(\Tmin)$ and first assume that $\tau$ is regular at $b$. If there were no solutions of $(\tau-z)u=0$ which lie in $\Lr$
 near $a$, we would have $\ker(\Tmax-z) = \lbrace 0\rbrace$ and hence $n(\Tmin)=0$, i.e., $\Tmin=\Tmax$.
 But since there is an $f\in\Tmax$ with
 \begin{align*}
  f(b) = 1 \quad\text{and}\quad f^\qd(b) = 0,
 \end{align*}
 this is a contradiction to Theorem~\ref{thm:Tmin}.
 
 In the general case we take some $c\in(a,b)$ and consider the minimal operator $T_c$ in $L^2((a,c);\varrho)$ 
 induced by $\tau|_{(a,c)}$. Then $z$ is a point of regular type of $T_c$.
 Indeed, we can extend each $f_c\in\dom{T_c}$ by setting it equal to zero on $(c,b)$ and obtain a function $f\in\dom{\Tmin}$.
 For these functions and some positive constant $C$ we have
 \begin{align*}
  \left\| (T_c - z)f_c \right\|_c = \left\| (\Tmin-z)f \right\| \geq C \left\| f\right\| = C\left\| f_c\right\|_c,
 \end{align*}
 where $\|\cdot\|_c$ is the norm on $L^2((a,c);\varrho)$. Now since the solutions of the equation $(\tau|_{(a,c)}-z)u=0$ are exactly the solutions of
 $(\tau-z)u=0$ restricted to $(a,c)$, the claim follows from what we already proved.
\end{proof}

\begin{corollary}\label{cor:regtypeuniqsol}
 If $z\in \reg(\Tmin)$ and $\tau$ is in the l.p.~case at $a$, then there is a unique
 non-trivial solution of $(\tau-z)u=0$ (up to scalar multiples), which lies in $\Lr$ near $a$. A similar result holds for the endpoint $b$.
\end{corollary}

\begin{proof}
 If there were two linearly independent solutions in $\Lr$ near $a$, $\tau$ would be in the l.c.~case at $a$.
\end{proof}

\begin{lemma}\label{lem:lclpwronski}
 $\tau$ is in the l.p.~case at $a$ if and only if
\begin{align*}
 W(f,g)(a) = 0, \quad f,\,g\in\Tmax.
\end{align*}
$\tau$ is in the l.c.~case at $a$ if and only if there is an $f\in\Tmax$ such that
\begin{align*}
 W(f,f^\ast)(a) = 0 \quad\text{and}\quad W(f,g)(a)\not=0 \quad\text{for some }g\in\Tmax.
\end{align*}
Similar results hold at the endpoint $b$.
\end{lemma}

\begin{proof}
 Let $\tau$ be in the l.c.~case at $a$ and $u_1$, $u_2$ be a real fundamental system of $\tau u=0$ with $W(u_1,u_2)=1$.
 Both, $u_1$ and $u_2$ lie in $\Tmax$ near $a$. Hence there are $f$, $g\in\Tmax$ with $f=u_1$ and $g=u_2$ near $a$
 and $f=g=0$ near $b$. Then we have
 \begin{align*}
   W(f,g)(a) = W(u_1,u_2)(a) = 1
 \end{align*}
 and
 \begin{align*}
  W(f,f^\ast)(a) = W(u_1,u_1^\ast)(a) = 0
 \end{align*}
 since $u_1$ is real.

 Now assume $\tau$ is in the l.p.~case at $a$ and regular at $b$. Then $\Tmax$ is a two-dimensional extension of $\Tmin$, since 
 $\dim\ker(\Tmax-\I)=1$ by Corollary~\ref{cor:regtypeuniqsol}. Let $v$, $w\in\Tmax$ with $v=w=0$ in a vicinity of $a$ and
 \begin{align*}
  v(b) = w^\qd(b) = 1 \quad\text{and}\quad v^\qd(b) = w(b) = 0.
 \end{align*}
 Then
 \begin{align*}
  \Tmax = \Tmin + \linspan\lbrace v,w\rbrace,
 \end{align*}
 since $v$ and $w$ are linearly independent modulo $\Tmin$ and do not lie in $\Tmin$. Then for each $f$, $g\in\Tmax$ there are
 $f_0$, $g_0\in\Tmin$ such that $f=f_0$ and $g=g_0$ in a vicinity of $a$ and therefore
 \begin{align*}
  W(f,g)(a) = W(f_0,g_0)(a) = 0.
 \end{align*}
 Now if $\tau$ is not regular at $b$ we take some $c\in(a,b)$. Then for each $f\in\Tmax$ the function $f|_{(a,c)}$ lies in
 the maximal relation induced by $\tau|_{(a,c)}$ and the claim follows from what we already proved.
\end{proof}

\begin{lemma}\label{lem:lplpnosol}
 Let $\tau$ be in the l.p.~case at both endpoints and $z\in\C\backslash\R$. Then there is no non-trivial 
 solution of $(\tau-z)u=0$ in $\Lr$.
\end{lemma}

\begin{proof}
 If $u\in\Lr$ is a solution of $(\tau-z)u=0$, then $u^\ast$ is a solution of $(\tau-z^\ast)u=0$ 
 and both, $u$ and $u^\ast$ lie in $\Tmax$. Now the Lagrange identity yields for each $\alpha$, $\beta\in(a,b)$ with $\alpha<\beta$ 
 \begin{align*}
  W(u,u^\ast)(\beta) - W(u,u^\ast)(\alpha) = (z-z^\ast)\int_\alpha^\beta uu^\ast d\varrho = 2 \I\, \im(z) \int_\alpha^\beta |u|^2 d\varrho.
 \end{align*}
 As $\alpha\rightarrow a$ and $\beta\rightarrow b$, the left-hand side converges to zero by Lemma~\ref{lem:lclpwronski}
 and the right-hand side converges to $2 \I\, \im(z) \|u\|^2$, hence $\|u\|=0$.
\end{proof}

\begin{theorem}\label{thm:TminDefIndLCLP}
 The deficiency index of the minimal relation is given by
\begin{align*}
 n(\Tmin) = \begin{cases}
              0, & \text{if }\tau\text{ is in the l.c.~case at no boundary point}, \\
              1, & \text{if }\tau\text{ is in the l.c.~case at exactly one boundary point}, \\
              2, & \text{if }\tau\text{ is in the l.c.~case at both boundary points.} 
            \end{cases}
\end{align*}
\end{theorem}

\begin{proof}
 If $\tau$ is in the l.c.~case at both endpoints, all solutions of $(\tau-\I)u=0$ lie in $\Lr$ and hence in $\Tmax$.
 Therefore, $n(\Tmin) = \dim\ker(\Tmax - \I) = 2$. 
 In the case when $\tau$ is in the l.c.~case at exactly one endpoint, there is (up to scalar multiples) 
 exactly one non-trivial solution of $(\tau-\I)u=0$ in $\Lr$, by Corollary~\ref{cor:regtypeuniqsol}.
 Now if $\tau$ is in the l.p.~case at both endpoints, we have $\ker(\Tmax-\I)=\lbrace 0\rbrace$ by Lemma~\ref{lem:lplpnosol} 
 and hence $n(\Tmin)=0$.
\end{proof}

\section{Self-adjoint relations}

We are interested in the self-adjoint restrictions of $\Tmax$ (or equivalent the self-adjoint extensions of $\Tmin$).
To this end recall that we introduced the convenient short-hand notation
\begin{align*}
W_a^b(f,g^\ast)= W(f,g^\ast)(b) -W(f,g^\ast)(a), \quad f,\, g\in\Tmax.
\end{align*}

\begin{theorem}
 Some relation $S$ is a self-adjoint restriction of $\Tmax$ if and only if
\begin{align}
 S = \lbrace f\in\Tmax \,|\, \forall g\in S: W_a^b(f,g^\ast) = 0 \rbrace.
\end{align}
\end{theorem}

\begin{proof}
 We denote the right-hand side by $S_0$. First assume $S$ is a self-adjoint restriction of $\Tmax$. If $f\in S$, then
 \begin{align*}
  0 = \spr{\tau f}{g} - \spr{f}{\tau g} = W_a^b(f,g^\ast)
 \end{align*}
 for each $g\in S$, hence $f\in S_0$.
 Now if $f\in S_0$, then 
 \begin{align*}
  0 = W_a^b(f,g^\ast) = \spr{\tau f}{g} - \spr{f}{\tau g}
 \end{align*}
 for each $g\in S$ and hence $f\in S^\ast=S$.
 
 Conversely, assume that $S=S_0$, then $S$ is symmetric since we have $\spr{\tau f}{g}=\spr{f}{\tau g}$ for each $f$, $g\in S$.
 Now let $f\in S^\ast\subseteq \Tmax$, then
 \begin{align*}
  0 = \spr{\tau f}{g} - \spr{f}{\tau g} = W_a^b(f,g^\ast)
 \end{align*}
 for each $g\in S$ and hence $f\in S_0=S$.
\end{proof}

The aim of this section is to determine all self-adjoint restrictions of $\Tmax$. 
If both endpoints are in the l.p.~case, this is an immediate consequence of Theorem~\ref{thm:TminDefIndLCLP}.

\begin{theorem}
 If $\tau$ is in the l.p.~case at both endpoints then $\Tmin=\Tmax$ is a self-adjoint operator.
\end{theorem}

Next we turn to the case when one endpoint is in the l.c.~case and the other one is in the l.p.~case.
But before we do this, we need some more properties of the Wronskian.

\begin{lemma}\label{lem:WronskPropSR}
 Let $v\in\Tmax$ such that $W(v,v^\ast)(a) = 0$ and suppose there is an $h\in\Tmax$ with
    $W(h,v^\ast)(a)\not=0$. Then for each $f$, $g\in\Tmax$ we have
 \begin{align}\label{eqn:WronskLCconj}
  W(f,v^\ast)(a) = 0 \quad\Leftrightarrow\quad W(f^\ast,v^\ast)(a) = 0
 \end{align}
 and
 \begin{align}\label{eqn:WronskLC2}
  W(f,v^\ast)(a) = W(g,v^\ast)(a) = 0 \quad\Rightarrow\quad W(f,g)(a) = 0.
 \end{align}
 Similar results hold at the endpoint $b$.
\end{lemma}

\begin{proof}
 Choosing $f_1=v$, $f_2=v^\ast$, $f_3=h$ and $f_4=h^\ast$ in the Pl\"{u}cker identity, we see that also $W(h,v)(a)\not=0$.
 Now let $f_1=f$, $f_2=v$, $f_3=v^\ast$ and $f_4=h$, then the Pl\"{u}cker identity yields~\eqref{eqn:WronskLCconj},
 whereas $f_1=f$, $f_2=g$, $f_3=v^\ast$ and $f_4=h$ yields~\eqref{eqn:WronskLC2}.
\end{proof}

\begin{theorem}\label{thm:SRLCLPWronsk}
 Suppose $\tau$ is in the l.c.~case at $a$ and in the l.p.~case at $b$. Then some relation $S$ is a 
  self-adjoint restriction of $\Tmax$ if and only if there is a $v\in\Tmax\backslash\Tmin$ 
  with $W(v,v^\ast)(a)=0$ such that
 \begin{align}
  S = \lbrace f\in\Tmax \,|\, W(f,v^\ast)(a)=0 \rbrace.
 \end{align}
 A similar result holds if $\tau$ is in the l.c.~case at $b$ and in the l.p.~case at $a$.
\end{theorem}

\begin{proof}
 Because of $n(\Tmin)=1$ the self-adjoint extensions of $\Tmin$ are precisely the one-dimensional, 
  symmetric extensions of $\Tmin$. Hence some relation $S$ is a self-adjoint extension of $\Tmin$ if and only
  if there is some $v\in\Tmax\backslash\Tmin$ with $W(v,v^\ast)(a)=0$ such that
 \begin{align*}
   S = \Tmin \dot{+} \linspan\lbrace v\rbrace.
 \end{align*}
 Hence we have to prove that 
 \begin{align*}
  \Tmin \dot{+} \linspan\lbrace v\rbrace = \lbrace f\in\Tmax \,|\, W(f,v^\ast)(a) = 0 \rbrace.
 \end{align*}
 The subspace on the left-hand side is included in the right one because of 
  Theorem~\ref{thm:Tmin} and $W(v,v^\ast)(a)=0$.
 But if the subspace on the right-hand side was larger, it would be equal to $\Tmax$ and hence
  would imply $v\in\Tmin$.
\end{proof}

Two such self-adjoint restrictions are distinct if and only if the corresponding functions $v$ are linearly
 independent modulo $\Tmin$.
Furthermore, $v$ can always be chosen such that $v$ is equal to some real solution of $(\tau-z)u=0$ 
 with $z\in\R$ in some vicinity of $a$.
By Lemma~\ref{lem:WronskPropSR} one sees that all these self-adjoint restrictions are real with respect to the natural conjugation.

In contrast to the classical theory, not all of this self-adjoint restrictions of $\Tmax$ are operators.
We will determine which of them are multi-valued in the following section.

It remains to consider the case when both endpoints are in the l.c.~case.

\begin{theorem}\label{thm:SRLCLCWronsk}
 Suppose $\tau$ is in the l.c.~case at both endpoints. Then some relation $S$ is a self-adjoint restriction of
  $\Tmax$ if and only if there are some $v$, $w\in\Tmax$, linearly independent modulo $\Tmin$, with
 \begin{align}\label{eqn:WronskLCLCvw}
  W_a^b(v,v^\ast) = W_a^b(w,w^\ast) = W_a^b(v,w^\ast) = 0,
 \end{align}
 such that
 \begin{align}
  S = \lbrace f\in\Tmax \,|\, W_a^b(f,v^\ast) = W_a^b(f,w^\ast)=0 \rbrace.
 \end{align}
\end{theorem}

\begin{proof}
 Since $n(\Tmin)=2$, the self-adjoint extensions of $\Tmin$ are precisely the two-dimensional, symmetric
  extensions of $\Tmin$. Hence a relation $S$ is a self-adjoint restriction of $\Tmax$ if and only if there 
  are $v$, $w\in\Tmax$, linearly independent modulo $\Tmin$, with~\eqref{eqn:WronskLCLCvw} such that
 \begin{align*}
   S =  \Tmin \dot{+} \linspan\lbrace v,w\rbrace.
 \end{align*}
 Therefore, we have to prove that
 \begin{align*}
    \Tmin \dot{+} \linspan\lbrace v,w\rbrace = \lbrace f\in\Tmax \,|\, W_a^b(f,v^\ast) = W_a^b(f,w^\ast) = 0 \rbrace = T,
 \end{align*}
 where we denote the subspace on the right-hand side by $T$.
 Indeed the subspace on the left-hand side is contained in $T$ by Theorem~\ref{thm:Tmin} and~\eqref{eqn:WronskLCLCvw}.
 In order to prove that it is also not larger, consider the linear functionals $F_v$, $F_w$ on $\Tmax$ defined by
 \begin{align*}
  F_v(f) = W_a^b(f,v^\ast) \quad\text{and}\quad F_w(f) = W_a^b(f,w^\ast) \quad\text{for }f\in\Tmax.
 \end{align*}
 The intersection of the kernels of these functionals is precisely $T$. Furthermore, these functionals are 
  linearly independent. Indeed, assume $c_1$, $c_2\in\C$ and $c_1 F_v + c_2 F_w=0$, then for all $f\in\Tmax$ 
  we have
 \begin{align*}
  0 = c_1 F_v(f) + c_2 F_w(f) = c_1 W_a^b(f,v^\ast) + c_2 W_a^b(f,w^\ast) = W_a^b(f,c_1v^\ast + c_2w^\ast).
 \end{align*}
 But by Lemma~\ref{lem:funcdomtmax} this yields
 \begin{align*}
  W(f,c_1v^\ast + c_2w^\ast)(a) = W(f,c_1v^\ast + c_2w^\ast)(b) = 0
 \end{align*}
 for all $f\in\Tmax$ and hence $c_1v^\ast + c_2 w^\ast\in\Tmin$. Now since $v$, $w$ are linearly
  independent modulo $\Tmin$, we get that $c_1=c_2=0$. Now from Lemma~\ref{lem:hilfslemmaadjoint} we infer that
 \begin{align*}
  \ker F_v \not\subseteq \ker F_w \quad\text{and}\quad \ker F_w \not\subseteq \ker F_v.
 \end{align*}
 Hence there exist $f_v$, $f_w\in\Tmax$ such that $W_a^b(f_v,v^\ast)=W_a^b(f_w,w^\ast)=0$ but 
  $W_a^b(f_v,w^\ast)\not=0$ and $W_a^b(f_w,v^\ast) \not=0$. Both, $f_v$ and $f_w$ do not lie in $T$ and are linearly
  independent. Hence $T$ is at most a two-dimensional extension of the minimal relation $\Tmin$.
\end{proof}

In the case when $\tau$ is in the l.c.~case at both endpoints, we may divide the self-adjoint restrictions of
 $\Tmax$ into two classes. Indeed, we say some relation is a self-adjoint restriction of $\Tmax$ with
 separated boundary conditions if it is of the form
 \begin{align}
  S = \lbrace f\in\Tmax \,|\, W(f,v^\ast)(a)=W(f,w^\ast)(b)=0 \rbrace,
 \end{align}
 where $v$, $w\in\Tmax$ are such that $W(v,v^\ast)(a)=W(w,w^\ast)(b)=0$ but $W(h,v^\ast)(a)\not=0\not=W(h,w^\ast)(b)$ for some $h\in\Tmax$.
Conversely, each relation of this form is a self-adjoint restriction of $\Tmax$ by Theorem~\ref{thm:SRLCLCWronsk} 
 and Lemma~\ref{lem:funcdomtmax}.
The remaining self-adjoint restrictions are called self-adjoint restrictions of $\Tmax$ with coupled boundary
 conditions.

From Lemma~\ref{lem:WronskPropSR} one sees that all self-adjoint restrictions of $\Tmax$ with separated
 boundary conditions are real with respect to the natural conjugation in $\Lr$.
In the case of coupled boundary conditions this is not the case in general.
Again we will determine the self-adjoint restrictions which are multi-valued in the next section.

\section{Boundary conditions}\label{secBC}

\begin{subequations}
In this section let $w_1$, $w_2\in\Tmax$ with
\begin{align}\label{eqn:ufuncBCa}
 W(w_1,w_2^\ast)(a) = 1 \quad\text{and}\quad W(w_1,w_1^\ast)(a) = W(w_2,w_2^\ast)(a) = 0,
\end{align}
if $\tau$ is in the l.c.~case at $a$ and 
\begin{align}\label{eqn:ufuncBCb}
 W(w_1,w_2^\ast)(b) = 1 \quad\text{and}\quad W(w_1,w_1^\ast)(b) = W(w_2,w_2^\ast)(b) = 0,
\end{align}
if $\tau$ is in the l.c.~case at $b$.
We will describe the self-adjoint restrictions of $\Tmax$ in terms of the linear functionals 
 $\BCa^1$, $\BCa^2$, $\BCb^1$ and $\BCb^2$ on $\Tmax$, defined by
\begin{align*}
 \BCa^1(f) = W(f,w_2^\ast)(a) \quad\text{and}\quad \BCa^2(f) = W(w_1^\ast,f)(a) \quad\text{for }f\in\Tmax,
\end{align*}
if $\tau$ is in the l.c.~case at $a$ and 
\begin{align*}
 \BCb^1(f) = W(f,w_2^\ast)(b) \quad\text{and}\quad \BCb^2(f) = W(w_1^\ast,f)(b) \quad\text{for }f\in\Tmax,
\end{align*}
if $\tau$ is in the l.c.~case at $b$.
\end{subequations}

Note that if $\tau$ is in the l.c.~case at some endpoint, such functions $w_1$, $w_2\in\Tmax$ with~\eqref{eqn:ufuncBCa} (respectively 
 with~\eqref{eqn:ufuncBCb}) always exist. Indeed, one may take them to coincide near this endpoint with some 
 real solutions $u_1$, $u_2$ of $(\tau-z)u=0$ with $W(u_1,u_2)=1$ for some $z\in\R$ and use Lemma~\ref{lem:funcdomtmax}.

In the regular case these functionals may take the form of point evaluations of the function
 and its quasi-derivative at the boundary point.

\begin{proposition}\label{prop:PointEvalBC}
 Suppose $\tau$ is regular at $a$. Then there are $w_1$, $w_2\in\Tmax$ with~\eqref{eqn:ufuncBCa} such that
  the corresponding linear functionals $\BCa^1$ and $\BCa^2$ are given by
 \begin{align*}
  \BCa^1(f) = f(a) \quad \text{and}\quad \BCa^2(f) = f^\qd(a) \quad\text{for }f\in\Tmax.
 \end{align*}
 A similar result holds at the endpoint $b$.
\end{proposition}

\begin{proof}
 Take $w_1$, $w_2\in\Tmax$ to coincide near $a$ with the real solutions $u_1$, $u_2$ of $\tau u=0$ with the initial conditions 
\begin{align*}
 u_1(a) = u_2^\qd(a) = 1 \quad\text{and}\quad u_1^\qd(a) = u_2(a) = 0.
\end{align*}
\end{proof}

Moreover, also if $\varrho$ has no weight near some endpoint, we may choose special functionals.

\begin{proposition}\label{prop:BCLefthandlim}
 Suppose that $\varrho$ has no weight near $a$, i.e., $\alpha_\varrho>a$. Then there are $w_1$, $w_2\in\Tmax$ with~\eqref{eqn:ufuncBCa} such that
  the corresponding linear functionals $\BCa^1$ and $\BCa^2$ are given by
 \begin{align*}
  \BCa^1(f) =  f(\alpha_\varrho-) \quad \text{and}\quad 
  \BCa^2(f) =  f^\qd(\alpha_\varrho-) \quad\text{for }f\in\Tmax.
 \end{align*}
 A similar result holds at the endpoint $b$.
\end{proposition}

\begin{proof}
 Take $w_1$, $w_2\in\Tmax$ to coincide near $a$ with the real solutions $u_1$, $u_2$ of $\tau u=0$ with the initial conditions 
\begin{align*}
  u_1(\alpha_\varrho-) = u_2^\qd(\alpha_\varrho-) = 1 
 \quad\text{and}\quad 
  u_1^\qd(\alpha_\varrho-) = u_2(\alpha_\varrho-) = 0.
\end{align*}
Then since the Wronskian is constant on $(a,\alpha_\varrho)$, we get
\begin{align*}
 BC_a^1(f) & = W(f,u_2)(\alpha_\varrho-) = f(\alpha_\varrho-)
\end{align*}
and
\begin{align*}
 BC_a^2(f) & = W(u_1,f)(\alpha_\varrho-) = f^\qd(\alpha_\varrho-)
\end{align*}
for each $f\in\Tmax$.
\end{proof}

Using the Pl\"{u}cker identity one easily obtains the equality
\begin{align*}
 W(f,g)(a) = \BCa^1(f)\BCa^2(g) - \BCa^2(f)\BCa^1(g), \quad f,\,g\in\Tmax
\end{align*}
for the Wronskian. 
Furthermore, for each $v\in\Tmax\backslash\Tmin$ with $W(v,v^\ast)(a)=0$ but $W(h,v^\ast)(a)\not=0$ for some $h\in\Tmax$, one may show that 
 there is a $\varphi_\alpha\in[0,\pi)$ such that for each $f\in\Tmax$ 
\begin{align}\label{eqnBCrelavtophi}
 W(f,v^\ast)(a) = 0 \quad\Leftrightarrow\quad \BCa^1(f) \cos\varphi_\alpha - \BCa^2(f)\sin\varphi_\alpha=0.
\end{align}
Conversely, if some $\varphi_\alpha\in[0,\pi)$ is given, then there is some $v\in\Tmax\backslash\Tmin$ with 
 $W(v,v^\ast)(a) = 0$ but $W(h,v^\ast)(a)\not=0$ for some $h\in\Tmax$ such that
\begin{align}\label{eqnBCrelaphitov}
 W(f,v^\ast)(a) = 0 \quad\Leftrightarrow\quad \BCa^1(f) \cos\varphi_\alpha - \BCa^2(f)\sin\varphi_\alpha=0
\end{align}
for each $f\in\Tmax$. 
Using this, Theorem~\ref{thm:SRLCLPWronsk} immediately yields the following characterization of the
 self-adjoint restrictions of $\Tmax$ in terms of the boundary functionals.

\begin{theorem}\label{thm:bcLCLP}
 Suppose $\tau$ is in the l.c.~case at $a$ and in the l.p.~case at $b$. Then some relation $S$ is a 
  self-adjoint restriction of $\Tmax$ if and only if 
 \begin{align*}
  S = \lbrace f\in\Tmax \,|\, \BCa^1(f) \cos\varphi_\alpha - \BCa^2(f) \sin\varphi_\alpha = 0\rbrace
 \end{align*}
 for some $\varphi_\alpha\in[0,\pi)$. A similar result holds if $\tau$ is in the l.c.~case at $b$ and in the l.p.~case at $a$.
\end{theorem} 

Now we will determine which self-adjoint restrictions of $\Tmax$ are operators in this case.
Of course, we only have to consider the case when $\alpha_\varrho>a$ and $\varrho$ has mass in $\alpha_\varrho$.

\begin{corollary}\label{corBCLCLPmulval}
 Suppose $\varrho$ has mass in $\alpha_\varrho$ and
  $\tau$ is in the l.p.~case at $b$. Then some self-adjoint 
  restriction $S$ of $\Tmax$ as in Theorem~\ref{thm:bcLCLP} is an operator if and only if 
 \begin{align}\label{eqnbcLCLPmulval}
  \cos\varphi_\alpha w_2(\alpha_\varrho-) + \sin\varphi_\alpha w_1(\alpha_\varrho-)\not=0.
 \end{align}
 A similar result holds for the endpoint $b$.
\end{corollary}
 
\begin{proof}
 Assume~\eqref{eqnbcLCLPmulval} does not hold and for each $c\in\C$ consider the functions
 \begin{align}\label{eqnBCLCLPfuncmulS}
  f_c(x) = \begin{cases}
            c\, u_a(x), & \text{if }x\in(a,\alpha_\varrho], \\
            0, & \text{if }x\in(\alpha_\varrho,b),
           \end{cases}
 \end{align}
 where $u_a$ is a solution of $\tau u=0$ with $u_a(\alpha_\varrho)=0$ and $u_a^\qd(\alpha_\varrho)=1$. 
 These functions lie in $S$ with $\tau f_c\not=0$, hence $S$ is multi-valued.
Conversely, assume~\eqref{eqnbcLCLPmulval} holds and let $f\in S$ such that $f=0$ and $\tau f=0$ almost
 everywhere with respect to $\varrho$. 
Then $f$ is of the form~\eqref{eqnBCLCLPfuncmulS}, but because of the boundary condition 
 \begin{align*}
  c = f^\qd(\alpha_\varrho) = f(\alpha_\varrho) \frac{\cos\varphi_\alpha w_2^\qd(\alpha_\varrho)^\ast + \sin\varphi_\alpha w_1^\qd(\alpha_\varrho)^\ast}
       {\cos\varphi_\alpha w_2(\alpha_\varrho)^\ast + \sin\varphi_\alpha w_1(\alpha_\varrho)^\ast} = 0,
 \end{align*}
 i.e., $f=0$.
\end{proof}

Note that in this case there is precisely one multi-valued, self-adjoint restriction $S$ of $\Tmax$.
In terms of the boundary functionals from Proposition~\ref{prop:BCLefthandlim} it is precisely the one with $\varphi_\alpha=0$.
That means that in this case each function in $S$ vanishes in $\alpha_\varrho$. Now since $\varrho$ has mass in this point one sees that the domain of $S$ is not dense and hence $S$ is not an operator. However, if we exclude the linear span of $\indik_{\alpha_\varrho}$ from $\Lr$ by setting
\begin{align*}
 \D = \overline{\dom{S}} = \Lr \ominus \linspan\lbrace\indik_{\lbrace\alpha_\varrho\rbrace}\rbrace,
\end{align*}
the linear relation $S_\D$ in the Hilbert space $\D$, given by
\begin{align*}
 S_\D = S \cap\left(\D\times\D\right),
\end{align*}
is a self-adjoint operator (see~\eqref{eqnLRSDef} in Appendix~\ref{appLR}). 
Also note that if $\tilde{\tau}$ is obtained from $\tau$ by removing the point mass in $\alpha_\varrho$ from the measure $\varrho$, then $S_\D$ is a self-adjoint restriction of the maximal relation corresponding to $\tilde{\tau}$.

Next we will give a characterization of the self-adjoint restrictions of $\Tmax$, if $\tau$ is in the l.c.~case at both endpoints.

\begin{theorem}\label{thm:LCLCBC}
 Suppose $\tau$ is in the l.c.~case at both endpoints. Then some relation $S$ is a self-adjoint restriction 
  of $\Tmax$ if and only if there are matrices $B_a$, $B_b\in\C^{2\times 2}$ with
 \begin{align}\label{eqnBCLCLCmatcond}
  \rang(B_a | B_b) = 2 \quad\text{and}\quad B_aJB_a^\ast = B_bJB_b^\ast \quad\text{with}\quad 
     J=\begin{pmatrix} 0 & -1 \\ 1 & 0 \end{pmatrix},
 \end{align}
 such that
 \begin{align}
  S = \left\lbrace f\in\Tmax \left|\, B_a\begin{pmatrix} \BCa^1(f) \\ \BCa^2(f) \end{pmatrix} 
    = B_b \begin{pmatrix} \BCb^1(f) \\ \BCb^2(f) \end{pmatrix}\right. \right\rbrace.
 \end{align}
\end{theorem}

\begin{proof}
If $S$ is a self-adjoint restriction of $\Tmax$, there exist $v$, $w\in\Tmax$, linearly independent
 modulo $\Tmin$, with 
 \[W_a^b(v,v^\ast)=W_a^b(w,w^\ast)=W_a^b(v,w^\ast)=0,\] such that 
 \[S = \lbrace f\in\Tmax \,|\, W_a^b(f,v^\ast)=W_a^b(f,w^\ast)=0\rbrace.\]
 Let $B_a$, $B_b\in\C^{2\times2}$ be defined by 
 \[B_a=\left(\begin{matrix} \BCa^2(v^\ast) & -\BCa^1(v^\ast) \\ \BCa^2(w^\ast) & -\BCa^1(w^\ast) \end{matrix}\right)\quad\text{and}\quad 
   B_b=\left(\begin{matrix} \BCb^2(v^\ast) & -\BCb^1(v^\ast) \\ \BCb^2(w^\ast) & -\BCb^1(w^\ast) \end{matrix}\right).\] 
 Then a simple computation shows that 
 \[B_aJB_a^*=B_bJB_b^*\quad\Leftrightarrow\quad W_a^b(v,v^\ast)=W_a^b(w,w^\ast)=W_a^b(v,w^\ast)=0.\] 
 In order to prove $\rang{(B_a|B_b)}=2$, let $c_1$, $c_2\in\C$ and 
 \[0=c_1\begin{pmatrix} \BCa^2(v^\ast) \\ -\BCa^1(v^\ast) \\ \BCb^2(v^\ast) \\ -\BCb^1(v^\ast) \end{pmatrix} + c_2 \begin{pmatrix} \BCa^2(w^\ast) \\ -\BCa^1(w^\ast) \\ \BCb^2(w^\ast) \\ -\BCb^1(w^\ast) \end{pmatrix}=\begin{pmatrix} \BCa^2(c_1v^\ast+c_2w^\ast) \\ -\BCa^1(c_1v^\ast+c_2w^\ast) \\ \BCb^2(c_1v^\ast+c_2w^\ast) \\ -\BCb^1(c_1v^\ast+c_2w^\ast) \end{pmatrix}.\] 
 Hence the function $c_1v^\ast+c_2w^\ast$ lies in the kernel of $\BCa^1$, $\BCa^2$, $\BCb^1$ and $\BCb^2$,
 therefore $W(c_1v^\ast+c_2w^\ast,f)(a)=0$ and $W(c_1v^\ast+c_2w^\ast,f)(b)=0$ for each $f\in\Tmax$. 
 This means that $c_1v^\ast+c_2w^\ast\in\Tmin$ and hence $c_1=c_2=0$, since $v^\ast$, $w^\ast$ are 
 linearly independent modulo $\Tmin$. This proves that $(B_a|B_b)$ has rank two. 
 Furthermore, a calculation yields that for each $f\in\Tmax$ 
 \[W_a^b(f,v^\ast)=W_a^b(f,w^\ast)=0\quad\Leftrightarrow\quad B_a\left(\begin{matrix}\BCa^1(f)\\ \BCa^2(f)\end{matrix}\right)
    =B_b\left(\begin{matrix} \BCb^1(f)\\ \BCb^2(f)\end{matrix}\right),\] which proves that $S$ is given as in the claim.

Conversely, let $B_a$, $B_b\in\C^{2\times2}$ with the claimed properties be given.
Then there are $v$, $w\in\Tmax$ such that 
\[B_a=\left(\begin{matrix} \BCa^2(v^\ast) & -\BCa^1(v^\ast) \\ \BCa^2(w^\ast) & -\BCa^1(w^\ast) \end{matrix}\right)
   \quad\text{and}\quad B_b=\left(\begin{matrix} \BCb^2(v^\ast) & -\BCb^1(v^\ast) \\ \BCb^2(w^\ast) & -\BCb^1(w^\ast) \end{matrix}\right).\]
In order to prove that $v$ and $w$ are linearly independent modulo $\Tmin$, assume $c_1v+c_2w\in\Tmin$ for some $c_1$, $c_2\in\C$, then
 \[0=\begin{pmatrix} \BCa^2(c_1^\ast v^\ast+c_2^\ast w^\ast) \\ -\BCa^1(c_1^\ast v^\ast+c_2^\ast w^\ast) \\
       \BCb^2(c_1^\ast v^\ast+c_2^\ast w^\ast) \\ -\BCb^1(c_1^\ast v^\ast+c_2^\ast w^\ast) \end{pmatrix}
    = c_1^\ast \begin{pmatrix} \BCa^2(v^\ast) \\ -\BCa^1(v^\ast) \\ \BCb^2(v^\ast) \\ -\BCb^1(v^\ast) \end{pmatrix} 
  + c_2^\ast \begin{pmatrix} \BCa^2(w^\ast) \\ -\BCa^1(w^\ast) \\ \BCb^2(w^\ast) \\ -\BCb^1(w^\ast) \end{pmatrix}.\] 
Now the rows of $(B_a|B_b)$ are linearly independent, hence $c_1=c_2=0$. 
Since again we have 
\[B_aJB_a^*=B_bJB_b^*\quad\Leftrightarrow\quad W_a^b(v,v^\ast)=W_a^b(w,w^\ast)=W_a^b(v,w^\ast)=0,\] 
the functions $v$, $w$ satisfy the assumptions of Theorem~~\ref{thm:SRLCLCWronsk}. 
As above one sees again that for each $f\in\Tmax$ 
\[B_a\left(\begin{matrix}\BCa^1(f)\\ \BCa^2(f)\end{matrix}\right)=B_b\left(\begin{matrix} \BCb^1(f)\\ \BCb^2(f)\end{matrix}\right)
  \quad\Leftrightarrow\quad W_a^b(f,w^\ast)=W_a^b(f,w^\ast)=0.\] 
 Hence $S$ is a self-adjoint restriction of $\Tmax$ by Theorem~\ref{thm:SRLCLCWronsk}.
\end{proof}

As in the preceding section, if $\tau$ is in the l.c.~case at both endpoints, we may divide the self-adjoint 
 restrictions of $\Tmax$ into two classes.
 
\begin{theorem}\label{thm:SRLCLCsepcoup}
Suppose $\tau$ is in the l.c.~case at both endpoints. 
Then some relation $S$ is a self-adjoint restriction of $\Tmax$ with separated boundary conditions if and only 
 if there are $\varphi_\alpha$, $\varphi_\beta\in[0,\pi)$ such that
 \begin{align}\label{eqn:SRLCLCsep}
 S = \left\lbrace f\in\Tmax \left| \begin{array}{l} \BCa^1(f)\cos\varphi_\alpha-\BCa^2(f)\sin\varphi_\alpha=0\\
         \BCb^1(f)\cos\varphi_\beta-\BCb^2(f)\sin\varphi_\beta=0 \end{array}\right.\right\rbrace.
 \end{align} 
 Furthermore, $S$ is a self-adjoint restriction of $\Tmax$ with coupled boundary conditions if and only if
 there are $\varphi\in[0,\pi)$ and $R\in\R^{2\times2}$ with $\det{R}=1$ such that
 \begin{align}\label{eqn:SRLCLCcoup}
 S = \left\lbrace f\in\Tmax \left| \left(\begin{matrix}\BCb^1(f)\\\BCb^2(f)\end{matrix}\right) 
      = e^{\I\varphi}R\left(\begin{matrix}\BCa^1(f)\\\BCa^2(f)\end{matrix}\right)\right.\right\rbrace.
 \end{align} 
\end{theorem}

\begin{proof}
Using~\eqref{eqnBCrelavtophi} and~\eqref{eqnBCrelaphitov} one easily sees that the self-adjoint restrictions
 of $\Tmax$ are precisely the ones given in~\eqref{eqn:SRLCLCsep}. Hence we only have to prove the second claim.
Let $S$ be a self-adjoint restriction of $\Tmax$ with coupled boundary conditions and 
 $B_a$, $B_b\in\C^{2\times 2}$ matrices as in Theorem~\ref{thm:LCLCBC}.
Then by~\eqref{eqnBCLCLCmatcond} either both of them have rank one or both have rank two.
In the first case we had
\begin{align*}
 B_a d = c_a^T d\, w_a \quad\text{and}\quad B_b d = c_b^T d\, w_b, \quad d\in\C^2
 \end{align*}
for some nonzero $c_a$, $c_b$, $w_a$, $w_b\in\C^2$.
Since the vectors $w_a$ and $w_b$ are linearly independent by $\rang(B_a|B_b)=2$ we have for each $f\in\Tmax$ 
\begin{align*}
 B_a\begin{pmatrix} \BCa^1(f) \\ \BCa^2(f) \end{pmatrix} = B_b \begin{pmatrix} \BCb^1(f) \\ \BCb^2(f) \end{pmatrix}
 \quad\Leftrightarrow\quad  B_a\begin{pmatrix} \BCa^1(f) \\ \BCa^2(f) \end{pmatrix} = B_b \begin{pmatrix} \BCb^1(f) \\ \BCb^2(f) \end{pmatrix}=0.
\end{align*}
In particular, this shows
\begin{align*}
 B_a J B_a^\ast = B_b J B_b^\ast \quad\Leftrightarrow\quad B_a JB_a^\ast = B_b J B_b^\ast = 0.
\end{align*}
Now let $v\in\Tmax$ with $\BCa^2(v^\ast)=c_1$ and $\BCa^1(v^\ast)=-c_2$. 
A simple calculation yields \begin{align*}0=B_aJB_a^* & = W(w_1,w_2)(a)(\BCa^1(v)\BCa^2(v^\ast)-\BCa^2(v)\BCa^1(v^\ast))w_a w_a^{\ast T} \\ & = W(w_1,w_2)(a)W(v,v^\ast)(a)w_a w_a^{\ast T}.\end{align*} 
Hence $W(v,v^\ast)(a)=0$ and since $(\BCa^1(v),\BCa^2(v))=(c_2,c_1)\not=0$, we also have $v\not\in\Tmin$. 
Furthermore, for each $f\in\Tmax$ we have
\begin{align*}
B_a\begin{pmatrix} \BCa^1(f) \\ \BCa^2(f)\end{pmatrix} & =(\BCa^1(f)\BCa^2(v^\ast) - \BCa^2(f)\BCa^1(v^\ast))w_a \\
                                                       & = W(f,v^\ast)(a)w_a.
\end{align*}
Similarly one gets a function $w\in\Tmax\backslash\Tmin$ with $W(w,w^\ast)(b)=0$ and
\begin{align*}
 B_b\begin{pmatrix} \BCb^1(f) \\ \BCb^2(f)\end{pmatrix}= W(f,w^\ast)(b) w_b, \quad f\in\Tmax.
\end{align*}
But this shows that $S$ is a self-adjoint restriction with separated boundary conditions.

Hence both matrices, $B_a$ and $B_b$ have rank two. If we set $B=B_b^{-1}B_a$, 
then $B=J(B^{-1})^\ast J^\ast$ and therefore $|\det B |=1$, hence $\det B = e^{2i\varphi}$ for some $\varphi\in[0,\pi)$.
If we set $R=e^{-i\varphi}B$, one sees from the equation
\begin{align*}
 B & = \begin{pmatrix} b_{11} & b_{12} \\ b_{21} & b_{22} \end{pmatrix} 
     = J(B^{-1})^\ast J^\ast = e^{2i\varphi} \begin{pmatrix} 0 & -1 \\ 1 & 0 \end{pmatrix}
       \begin{pmatrix} b_{22}^\ast & - b_{21}^\ast \\ -b_{12}^\ast & b_{11}^\ast \end{pmatrix}
       \begin{pmatrix} 0 & 1 \\ -1 & 0 \end{pmatrix} \\
   & = e^{2i\varphi} \begin{pmatrix} b_{11}^\ast & b_{12}^\ast \\ b_{21}^\ast & b_{22}^\ast \end{pmatrix}
\end{align*}
that $R\in\R^{2\times2}$ with $\det R=1$.
Now because we have for each $f\in\Tmax$ 
\begin{align*}
 B_a\begin{pmatrix} \BCa^1(f) \\ \BCa^2(f) \end{pmatrix} =
   B_b\begin{pmatrix} \BCb^1(f) \\ \BCb^2(f) \end{pmatrix} \quad\Leftrightarrow\quad 
   \begin{pmatrix} \BCb^1(f) \\ \BCb^2(f) \end{pmatrix} = e^{\I\varphi} R 
    \begin{pmatrix} \BCa^1(f) \\ \BCa^2(f) \end{pmatrix},
\end{align*}
$S$ has the claimed representation.

Conversely, if $S$ is of the form~\eqref{eqn:SRLCLCcoup}, then Theorem~\ref{thm:LCLCBC} shows that it is a 
  self-adjoint restriction of $\Tmax$. Now if $S$ was a self-adjoint restriction with separated boundary
 conditions, we would have an $f\in S\backslash\Tmin$, vanishing in some vicinity of $a$. But then, because of the boundary condition 
 we would also have $\BCb^1(f)=\BCb^2(f)=0$, i.e., $f\in\Tmin$. Hence $S$ can not be a self-adjoint restriction with
 separated boundary conditions.
\end{proof}

\begin{subequations}
Now we will again determine the self-adjoint restrictions of $\Tmax$ which are multi-valued.
In the case of separated boundary conditions these are determined by whether
 \begin{align}\label{eqnBCLCLCmulvala}
   \cos\varphi_\alpha w_2(\alpha_\varrho-) + \sin\varphi_\alpha w_1(\alpha_\varrho-) & \not= 0, \\
 \label{eqnBCLCLCmulvalb}
   \cos\varphi_\beta w_2(\beta_\varrho+) + \sin\varphi_\beta w_1(\beta_\varrho+) & \not= 0,
 \end{align}\end{subequations}
hold or not.
Note that if one takes the functionals from Proposition~\ref{prop:BCLefthandlim}, then~\eqref{eqnBCLCLCmulvala}
 (respectively~\eqref{eqnBCLCLCmulvalb}) is equivalent to $\varphi_\alpha\not=0$ (respectively $\varphi_\beta\not=0$). 
We start with the case when $\dim\mul{\Tmax}=1$.

\begin{corollary}\label{corBCLCLCmulvala}
 Suppose $\tau$ is in the l.c.~case at both endpoints and $\varrho$ has mass in $\alpha_\varrho$ but not 
  in $\beta_\varrho$, i.e., $\dim\mul{\Tmax}=1$.
 Then for each self-adjoint restriction $S$ of $\Tmax$ with separated boundary conditions as in 
  Theorem~\ref{thm:SRLCLCsepcoup} we have
 \begin{align*}
  \mul{S} = \begin{cases}
            \lbrace 0\rbrace, & \text{if \eqref{eqnBCLCLCmulvala} holds}, \\
            \linspan\lbrace\indik_{\lbrace\alpha_\varrho\rbrace}\rbrace, & \text{if \eqref{eqnBCLCLCmulvala} does not hold}.
           \end{cases}
 \end{align*}
 Furthermore, each self-adjoint restriction of $\Tmax$ with coupled boundary conditions is an operator.
 Similar results hold if $\varrho$ has mass in $\beta_\varrho$ but no mass in $\alpha_\varrho$.
\end{corollary}

\begin{proof}
If $S$ is a self-adjoint restriction of $\Tmax$ with separated boundary conditions,
 then the claim follows as in the proof of Corollary~\ref{corBCLCLPmulval}.

 Now let $S$ be a self-adjoint restriction of $\Tmax$ with coupled boundary conditions as in 
  Theorem~\ref{thm:SRLCLCsepcoup} and $f\in S$ with $f=0$ and $\tau f=0$ almost everywhere with respect 
  to $\varrho$. Then again $f$ is of the form~\eqref{eqnBCLCLPfuncmulS}.
  But because of the boundary conditions this shows that $\BCa^1(f)=\BCa^2(f)=0$, hence $f$ vanishes everywhere.
\end{proof}

The remark after Corollary~\ref{corBCLCLPmulval} also holds literally here under the assumptions of Corollary~\ref{corBCLCLCmulvala}.
It remains to determine the self-adjoint restrictions of $\Tmax$ which are multi-valued in the case 
 when $\varrho$ has mass in $\alpha_\varrho$ and in $\beta_\varrho$.

\begin{corollary}
 Suppose $\varrho$ has mass in $\alpha_\varrho$ and in $\beta_\varrho$, i.e., $\dim\mul{\Tmax}=2$.
 If $S$ is a self-adjoint restriction of $\Tmax$ with separated boundary conditions as in 
 Theorem~\ref{thm:SRLCLCsepcoup}, then
 \begin{align*}
  \mul{S} = \begin{cases}
                 \lbrace 0 \rbrace, & 
      \text{if \eqref{eqnBCLCLCmulvala} and \eqref{eqnBCLCLCmulvalb} hold}, \\
                 \linspan\lbrace\indik_{\lbrace\alpha_\varrho\rbrace}\rbrace, &
       \text{if \eqref{eqnBCLCLCmulvalb} holds and \eqref{eqnBCLCLCmulvala} does not}, \\
                 \linspan\lbrace\indik_{\lbrace\beta_\varrho\rbrace}\rbrace, & 
      \text{if \eqref{eqnBCLCLCmulvala} holds and \eqref{eqnBCLCLCmulvalb} does not}, \\
                 \linspan\lbrace\indik_{\lbrace\alpha_\varrho\rbrace},\indik_{\lbrace\beta_\varrho\rbrace}\rbrace, &
    \text{if neither \eqref{eqnBCLCLCmulvala} nor \eqref{eqnBCLCLCmulvalb} holds}.    
               \end{cases}
 \end{align*}
 If $S$ is a self-adjoint restriction of $\Tmax$ with coupled boundary conditions as in 
 Theorem~\ref{thm:SRLCLCsepcoup} and
 \begin{align*}
    \tilde{R} = \begin{pmatrix}
                    w_2^\qd(\beta_\varrho+)^\ast & -w_2(\beta_\varrho+)^\ast \\ -w_1^\qd(\beta_\varrho+)^\ast & w_1(\beta_\varrho+)^\ast
                  \end{pmatrix}^{-1}
     R \begin{pmatrix}
        w_2^\qd(\alpha_\varrho-)^\ast & -w_2(\alpha_\varrho-)^\ast \\ -w_1^\qd(\alpha_\varrho-)^\ast & w_1(\alpha_\varrho-)^\ast
       \end{pmatrix},
 \end{align*}
 then 
 \begin{align*}
  \mul{S} = \begin{cases}
      \lbrace 0\rbrace, & \text{if }\tilde{R}_{12}\not=0, \\
      \linspan\lbrace \indik_{\lbrace\alpha_\varrho\rbrace} + e^{\I\varphi} \tilde{R}_{22} \indik_{\lbrace\beta_\varrho\rbrace} \rbrace, & 
                  \text{if }\tilde{R}_{12}=0.
      \end{cases}
 \end{align*}
\end{corollary}

\begin{proof}
 If $S$ is a self-adjoint restriction of $\Tmax$ with separated boundary conditions, 
  the claim follows as in the proof of Corollary~\ref{corBCLCLPmulval}.

 In order to prove the second part let $S$ be a self-adjoint restriction of $\Tmax$ with coupled boundary 
  conditions, which can be written as
 \begin{align*}
   \begin{pmatrix} f(\beta_\varrho+) \\ f^\qd(\beta_\varrho+) \end{pmatrix} & = 
     e^{\I\varphi} \tilde{R}
   \begin{pmatrix} f(\alpha_\varrho-) \\ f^\qd(\alpha_\varrho-) \end{pmatrix}, \quad f\in S.
 \end{align*}
 Now assume that $\tilde{R}_{12}\not=0$ and let $f\in S$ with $f=0$ almost everywhere with respect to $\varrho$. 
  Then from this boundary condition we infer that also $f^\qd(\alpha_\varrho-)=f^\qd(\beta_\varrho+)=0$, i.e., $f=0$.
 Otherwise, if we assume that $\tilde{R}_{12}=0$, then the boundary condition becomes 
 \begin{align*}
  f^\qd(\beta_\varrho+) = e^{\I\varphi} \tilde{R}_{22} f^\qd(\alpha_\varrho-), \quad f\in S.
 \end{align*}
 Hence all functions $f\in S$ which vanish almost everywhere with respect to $\varrho$ are of the form
 \begin{align*}
  f(x) = \begin{cases}
          c_a u_a(x), & \text{if }x\in(a,\alpha_\varrho], \\
          0,          & \text{if }x\in(\alpha_\varrho,\beta_\varrho], \\
          e^{\I\varphi} \tilde{R}_{22} c_a u_b(x), & \text{if } x\in(\beta_\varrho,b).
         \end{cases}
 \end{align*}
 Conversely, all functions of this form lie in $S$, which yields the claim.
\end{proof}

Note that if one uses the boundary functionals of Proposition~\ref{prop:BCLefthandlim}, then $\tilde{R}=R$.
In contrast to Corollary~\ref{corBCLCLCmulvala}, in this case there is a multitude of multi-valued, self-adjoint restrictions $S$ of $\Tmax$. However, if we again restrict $S$ to the closure $\D$ of the domain of $S$ by
\begin{align*}
S_\D = S\cap\left(\D\times\D\right),
\end{align*}
we obtain a self-adjoint operator in the Hilbert space $\D$.

\begin{remark}\label{rem:bcev}
Under the same assumptions as in Proposition~\ref{prop:BCLefthandlim} and given the linear functionals $\BCa^1$, $\BCa^2$ as in this proposition plus solutions $u$ of $(\tau-z)u=0$, we have
\begin{align*}
 \BCa^1(u) = u(\alpha_\varrho-) \quad\text{and}\quad \BCa^2(u) = u^\qd(\alpha_\varrho+) + z\varrho(\lbrace\alpha_\varrho\rbrace) u(\alpha_\varrho).
\end{align*}
Hence, upon adding a suitable point mass in $\alpha_\varrho$ to $\varrho$, it is possible to end up with eigenvalue dependent boundary conditions. Moreover, adding point masses to the left of $\alpha_\varrho$ to $\varrho$ and varying the remaining coefficients even yields far more general eigenvalue dependent boundary conditions than that. In particular, it is possible to obtain boundary conditions which depend polynomially on the eigenvalue parameter. These considerations show that our general theory also includes self-adjoint operators associated with Sturm--Liouville problems with certain eigenvalue dependent boundary conditions.
\end{remark}

\section{Spectrum and resolvent}\label{secSR}

In this section we will compute the resolvent of the self-adjoint restrictions $S$ of $\Tmax$.
The resolvent set $\rho(S)$ is the set of all $z\in\C$ such that 
\begin{align*}
 R_z = (S-z)^{-1} = \lbrace (g,f)\in\Lr\times\Lr \,|\, (f,g)\in S-z \rbrace
\end{align*}
is an everywhere defined operator in $\Lr$, i.e., $\dom{R_z}=\Lr$ and $\mul{R_z} = \lbrace0\rbrace$.
According to Theorem~\ref{AppLRthmres}, the resolvent set $\rho(S)$ is a non-empty, open subset of $\C$ and the resolvent $z\mapsto R_z$ is an
 analytic function of $\rho(S)$ into the space of bounded linear operators on $\Lr$.
Note that in general the operators $R_z$, $z\in\rho(S)$ are not injective, indeed we have
\begin{align}\label{eqnSpecRessets}
 \ker(R_z) = \mul{S} = \dom{S}^\bot = \ran(R_z)^\bot, \quad z\in\rho(S).
\end{align}
First we deal with the case, when both endpoints are in the l.c.~case.

\begin{theorem}\label{thmSResolLCLC}
 Suppose $\tau$ is in the l.c.~case at both endpoints and $S$ is a self-adjoint restriction of $\Tmax$.
 Then for each $z\in\rho(S)$ the resolvent $R_z$ is an integral operator
 \begin{align}
  R_z g (x) = \int_a^b G_z(x,y) g(y) d\varrho(y), \quad x\in(a,b),~g\in\Lr,
 \end{align}
 with a  square integrable kernel $G_z$ (in particular, $R_z$ is Hilbert--Schmidt). For any given linearly independent solutions $u_1$, $u_2$ of $(\tau-z)u=0$,
 there are coefficients $m^\pm_{ij}(z)\in\C$, $i$, $j\in\lbrace 1,2\rbrace$ such that the kernel is given by
 \begin{align}
  G_z(x,y) = \begin{cases}
               \sum_{i,j=1}^2 m^+_{ij}(z) u_i(x) u_j(y), & \text{if }y\leq x, \\
               \sum_{i,j=1}^2 m^-_{ij}(z) u_i(x) u_j(y), & \text{if }y\geq x.
             \end{cases}
 \end{align}
\end{theorem}

\begin{proof}
Let $u_1$, $u_2$ be two linearly independent solutions of $(\tau-z)u=0$ with $W(u_1,u_2)=1$. 
 If $g\in L^2_c((a,b);\varrho)$, then $(R_z g,g)\in (S-z)$, hence there is some $f\in\Deftau$ satisfying 
 the boundary conditions with $f=R_z g$ and $(\tau-z) f=g$. 
 From Proposition~\ref{prop:repsol} we get for suitable constants $c_1$, $c_2\in\C$
 \begin{align}\label{eqn::resoleqn}
  f(x)=u_1(x)\left(c_1+\int_a^x{u_2 g\, d\varrho}\right)+u_2(x)\left(c_2-\int_a^x{u_1 g\, d\varrho}\right)
 \end{align} 
 for each $x\in(a,b)$. 
 Furthermore, since $f$ satisfies the boundary conditions 
 \begin{align*}
  B_a\begin{pmatrix} \BCa^1(f) \\ \BCa^2(f) \end{pmatrix} = B_b\begin{pmatrix} \BCb^1(f) \\ \BCb^2(f) \end{pmatrix}
 \end{align*} for some suitable matrices $B_a$, $B_b\in\C^{2\times 2}$ as in Theorem~\ref{thm:LCLCBC}.
 Now because $g$ has compact support, we have
 \begin{align*}
  \begin{pmatrix} \BCa^1(f) \\ \BCa^2(f) \end{pmatrix} & = \begin{pmatrix} c_1\BCa^1(u_1)+c_2\BCa^1(u_2) \\ 
     c_1\BCa^2(u_1)+c_2\BCa^2(u_2) \end{pmatrix} = \begin{pmatrix} \BCa^1(u_1) & \BCa^1(u_2) \\ 
  \BCa^2(u_1) & \BCa^2(u_2) \end{pmatrix}\begin{pmatrix} c_1 \\ c_2 \end{pmatrix} \\
 & = M_\alpha \begin{pmatrix} c_1 \\ c_2 \end{pmatrix}
 \end{align*} 
 as well as 
 \begin{align*}\begin{pmatrix} \BCb^1(f) \\ \BCb^2(f) \end{pmatrix} & = 
   \begin{pmatrix} \left(c_1+\int_a^b u_2 g d\varrho\right)\BCb^1(u_1) \\ 
     \left(c_1+\int_a^b u_2 g d\varrho\right)\BCb^2(u_1) \end{pmatrix} + 
  \begin{pmatrix} \left(c_2-\int_a^b u_1 g d\varrho\right)\BCb^1(u_2) \\ 
        \left(c_2-\int_a^b u_1 g d\varrho\right)\BCb^2(u_2) \end{pmatrix} \\
 & = \begin{pmatrix} \BCb^1(u_1) & \BCb^1(u_2) \\ \BCb^2(u_1) & \BCb^2(u_2) \end{pmatrix}
   \begin{pmatrix} c_1+\int_a^b{u_2 g\, d\varrho} \\ c_2-\int_a^b{u_1 g\, d\varrho} \end{pmatrix} \\
 & = M_\beta \begin{pmatrix} c_1 \\ c_2 \end{pmatrix} + M_\beta \begin{pmatrix} \int_a^b{u_2 g\, d\varrho} \\ 
  -\int_a^b{u_1 g\, d\varrho}\end{pmatrix}.\end{align*}
 Hence we have 
 \begin{align*}
   \left(B_aM_\alpha -B_b M_\beta\right)\begin{pmatrix} c_1 \\ c_2 \end{pmatrix} = B_b M_\beta \begin{pmatrix} \int_a^b{u_2 g\, d\varrho} \\ -\int_a^b{u_1 g\, d\varrho} \end{pmatrix}.
 \end{align*}
 Now if $B_aM_\alpha-B_bM_\beta$ was not invertible, we would have 
 \begin{align*}
  \begin{pmatrix} d_1 \\ d_2 \end{pmatrix}\in\C^2\setminus\left\lbrace\begin{pmatrix}0 \\0 \end{pmatrix}\right\rbrace \quad\text{with}\quad B_aM_\alpha\begin{pmatrix} d_1 \\ d_2 \end{pmatrix}=B_bM_\beta\begin{pmatrix} d_1 \\ d_2 \end{pmatrix}.
 \end{align*} 
 Then the function $d_1u_1+d_2u_2$ would be a solution of $(\tau-z)u=0$ satisfying the boundary conditions
  of $S$, hence an eigenvector with eigenvalue $z$. But since this would contradict $z\in\rho(S)$, 
  $B_aM_\alpha-B_bM_\beta$ has to be invertible. 
  Now because of 
  \begin{align*}
   \begin{pmatrix} c_1 \\ c_2 \end{pmatrix} = \left(B_aM_\alpha -B_b M_\beta\right)^{-1} B_b M_\beta \begin{pmatrix} \int_a^b{u_2 g\, d\varrho} \\ -\int_a^b{u_1 g\, d\varrho} \end{pmatrix},
  \end{align*}
  the constants $c_1$ and $c_2$ may be written as linear combinations of 
 \[
 \int_a^b{u_2 g\, d\varrho}\quad\text{and}\quad \int_a^b{u_1 g\, d\varrho},
 \] 
 where the coefficients are independent of $g$. 
 Now using equation~\eqref{eqn::resoleqn} one sees that $f$ has an integral-representation
 with a function $G_z$ as claimed. Moreover, the function $G_z$ is square-integrable, since the solutions
 $u_1$ and $u_2$ lie in $\Lr$ by assumption. 
 Finally, since the operator $K_z$ 
 \[
 K_z g(x) = \int_a^b G_z(x,y)g(y)d\varrho(y),\quad x\in(a,b),~g\in\Lr
 \]
 on $\Lr$, as well as the resolvent $R_z$ are bounded, the claim follows since they coincide on a dense subspace.
\end{proof}

As in the classical case, the compactness of the resolvent implies discreteness of the spectrum.

\begin{corollary}\label{corSpecRDis}
 Suppose $\tau$ is in the l.c.~case at both endpoints and $S$ is a self-adjoint restriction of $\Tmax$. 
 Then the relation $S$ has purely discrete spectrum, i.e., $\sigma(S)=\sigdis(S)$ with
 \begin{align*}
  \mathop{\sum_{\lambda\in\sigma(S)}}_{\lambda\not=0} \frac{1}{\lambda^2} < \infty \quad\text{and}\quad 
    \dim\ker(S-\lambda)\leq 2, \quad \lambda\in\sigma(S).
 \end{align*}
\end{corollary}

\begin{proof}
 Since the resolvent is compact, Theorem~\ref{AppLRthmSMT} shows that the spectrum of $S$ consists of isolated eigenvalues.
  Furthermore, the sum converges since the resolvent is Hilbert--Schmidt.
  Finally, their multiplicity is at most two because of~\eqref{eqnSLPdimkerTloc}.
\end{proof}

If $S$ is a self-adjoint restriction of $\Tmax$ with separated boundary conditions or if not both endpoints 
 are in the l.c.~case, the resolvent has a simpler form.

\begin{theorem}\label{thm:ressep}
 Suppose $S$ is a self-adjoint restriction of $\Tmax$ with separated boundary conditions (if $\tau$ is in the
  l.c.\ case at both endpoints) and $z\in\rho(S)$.
 Furthermore, let $u_a$ and $u_b$ be non-trivial solutions of $(\tau-z)u=0$, such that
 \begin{align*}
  u_a\, \begin{cases}
         \text{satisfies the boundary condition at }a\text{ if }\tau\text{ is in the l.c.~case at }a, \\
         \text{lies in }\Lr\text{ near }a\text{ if }\tau\text{ is in the l.p.~case at }a,
      \end{cases}
 \end{align*}
 and
 \begin{align*}
  u_b\, \begin{cases}
         \text{satisfies the boundary condition at }b\text{ if }\tau\text{ is in the l.c.~case at }b, \\
         \text{lies in }\Lr\text{ near }b\text{ if }\tau\text{ is in the l.p.~case at }b.
      \end{cases}
 \end{align*}
 Then the resolvent $R_z$ is given by
 \begin{align}\label{eq:ressbc}
  R_z g(x) 
           & = \int_{a}^b G_z(x,y) g(y) d\varrho(y), \quad x\in(a,b),~g\in\Lr,
 \end{align}
 where
 \begin{align}
  G_z(x,y) = \frac{1}{W(u_b,u_a)}\begin{cases}
               u_a(y) u_b(x), & \text{if }y\leq x, \\
               u_a(x) u_b(y), & \text{if }y \geq x. \\
             \end{cases}
 \end{align}
\end{theorem}

\begin{proof}
The functions $u_a$, $u_b$ are linearly independent, since otherwise they were an eigenvector of $S$ 
 corresponding to the eigenvalue $z$. Hence they form a fundamental system of $(\tau-z)u=0$.
 Now for each $g\in\Lr$ we define a function $f_g$ by 
 \[
 f_g(x)=W(u_b,u_a)^{-1}\left(u_b(x)\int_a^xu_a g\, d\varrho+u_a(x)\int_x^bu_b g\, d\varrho\right), \quad x\in(a,b).
 \] 
 If $g\in L^2_c((a,b);\varrho)$, then $f_g$ is a solution of $(\tau-z)f=g$ by Proposition~\ref{prop:repsol}.
 Moreover, $f_g$ is a scalar multiple of $u_a$ near $a$ and a scalar multiple of $u_b$ near $b$. 
 Hence the function $f_g$ satisfies the boundary conditions of $S$
  and therefore $(f_g,\tau f_g - z f_g)=(f_g, g)\in (S-z)$, i.e., $R_z g=f_g$.
 Now if $g\in\Lr$ is arbitrary and $g_n\in L^2_c((a,b);\varrho)$ is a sequence with $g_n\rightarrow g$ as $n\rightarrow\infty$,
  we have, since the resolvent is bounded $R_z g_n\rightarrow R_z g$. 
  Furthermore, $f_{g_n}$ converges pointwise to $f_g$ and hence $R_z g=f_g$. 
\end{proof}

If $\tau$ is in the l.p.~case at some endpoint, then Corollary~\ref{cor:regtypeuniqsol} shows that
 there is always a unique non-trivial solution of $(\tau-z)u=0$ (up to scalar multiples), 
 lying in $\Lr$ near this endpoint. 
 Also if $\tau$ is in the l.c.~case at some endpoint, there exists a unique non-trivial solution of $(\tau-z)u=0$ (up to scalar multiples), satisfying the boundary condition at this endpoint. 
 Hence functions $u_a$ and $u_b$, as in Theorem~\ref{thm:ressep} always exist.

\begin{corollary}\label{corSpecRSimple}
 If $S$ is a self-adjoint restriction of $\Tmax$ with separated boundary conditions (if $\tau$ is in the
  l.c.~at both endpoints), then all eigenvalues of $S$ are simple.
\end{corollary}

\begin{proof}
 Suppose $\lambda\in\R$ is an eigenvalue and $u_j\in S$ with $\tau u_j=\lambda u_j$ for $j=1,2$, i.e., they are solutions of $(\tau-\lambda)u=0$.
  If $\tau$ is in the l.p.~case at some endpoint, then clearly the Wronskian $W(u_1,u_2)$ vanishes.
  Otherwise, since both functions satisfy the same boundary conditions this follows using the Pl\"{u}cker identity.
\end{proof}

According to Theorem~\ref{AppLRthmSAEessspec} the essential spectrum of self-adjoint restrictions is independent of the boundary conditions, i.e., all self-adjoint restrictions of $\Tmax$ have the same essential spectrum.
We conclude this section by proving that the essential spectrum of the self-adjoint restrictions of $\Tmax$ is determined by the behavior of the
 coefficients in some arbitrarily small neighborhood of the endpoints. 
In order to state this result, which is originally due to H.\ Weyl, we need some notation. Fix some $c\in(a,b)$ and denote by $\tau|_{(a,c)}$ (respectively by $\tau|_{[c,b)}$) the differential 
 expression on $(a,b)$ corresponding to the coefficients $\varsigma$, $\chi$ and $\varrho|_{(a,c)}$ (respectively $\varrho|_{[c,b)}$). Furthermore, let 
 $S_{(a,c)}$ (respectively $S_{[c,b)}$) be some self-adjoint realizations of $\tau|_{(a,c)}$ (respectively of $\tau|_{[c,b)}$).

\begin{theorem}
 For each $c\in(a,b)$ we have
 \begin{align}
  \sigess\left(S\right) = \sigess\left(S_{(a,c)}\right) \cup \sigess\left(S_{[c,b)}\right). 
 \end{align}
\end{theorem}

\begin{proof}
 If one identifies $\Lr$ with the orthogonal sum 
 \begin{align*}
  \Lr = L^2((a,b);\varrho|_{(a,c)}) \oplus L^2((a,b);\varrho|_{[c,b)}),
 \end{align*}
 then the linear relation
 \begin{align*}
  S_c = S_{(a,c)} \oplus S_{[c,b)}
 \end{align*}
 is self-adjoint in $\Lr$.
 Now since $S$ and $S_c$ both are finite dimensional extensions of the symmetric linear relation
 \begin{align*}
  T_c = \lbrace f\in\Tmin \,\big|\, f(c) = f^\qd(c) = 0 \rbrace,
 \end{align*}
 an application of Theorem~\ref{AppLRthmSAEessspec} and Theorem~\ref{AppLRthmOSessspec} yields the claim.
\end{proof}

As an immediate consequence one sees that the essential spectrum only depends on the coefficients in some neighborhood of the endpoints.

\begin{corollary}
 For each $\alpha$, $\beta\in(a,b)$ we have
 \begin{align}
  \sigess\left(S\right) = \sigess\left(S_{(a,\alpha)}\right) \cup \sigess\left(S_{[\beta,b)}\right). 
 \end{align}
\end{corollary}

\section{Singular Weyl--Titchmarsh--Kodaira functions}\label{secweyltitchm}

In this section let $S$ be a self-adjoint restriction of $\Tmax$ with separated boundary conditions (if $\tau$ is
 in the l.c.~case at both endpoints). Our aim is to define a singular Weyl--Titchmarsh--Kodaira function as originally
 advocated by Kodaira \cite{kod} and Kac \cite{kac}. We follow the recent approach by Kostenko, Sakhnovich and Teschl  \cite{kst} for Schr\"{o}dinger operators. To this end we need a real entire fundamental system $\theta_z$, $\phi_z$, $z\in\C$
 of $(\tau-z)u=0$ with $W(\theta_z,\phi_z)=1$, such that $\phi_z$ lies in $S$ near $a$, i.e., $\phi_z$ lies in $\Lr$ near $a$ and satisfies 
 the boundary condition at $a$ if $\tau$ is in the l.c.~case there.

\begin{hypothesis}\label{hypREFS}
 There is a real entire fundamental system $\theta_z$, $\phi_z$, $z\in\C$ of $(\tau-z)u=0$ with $W(\theta_z,\phi_z)=1$, 
  such that $\phi_z$ lies in $S$ near $a$.
\end{hypothesis}

Under the assumption of Hypothesis~\ref{hypREFS} we
may define a complex-valued function $m$ on $\rho(S)$ by requiring that the solutions
  \begin{align}
   \psi_z = \theta_z + m(z)\phi_z, \quad z\in\rho(S)
   \end{align} 
 lie in $S$ near $b$, i.e., they lie in $\Lr$ near $b$ and satisfy the boundary condition at $b$, if $\tau$ is 
 in the l.c.~case at $b$.
 This function $m$ is well-defined (use Corollary~\ref{cor:regtypeuniqsol} if $\tau$ is in the l.p.~case at $b$)
  and called the singular Weyl--Titchmarsh--Kodaira function of $S$. 
The solutions $\psi_z$, $z\in\rho(S)$ are referred to as the Weyl solutions of $S$. 

\begin{theorem}\label{thmweyltitchManal}
 If Hypothesis~\ref{hypREFS} holds, then the corresponding singular Weyl--Titchmarsh--Kodaira function $m$ is analytic and furthermore satisfies 
 \begin{align}\label{eqprop::mpsi} 
  m(z)=m(z^\ast)^\ast, \quad z\in\rho(S).
 \end{align}
\end{theorem}

\begin{proof}
 Let $c$, $d\in(a,b)$ with $c<d$. From Theorem~\ref{thm:ressep} and the equation
 \[W(\psi_z,\phi_z)=W(\theta_z,\phi_z)+m(z)W(\phi_z,\phi_z)=1,\quad z\in\rho(S),\] we get for each 
 $z\in\rho(S)$ and $x\in[c,d)$ 
 \begin{align*}
  R_z\indik_{[c,d)}(x) & =\psi_z(x)\int_{c}^x \phi_z\, d\varrho + \phi_z(x)\int_x^d \psi_z\, d\varrho \\
 & = (\theta_z(x)+m(z)\phi_z(x))\int_c^x \phi_z\, d\varrho + \phi_z(x)\int_x^d \theta_z+m(z)\phi_z\, d\varrho \\
 & = m(z)\phi_z(x)\int_c^d \phi_z(y)d\varrho(y) + \int_c^d \tilde{G}_z(x,y)d\varrho(y), \end{align*} 
 where \[\tilde{G}_z(x,y)=\begin{cases} \phi_z(y)\theta_z(x), & \text{if }y\leq x, \\
         \phi_z(x)\theta_z(y), & \text{if }y \geq x,\end{cases}\]
 and hence
 \begin{align*} \spr{R_z\indik_{[c,d)}}{\indik_{[c,d)}} & = m(z)\left(\int_c^d \phi_z(y)d\varrho(y)\right)^2 + \int_c^d\int_c^d \tilde{G}_z(x,y)d\varrho(y)\,d\varrho(x). \end{align*}
 The left-hand side of this equation is analytic on $\rho(S)$ since the resolvent is.
 Furthermore, the integrals are analytic on $\rho(S)$ as well, since the integrands are analytic and locally 
  bounded by Theorem~\ref{thmMSLEanaly}.
 Hence $m$ is analytic provided that for each $z_0\in\rho(S)$, there are some $c$, $d\in(a,b)$ such that 
 \[\int_c^d \phi_{z_0}(y) d\varrho(y)\not=0.\] 
 But this is true since otherwise $\phi_{z_0}$ would vanish almost everywhere with respect to $\varrho$. 
 Moreover, equation~\eqref{eqprop::mpsi} is valid since the functions
 \[\theta_{z^\ast}+m(z)^\ast\phi_{z^\ast}=\left(\theta_z+m(z)\phi_z\right)^\ast, \quad z\in\rho(S) \] 
 lie in $S$ near $b$ 
  by Lemma~\ref{lem:WronskPropSR}.
\end{proof}

As an immediate consequence of Theorem~\ref{thmweyltitchManal} we see that the functions $\psi_z(x)$ 
 and $\psi_z^\qd(x)$ are analytic in $z\in\rho(S)$ for each $x\in(a,b)$.

\begin{remark}\label{remWTtilde}
 Note that a fundamental system as in Hypothesis~\ref{hypREFS} is not unique. In fact, any other such system is given by
 \begin{align}
  \tilde{\theta}_z = e^{-g(z)} \theta_z - f(z)\phi_z \quad\text{and}\quad \tilde{\phi}_z = e^{g(z)} \phi_z, \quad z\in\C
 \end{align}
 for some real entire functions $f$, $\E^g$. Moreover, the corresponding singular Weyl--Titchmarsh--Kodaira functions are related via
 \begin{align}
  \tilde{m}(z) = e^{-2g(z)} m(z) + e^{-g(z)}f(z), \quad z\in\rho(S).
 \end{align}
 In particular, the maximal domain of holomorphy or the structure of poles and singularities of $m$ do not change under such a transformation.
\end{remark}

We continue with the construction of a real entire fundamental system in the case when $\tau$ is in the l.c.~case at $a$.

\begin{theorem}\label{thmweyltitchfundsys}
 Suppose $\tau$ is in the l.c.~case at $a$. Then there exists a real entire fundamental system 
 $\theta_z$, $\phi_z$, $z\in\C$ of $(\tau-z)u=0$ with $W(\theta_z,\phi_z)=1$,
 such that $\phi_z$ lies in $S$ near $a$ and for each $z_1$, $z_2\in\C$ we have 
 \begin{align}
  W(\theta_{z_1},\phi_{z_2})(a) = 1 \quad\text{and}\quad W(\theta_{z_1},\theta_{z_2})(a) = W(\phi_{z_1},\phi_{z_2})(a)=0.
 \end{align}
\end{theorem}

\begin{proof}
 Let $\theta$, $\phi$ be a real fundamental system of $\tau u=0$ with $W(\theta,\phi)=1$ such that $\phi$ lies in $S$ near $a$.
 Now fix some $c\in(a,b)$ and for each $z\in\C$ let $u_{z,1}$, $u_{z,2}$ be the fundamental system of
 \begin{align*}
  (\tau-z)u=0 \quad\text{with}\quad u_{z,1}(c)=u_{z,2}^\qd(c) = 1 \quad\text{and}\quad u_{z,1}^\qd(c)=u_{z,2}(c)=0.
 \end{align*}
 Then by Theorem~\ref{thm:exisuniq} we have $u_{z^\ast,j}=u_{z,j}^\ast$ for $j=1$, $2$. If we introduce
 \begin{align*}
  \theta_z(x) & = W(u_{z,1},\theta)(a) u_{z,2}(x) - W(u_{z,2},\theta)(a) u_{z,1}(x), & x\in(a,b), \\
  \phi_z(x)   & = W(u_{z,1},\phi)(a) u_{z,2}(x)   - W(u_{z,2},\phi)(a) u_{z,1}(x),   & x\in(a,b),
 \end{align*}
 then the functions $\phi_z$ lie in $S$ near $a$ since
 \begin{align*}
  W(\phi_z,\phi)(a) & = W(u_{z,1},\phi)(a)W(u_{z,2},\phi)(a) - W(u_{z,2},\phi)(a) W(u_{z_1},\phi)(a) = 0.
 \end{align*}
 Furthermore, a direct calculation shows that $\theta_{z^\ast}=\theta_z^\ast$ and $\phi_{z^\ast}=\phi_z^\ast$.
 The remaining equalities follow using the Pl\"{u}cker identity several times.
 It remains to prove that the functions $W(u_{z,1},\theta)(a)$, $W(u_{z,2},\theta)(a)$,
  $W(u_{z,1},\phi)(a)$ and $W(u_{z,2},\phi)(a)$ are entire in $z$.
 Indeed, we get from the Lagrange identity
 \begin{align*}
  W(u_{z,1},\theta)(a) = W(u_{z,1},\theta)(c) - z \lim_{x\rightarrow a} \int_x^c \theta(t) u_{z,1}(t) d\varrho(t), \quad z\in\C.
 \end{align*}
 Now the integral on the right-hand side is analytic by Theorem~\ref{thmMSLEanaly} and in order to prove that 
  the limit is also analytic we need to show that the integral is bounded as $x\rightarrow a$, locally uniformly
  in $z$. But the proof of Lemma~\ref{lemweylaltLC} shows that for each $z_0\in\C$ we have
 \begin{align*}
  \left| \int_x^c \theta(t) u_{z,1}(t) d\varrho(t) \right|^2 & \leq K \int_a^c \left|\theta\right|^2 d\varrho 
                   \int_a^c \left(\left|u_{z_0,1}\right| +\left|u_{z_0,2}\right|\right)^2 d\varrho
                    \end{align*}
 for some constant $K\in\R$ and all $z$ in some neighborhood of $z_0$. Analyticity of the other functions
  is proved similarly.
\end{proof}

If $\tau$ is even regular at $a$, then one may take $\theta_z$, $\phi_z$, $z\in\C$ to be the solutions of $(\tau-z)u=0$ 
 with the initial values 
\[ \theta_z(a)=\phi_z^\qd(a)=\cos\varphi_\alpha \quad\text{and}\quad -\theta_z^\qd(a)=\phi_z(a)=\sin\varphi_\alpha \] 
for some suitable $\varphi_\alpha\in[0,\pi)$.
Furthermore, in the case when $\varrho$ has no weight near $a$, one may take for $\theta_z$, $\phi_z$, $z\in\C$ the 
 solutions of $(\tau-z)u=0$ with the initial values 
\[ \theta_z(\alpha_\varrho-)=\phi_z^\qd(\alpha_\varrho-)=\cos\varphi_\alpha \quad\text{and}\quad 
      -\theta_z^\qd(\alpha_\varrho-)=\phi_z(\alpha_\varrho-)=\sin\varphi_\alpha\]
 for some $\varphi_\alpha\in[0,\pi)$.

\begin{corollary}\label{spek::manal}
 Suppose $\tau$ is in the l.c.~case at $a$ and $\theta_z$, $\phi_z$, $z\in\C$ is a real entire fundamental system of $(\tau-z)u=0$
  as in Theorem~\ref{thmweyltitchfundsys}.
 Then the singular Weyl--Titchmarsh--Kodaira function $m$ is a Herglotz--Nevanlinna function.
\end{corollary} 

\begin{proof}
In order to prove the claim, we show that
 \begin{align}\label{eqnWTHerglotz}
  0<\|\psi_z\|^2=\frac{\im(m(z))}{\im(z)}, \quad z\in\C\backslash\R.
 \end{align}
Indeed if $z_1$, $z_2\in\rho(S)$, then 
\begin{align*} 
 W(\psi_{z_1},\psi_{z_2})(a) & = W(\theta_{z_1},\theta_{z_2})(a)+m(z_2)W(\theta_{z_1},\phi_{z_2})(a) \\ & \hspace{1.2cm} + m(z_1)W(\phi_{z_1},\theta_{z_2})(a)+m(z_1)m(z_2)W(\phi_{z_1},\phi_{z_2})(a) \\ & = m(z_2)-m(z_1).
\end{align*}
If $\tau$ is in the l.p.~case at $b$, then furthermore we have 
\begin{align*}\label{eqnSingWTPsiB} W(\psi_{z_1},\psi_{z_2})(b)=0,\end{align*} since clearly $\psi_{z_1}$, $\psi_{z_2}\in\Tmax$. 
This also holds if $\tau$ is in the l.c.~case at $b$, since then $\psi_{z_1}$ and $\psi_{z_2}$ satisfy the same boundary condition at $b$.
Now using the Lagrange identity yields 
\begin{align*}
(z_1-z_2)\int_a^b \psi_{z_1}(t)\psi_{z_2}(t)d\varrho(t) & =W(\psi_{z_1},\psi_{z_2})(b)-W(\psi_{z_1},\psi_{z_2})(a) \\ & =m(z_1)-m(z_2).\end{align*} 
In particular, for $z\in\C\backslash\R$, using $m(z^\ast)=m(z)^\ast$ as well as 
\[\psi_{z^\ast}=\theta_{z^\ast}+m(z^\ast)\phi_{z^\ast}=\psi_z^\ast,\] we get 
\begin{align*} 
||\psi_z||^2 & = \int_a^b \psi_z(t) \psi_{z^\ast}(t)d\varrho(t)=\frac{m(z)-m(z^\ast)}{z-z^\ast}=\frac{\im(m(z))}{\im(z)}.\end{align*} 
Since $\psi_z$ is a non-trivial solution, we furthermore have $0<||\psi_z||^2$.
\end{proof}

We conclude this section with a necessary and sufficient condition for Hypothesis~\ref{hypREFS} to hold.
 To this end recall that for each $c\in(a,b)$, $S_{(a,c)}$ is some self-adjoint operator associated with the restricted differential expression $\tau|_{(a,c)}$.
 The proofs are the same as those for Schr\"{o}dinger operators given in~\cite[Lemma~2.2 and Lemma~2.4]{kst}.

\begin{theorem}
 The following properties are equivalent:
\begin{enumerate}
 \item Hypothesis~\ref{hypREFS}.
 \item The spectrum of $S_{(a,c)}$ is purely discrete for some $c\in(a,b)$.
 \item There is a real entire solution $\phi_z$, $z\in\C$ of $(\tau-z)u=0$ which lies in $S$ near $a$.
\end{enumerate}
\end{theorem}

\section{Spectral transformation}\label{secST}

In this section let $S$ again be a self-adjoint restriction of $\Tmax$ with separated boundary conditions as in the
 preceding section. Furthermore, we assume that there is a real entire fundamental system $\theta_z$, $\phi_z$, $z\in\C$ of the differential equation $(\tau-z)u=0$
 with $W(\theta_z,\phi_z)=1$ such that $\phi_z$ lies in $S$ near $a$. With $m$ we denote the 
 corresponding singular Weyl--Titchmarsh--Kodaira function and with $\psi_z$, $z\in\rho(S)$ the Weyl solutions of $S$.

Recall that by Lemma~\ref{lemspectransEfg} for all functions $f$, $g\in\Lr$ there is a unique complex measure $E_{f,g}$ on $\R$ such that
 \begin{align*}
  \spr{R_z f}{g} = \int_\R \frac{1}{\lambda-z} dE_{f,g}(\lambda), \quad z\in\rho(S).
 \end{align*}
 Indeed, these measures are obtained by applying a variant of the spectral theorem to the operator part 
 \begin{align*}
  S_\D = S \cap \left(\D\times\D\right), \qquad \D= \overline{\dom{S}} = \mul{S}^\bot
 \end{align*}
 of $S$ (see Lemma~\ref{lemspectransEfg} in Appendix~\ref{appLR}).
 
In order to obtain a spectral transformation we define for each function $f\in L^2_c((a,b);\varrho)$ the transform of $f$ as
\begin{align}\label{eqnSThatf}
 \hat{f}(z) = \int_a^b \phi_z(x) f(x) d\varrho(x), \quad z\in\C.
\end{align}
Next we can use this to associate a measure with $m$ by virtue of
the Stieltjes--Liv\v{s}i\'{c} inversion formula, following literally the proof of \cite[Lemma~3]{kst}.

\begin{lemma}\label{lemspectransEfgmu}
There is a unique Borel measure $\mu$ on $\R$ defined via
\be\label{defrho}
 \mu((\lambda_1,\lambda_2]) = \lim_{\delta\downarrow 0}\,\lim_{\varepsilon\downarrow 0} \frac{1}{\pi} 
                      \int_{\lambda_1+\delta}^{\lambda_2+\delta} \im(m(\lambda+\I\varepsilon))d\lambda
\ee
 for each $\lambda_1$, $\lambda_2\in\R$ with $\lambda_1<\lambda_2$, such that for every $f$, $g\in L^2_c((a,b);\varrho)$
\be\label{rhomuf}
E_{f,g} = \hat{f}\, \hat{g}^\ast \mu
\ee
and, in particular, 
\begin{align}
 \spr{R_z f}{g} = \int_\R \frac{\hat{f}(\lambda) \hat{g}(\lambda)^\ast}{\lambda-z} d\mu(\lambda), \quad z\in\rho(S).
\end{align}
\end{lemma}

In particular, the preceding lemma shows that the mapping $f\mapsto \hat{f}$ is an isometry from $L^2_c((a,b);\varrho)\cap\D$ into $\Lrmu$.
In fact, for each function $f\in L^2_c((a,b);\varrho)\cap\D$ we have 
\begin{align*}
 \| \hat{f}\|^2_\mu = \int_\R \hat{f}(\lambda) \hat{f}(\lambda)^\ast d\mu(\lambda) = \int_\R dE_{f,f} =  \| f\|^2.
\end{align*}
Hence we may uniquely extend this mapping to an isometric linear operator $\mathcal{F}$ on the Hilbert space $\D$ into $\Lrmu$ by
\begin{align*}
 \mathcal{F} f(\lambda) = \lim_{\alpha\rightarrow a}\, \lim_{\beta\rightarrow b} \int_\alpha^\beta \phi_\lambda(x) f(x) d\varrho(x), \quad \lambda\in\R,~f\in\D,
\end{align*}
where the limit on the right-hand side is a limit in the Hilbert space $\Lrmu$.
Using this operator $\mathcal{F}$, it is quite easy to extend the result of Lemma~\ref{lemspectransEfgmu} to functions $f$, $g\in\D$.
Indeed, one gets that $E_{f,g} =  \mathcal{F}f\, \mathcal{F}g^\ast \mu$, i.e.,
\begin{align*}
 \spr{R_z f}{g} = \int_\R \frac{\mathcal{F}f(\lambda) \mathcal{F}g(\lambda)^\ast}{\lambda-z} d\mu(\lambda), \quad z\in\rho(S).
\end{align*}
We will see below that $\mathcal{F}$ is not only isometric, but also onto. 
In order to compute the inverse and the adjoint of $\mathcal{F}$, we introduce for each function $g\in L^2_c(\R;\mu)$ the transform
\begin{align*}
 \check{g}(x) = \int_\R \phi_\lambda(x) g(\lambda) d\mu(\lambda), \quad x\in(a,b).
\end{align*}
For arbitrary $\alpha$, $\beta\in(a,b)$ with $\alpha<\beta$ we have
\begin{align*}
 \int_\alpha^\beta \left|\check{g}(x)\right|^2 d\varrho(x) & = \int_\alpha^\beta \check{g}(x) \int_\R \phi_\lambda(x) g(\lambda)^\ast d\mu(\lambda)\, d\varrho(x) \\
          & = \int_\R g(\lambda)^\ast  \int_\alpha^\beta \phi_\lambda(x) \check{g}(x) d\varrho(x)\, d\mu(\lambda)  \\
          & \leq \left\| g\right\|_\mu \left\| \mathcal{F}\left(\indik_{[\alpha,\beta)}\check{g}\right)\right\|_\mu  \\
          & \leq \left\| g\right\|_\mu \sqrt{\int_\alpha^\beta \left|\check{g}(x)\right|^2 d\varrho(x)}.
\end{align*}
Hence $\check{g}$ lies in $\Lr$ with $\|\check{g}\|\leq\|g\|_\mu$ and we may uniquely extend this mapping to a bounded linear operator $\mathcal{G}$ on $\Lrmu$ into $\D$.

If $F$ is a Borel measurable function on $\R$, then we denote with $\M_F$ the maximally defined operator of 
 multiplication with $F$ in $\Lrmu$.

\begin{lemma}\label{lemspectransfourtrans}
 The isometry $\mathcal{F}$ is onto with inverse $\mathcal{F}^{-1}=\mathcal{G}$ and adjoint 
 \begin{align}
  \mathcal{F}^\ast = 
 \lbrace (g,f)\in\Lrmu\times\Lr \,|\, \mathcal{G} g - f\in\mul{S} \rbrace.
 \end{align}
\end{lemma}

\begin{proof}
 In order to prove $\ran(\mathcal{G})\subseteq\D$, let $g\in L^2_c(\R;\mu)$.
  If $\indik_{\lbrace\alpha_\varrho\rbrace}\in\mul{S}$, then the solutions $\phi_z$, $z\in\C$ vanish in $\alpha_\varrho$, hence also
 \begin{align*}
  \check{g}(\alpha_\varrho) =  \int_\R \phi_\lambda(\alpha_\varrho) g(\lambda) d\mu(\lambda) = 0.
 \end{align*}
 Furthermore, if $\indik_{\lbrace\beta_\varrho\rbrace}\in\mul{S}$, then the spectrum of $S$ is discrete and
  the solutions $\phi_\lambda$, $\lambda\in\sigma(S)$ vanish in $\beta_\varrho$. Now since $\mu$ is supported
  on $\sigma(S)$, we also have
 \begin{align*}
  \check{g}(\beta_\varrho) = \int_{\sigma(S)} \phi_\lambda(\beta_\varrho) g(\lambda)d\mu(\lambda) = 0.
 \end{align*}
 From this one sees that $\check{g}\in\mul{S}^\bot = \D$, i.e., $\ran(\mathcal{G})\subseteq\D$.

 Next we prove $\mathcal{G}\mathcal{F}f=f$ for each $f\in\D$.
 Indeed, if $f$, $g\in L^2_c((a,b);\varrho)\cap\D$, then we have 
 \begin{align*}
  \spr{f}{g} & = \int_\R dE_{f,g} = \int_\R \hat{f}(\lambda) \hat{g}(\lambda)^\ast d\mu(\lambda) \\
        & = \lim_{n\rightarrow\infty} \int_{(-n,n]} \hat{f}(\lambda) \int_a^b \phi_\lambda(x) g(x)^\ast d\varrho(x)\, d\mu(\lambda) \\
        & = \lim_{n\rightarrow\infty} \int_a^b g(x)^\ast \int_{(-n,n]} \hat{f}(\lambda) \phi_\lambda(x) d\mu(\lambda)\, d\varrho(x) \\
        & = \lim_{n\rightarrow\infty} \spr{\mathcal{G}\M_{\indik_{(-n,n]}}\mathcal{F}f}{g} = \spr{\mathcal{GF}f}{g}.
\end{align*}
 Now since $\ran(\mathcal{G})\subseteq\D$ and $L^2_c((a,b);\varrho)\cap\D$ is
  dense in $\D$ we infer that $\mathcal{GF}f=f$ for all $f\in\D$.
In order to prove that $\mathcal{G}$ is the inverse of $\mathcal{F}$, it remains to show that $\mathcal{F}$ 
 is onto, i.e., $\ran(\mathcal{F})=\Lrmu$.
 Therefore, pick some $f$, $g\in\D$ and let $F$, $G$ be bounded measurable functions on $\R$.
 Since $E_{f,g}$ is the spectral measure of the operator part $S_\D$ of $S$ (see the proof of Lemma~\ref{lemspectransEfg}) we get
 \begin{align*}
  \spr{\M_G \mathcal{F} F(S_{\D})f}{\mathcal{F}g}_\mu = \spr{G(S_{\D})F(S_{\D})f}{g}
     = \spr{\M_G \M_F \mathcal{F}f}{\mathcal{F}g}_\mu.
 \end{align*}
 Now if we set $h=F(S_{\D})f$, we get from this last equation
 \begin{align*}
  \int_\R G(\lambda) \mathcal{F}g(\lambda)^\ast \left( \mathcal{F}h(\lambda) - F(\lambda)\mathcal{F}f(\lambda)\right)d\mu(\lambda) = 0. 
 \end{align*}
 Since this holds for each bounded measurable function $G$, we infer 
 \begin{align*}
 \mathcal{F}g(\lambda)^\ast \left( \mathcal{F}h(\lambda) - F(\lambda)\mathcal{F}f(\lambda)\right)=0
 \end{align*}
  for almost all $\lambda\in\R$ with respect to $\mu$. Furthermore, for each $\lambda_0\in\R$ we can find
  a $g\in L^2_c((a,b);\varrho)\cap\D$ such that $\hat{g}\not=0$ in a vicinity of $\lambda_0$. 
 Hence we even have $\mathcal{F}h = F\mathcal{F}f$ 
  almost everywhere with respect to $\mu$. But this shows that $\ran(\mathcal{F})$ contains all characteristic 
  functions of intervals. Indeed, let $\lambda_0\in\R$ and choose $f\in L^2_c((a,b);\varrho)\cap\D$ such that 
  $\hat{f}\not=0$ in a vicinity of $\lambda_0$. Then for each interval $J$, whose closure is contained in 
  this vicinity one may choose 
  \begin{align*}
   F(\lambda) = \begin{cases} \hat{f}(\lambda)^{-1}, & \text{if }\lambda\in J, \\
                              0, & \text{if }\lambda\in\R\backslash J,
                \end{cases}
  \end{align*}
and gets $\indik_J = \mathcal{F} h\in\ran(\mathcal{F})$.
 Thus we have obtained $\ran(\mathcal{F})=\Lrmu$.
 Finally the fact that the adjoint is given as in the claim follows from the equivalence
 \begin{align*}
  \mathcal{G}g-f\in\mul{S} \quad\Leftrightarrow\quad \forall u\in\D: 0 = \spr{\mathcal{G}g-f}{u} = \spr{g}{\mathcal{F}u}_\mu - \spr{f}{u},
 \end{align*}
 which holds for every $f\in\Lr$ and $g\in\Lrmu$.
\end{proof}

Note that $\mathcal{F}$ is a unitary map from $\Lr$ onto $\Lrmu$ if and only if $S$ is an operator. 

\begin{theorem}\label{thmspectransS}
  The self-adjoint relation $S$ is given by $S = \mathcal{F}^\ast \M_{\mathrm{id}} \mathcal{F}$.
\end{theorem}

\begin{proof}
  First note that for each $f\in\D$ we have
 \begin{align*}
  f\in\dom{S} & \quad\Leftrightarrow\quad \int_\R |\lambda|^2 dE_{f,f}(\lambda)<\infty &\Leftrightarrow\quad &\int_\R |\lambda|^2 |\mathcal{F}f(\lambda)|^2 d\mu(\lambda) < \infty \\
              & \quad\Leftrightarrow\quad \mathcal{F}f\in\dom{\M_{\mathrm{id}}} &\Leftrightarrow\quad &f\in\dom{\mathcal{F}^\ast\M_{\mathrm{id}}\mathcal{F}}.
 \end{align*}
 Furthermore, if $(f,f_\tau)\in S$, then from Lemma~\ref{lemspectransEfg} and Lemma~\ref{lemspectransEfgmu} we infer
 \begin{align*}
  \spr{f_\tau}{g} &= \int_\R\lambda\, dE_{f,g}(\lambda) = \int_\R \lambda \mathcal{F}f(\lambda) \mathcal{F}g(\lambda)^\ast d\mu(\lambda) \\
              & = \int_\R \M_{\mathrm{id}}\mathcal{F} f(\lambda) \mathcal{F}g(\lambda)^\ast d\mu(\lambda) 
                = \spr{\mathcal{G}\M_{\mathrm{id}}\mathcal{F}f}{g}, \quad g\in\D
 \end{align*}
  and hence $\mathcal{G}\M_{\mathrm{id}}\mathcal{F}f=Pf_\tau$, where $P$ is the orthogonal projection onto $\D$.
 This and Lemma~\ref{lemspectransfourtrans} show that $(\M_{\mathrm{id}} \mathcal{F} f,f_\tau)\in\mathcal{F}^\ast$,
  which is equivalent to $(f,f_\tau)\in\mathcal{F}^\ast \M_{\mathrm{id}} \mathcal{F}$.
  Now if we conversely assume that $(g,g_\tau)\in\mathcal{F}^\ast \M_{\mathrm{id}} \mathcal{F}$, then 
  $(\M_{\mathrm{id}} \mathcal{F} g,g_\tau)\in\mathcal{F}^\ast$ (note that $g\in\dom{S}$).
  Hence $\mathcal{G}\M_{\mathrm{id}} \mathcal{F}g-g_\tau$ lies in $\mul{S}$ and since 
  $(g,\mathcal{G}\M_{\mathrm{id}} \mathcal{F}g)\in S$, we also get $(g,g_\tau)\in S$.
\end{proof}

Note that the self-adjoint operator $S_\D$ is unitarily equivalent to the operator of multiplication $\M_{\mathrm{id}}$.
 In fact, $\mathcal{F}$ is unitary as an operator from $\D$ onto $\Lrmu$ and maps $S_\D$ onto multiplication with the independent variable.
Now the spectrum can be read off from the boundary behavior of the singular Weyl--Titchmarsh--Kodaira function $m$ in the usual way.

\begin{corollary}\label{cor:splimm}
 The spectrum of $S$ is given by
\begin{align}
 \sigma(S) & = \sigma(S_\D) = \supp(\mu) = \overline{\{ \lambda\in\R \,|\, 0 < \limsup_{\varepsilon\downarrow 0} \im(m(\lambda+\I\varepsilon))\}}.
\end{align}
Moreover,
\begin{align*}
\sigma_p(S_\D) & = \{ \lambda\in\R \,|\, 0 < \lim_{\varepsilon\downarrow0} \varepsilon \im(m(\lambda+\I\varepsilon)) \}, \\
\sigma_{ac}(S_\D) & = \overline{\{ \lambda\in\R \,|\, 0 < \limsup_{\varepsilon\downarrow0} \im(m(\lambda+\I\varepsilon)) < \infty \}}^{ess},
\end{align*}
where $\overline{\Omega}^{ess} = \{ \lambda\in\R \,|\, |(\lambda-\varepsilon,\lambda+\varepsilon)\cap\Omega|>0 \text{ for all }\varepsilon>0 \}$,
is the essential closure of a Borel set $\Omega\subseteq\R$, and
\begin{align*}
\Sigma_s = \{ \lambda\in\R \,|\, \limsup_{\varepsilon\downarrow0} \im(m(\lambda+\I\varepsilon)) = \infty \}
\end{align*}
is a minimal support for the singular spectrum (singular continuous plus pure point spectrum) of $S_\D$.
\end{corollary}

\begin{proof}
Since the operator part $S_\D$ of $S$ is unitary equivalent to $\M_{\mathrm{id}}$ we infer from Lemma~\ref{lemSARelOper} that $\sigma(S)=\sigma(\M_{\mathrm{id}})=\supp(\mu)$. Now the remaining part of the claim follows as in~\cite[Corollary~3.5]{kst}.
\end{proof}

\begin{proposition}\label{propspectransEVmu}
 If $\lambda\in\sigma(S)$ is an eigenvalue, then
 \begin{align}
  \mu(\lbrace\lambda\rbrace) = \left\| \phi_\lambda\right\|^{-2}.
 \end{align}
\end{proposition}

\begin{proof}
 By assumption, $\phi_\lambda$ is an eigenvector, i.e., $(\phi_\lambda,\lambda \phi_\lambda)\in S$. 
  Hence we get from the proof of Theorem~\ref{thmspectransS} that $\M_{\mathrm{id}}\mathcal{F}\phi_\lambda=\lambda\mathcal{F}\phi_\lambda$.
  But this shows that $\mathcal{F}\phi_\lambda(z)$ vanishes for almost all $z\not=\lambda$ with respect to $\mu$.
  Now from this we get
  \begin{align*}
   \left\| \phi_\lambda\right\|^2 & = \left\| \mathcal{F}\phi_\lambda\right\|_\mu^2 = \int_{\lbrace\lambda\rbrace} \left|\mathcal{F}\phi_\lambda(z)\right|^2 d\mu(z) \\
                                  & = \mu(\lbrace\lambda\rbrace) \left(\int_a^b \phi_\lambda(x)^2 d\varrho(x)\right)^2 = \mu(\lbrace\lambda\rbrace) \left\|\phi_\lambda\right\|^4.
  \end{align*}
\end{proof}

With $P$ we denote the orthogonal projection from $\Lr$ onto $\D$. If $S$ is an operator, $P$ is simply the identity.

\begin{lemma}\label{lemspectransGreentransform}
 For every $z\in\rho(S)$ and all $x\in(a,b)$ the transform of the Green function $G_z(x,\cdot\,)$ and its quasi-derivative $\partial_x^\qd G_z(x,\cdot\,)$ are given by
 \begin{align*}
  \mathcal{F}P G_z(x,\cdot\,) (\lambda) = \frac{\phi_\lambda(x)}{\lambda-z} \quad\text{and}\quad 
      \mathcal{F}P \partial_x^\qd G_z(x,\cdot\,) (\lambda) = \frac{\phi_\lambda^\qd(x)}{\lambda-z}, \quad \lambda\in\R.
 \end{align*}
\end{lemma}

\begin{proof}
 First note that $G_z(x,\cdot\,)$ and $\partial_x^\qd G_z(x,\cdot\,)$ both lie in $\Lr$.
 Then, using Lemma~\ref{lemspectransEfgmu} we get for each $f\in L^2_c((a,b);\varrho)$ and $g\in L^2_c(\R;\mu)$
 \begin{align*}
  \spr{R_z \check{g}}{f} & = \int_\R \frac{g(\lambda) \hat{f}(\lambda)^\ast}{\lambda-z}d\mu(\lambda) 
            = \int_a^b \int_\R \frac{\phi_\lambda(x)}{\lambda-z} g(\lambda) d\mu(\lambda)\, f(x)^\ast d\varrho(x).
 \end{align*}
 Hence we have 
 \begin{align*}
  R_z \check{g}(x) = \int_\R \frac{\phi_\lambda(x)}{\lambda-z} g(\lambda) d\mu(\lambda)
 \end{align*}
 for almost all $x\in(a,b)$ with respect to $\varrho$. Using Theorem~\ref{thm:ressep} one gets
 \begin{align*}
  \spr{\mathcal{F}P G_z(x,\cdot\,)}{g^\ast}_\mu & = \spr{G_z(x,\cdot\,)}{\check{g}^\ast} = \int_\R \frac{\phi_\lambda(x)}{\lambda-z} g(\lambda) d\mu(\lambda)
 \end{align*}
 for almost all $x\in(a,b)$ with respect to $\varrho$. Since all three terms are absolutely
  continuous with respect to $\varsigma$, this equality is true for all $x\in(a,b)$, which proves the first 
  part of the claim.
 The second equality follows from
 \begin{align*}
  \spr{\mathcal{F}P \partial_x G_z(x,\cdot\,)}{g^\ast}_\mu = \spr{\partial_x G_z(x,\cdot\,)}{\check{g}^\ast} = R_z \check{g}^\qd(x) = \int_\R \frac{\phi_\lambda^\qd(x)}{\lambda-z} g(\lambda) d\mu(\lambda).
 \end{align*}
\end{proof}

Note that $\mathcal{F}P$ is the unique extension to $\Lr$ of the bounded linear mapping defined in~\eqref{eqnSThatf} on $L^2_c((a,b);\varrho)$.

\begin{lemma}\label{lemspectranweyltrans}
 Suppose $\tau$ is in the l.c.~case at $a$ and $\theta_z$, $\phi_z$, $z\in\C$ is a fundamental system as in Theorem~\ref{thmweyltitchfundsys}.
  Then for each $z\in\rho(S)$ the transform of the Weyl solution $\psi_z$ is given by
 \begin{align}
  \mathcal{F}P \psi_z(\lambda) = \frac{1}{\lambda-z},\quad \lambda\in\R.
 \end{align}
\end{lemma}

\begin{proof}
 From Lemma~\ref{lemspectransGreentransform} we obtain for each $x\in(a,b)$
 \begin{align*}
  \mathcal{F}P \tilde{\psi}_z(x,\cdot\,)(\lambda) = \frac{W(\theta_z,\phi_\lambda)(x)}{\lambda-z}, \quad \lambda\in\R,
 \end{align*}
 where
 \begin{align*}
  \tilde{\psi}_z(x,y) = \begin{cases}
                         m(z) \phi_z(y), & \text{if }y< x, \\
                         \psi_z(y), & \text{if }y\geq x.
                        \end{cases}
 \end{align*}
 Now the claim follows by letting $x\rightarrow a$, using Theorem~\ref{thmweyltitchfundsys}.
\end{proof}

 Under the assumptions of Lemma~\ref{lemspectranweyltrans}, $m$ is a Herglotz--Nevanlinna function. Hence we have
 \begin{align}\label{eqnSThergMrep}
  m(z) = c_1 + c_2 z + \int_\R \frac{1}{\lambda-z} - \frac{\lambda}{1+\lambda^2}d\mu(\lambda),\quad z\in\C\backslash\R,
 \end{align}
 where the constants $c_1$, $c_2$ are given by 
\begin{align*}
 c_1 = \re(m(\I)) \quad\text{and}\quad c_2=\lim_{\eta\uparrow\infty} \frac{m(\I\eta)}{\I\eta}\geq 0.
\end{align*}

\begin{corollary}
 Suppose $\tau$ is in the l.c.~case at $a$ and $\theta_z$, $\phi_z$, $z\in\C$ is a fundamental system as in Theorem~\ref{thmweyltitchfundsys}.
  Then the second constant in~\eqref{eqnSThergMrep} is given by
 \begin{align*}
  c_2=\lim_{\eta\uparrow\infty} \frac{m(i\eta)}{i\eta} = 
   \begin{cases} \theta_z(\alpha_\varrho)^2 \varrho(\lbrace\alpha_\varrho\rbrace), & \text{if }\indik_{\lbrace\alpha_\varrho\rbrace}\in\mul{S}, \\
      0, &\text{otherwise.} \end{cases}
\end{align*}
\end{corollary}

\begin{proof}
 Taking imaginary parts in~\eqref{eqnSThergMrep} yields for each $z\in\C\backslash\R$
 \begin{align*}
  \im(m(z)) & = c_2\,\im(z) + \int_\R\im\left(\frac{1}{\lambda-z}\right)d\mu(\lambda) \\
            & = c_2\,\im(z) + \int_\R \frac{\im(z)}{|\lambda-z|^2}d\mu(\lambda).
 \end{align*}
 From this we get, using Lemma~\ref{lemspectranweyltrans} and~\eqref{eqnWTHerglotz}
 \begin{align*}
  c_2 + \int_\R \frac{1}{|\lambda-z|^2}d\mu(\lambda) & = \frac{\im(m(z))}{\im(z)} = \|\psi_z\|^2 = \|(I-P)\psi_z\|^2 + \|\mathcal{F}P\psi_z\|_\mu^2 \\
      & = \|(I-P)\psi_z\|^2 + \int_\R \frac{1}{|\lambda-z|^2}d\mu(\lambda). 
 \end{align*}
 Hence we have (note that $\psi_z(\beta_\varrho)=0$ if $\indik_{\lbrace\beta_\varrho\rbrace}\in\mul{S}\backslash\lbrace0\rbrace$)
 \begin{align*}
  c_2 = \|(I-P)\psi_z\|^2 = \begin{cases}
                             \left|\psi_z(\alpha_\varrho)\right|^2 \varrho(\lbrace\alpha_\varrho\rbrace), & \text{if }\indik_{\lbrace\alpha_\varrho\rbrace}\in\mul{S},\\
                             0, &\text{otherwise.}
                            \end{cases}
 \end{align*}
 Now assume $\indik_{\lbrace\alpha_\varrho\rbrace}\in\mul{S}\backslash\lbrace0\rbrace$, then $\phi_z(\alpha_\varrho)=0$ and hence
 \begin{align*}
  c_2 = \left|\theta_z(\alpha_\varrho) + m(z)\phi_z(\alpha_\varrho)\right|^2 \varrho(\lbrace\alpha_\varrho\rbrace) = \left|\theta_z(\alpha_\varrho)\right|^2 \varrho(\lbrace\alpha_\varrho\rbrace), \quad z\in\C\backslash\R.
 \end{align*}
 Finally, since $\theta_z$ is a real entire function, this proves the claim.
\end{proof}

\begin{remark}
Given another singular Weyl--Titchmarsh--Kodaira function $\tilde{m}$ as in Remark~\ref{remWTtilde}, the corresponding spectral measures are related by
\begin{align}
 \tilde{\mu} = e^{-2g} \mu,
\end{align}
 where $\E^g$ is the real entire function appearing in Remark~\ref{remWTtilde}.
 Hence the measures are mutually absolutely continuous and the associated spectral transformations just differ by a simple rescaling with the positive function $e^{-2g}$.
 Also note that the spectral measure does not depend on the choice of the second solution $\theta_z$.
\end{remark}

\section{Spectral transformation II}

In this section let $S$ be a self-adjoint restriction of $\Tmax$ with separated boundary conditions as in the preceding section.
We now want to consider the case where none of the endpoints satisfies the requirements of the previous section. In such a
situation the spectral multiplicity of $S$ could be two and hence we will need to work with a matrix-valued transformation.

In the following we will fix some $x_0\in(a,b)$ and consider the real entire fundamental system of solutions
$\theta_z$, $\phi_z$, $z\in\C$ with the initial conditions
\begin{align*}
 \phi_z(x_0) = -\theta_z^\qd(x_0) = -\sin(\vphi_\alpha) \quad\text{and}\quad \phi^\qd_z(x_0)=\theta_z(x_0) = \cos(\vphi_\alpha)
\end{align*}
for some fixed $\vphi_\alpha\in[0,\pi)$.
The Weyl solutions are given by
\begin{align}
 \psi_{z,\pm}(x) = \theta_z(x) \pm m_\pm(z)\phi_z(x), \quad x\in(a,b),~z\in\C\backslash\R,
\end{align} 
such that $\psi_-$ lies in $L^2((a,b);\varrho)$ near $a$ and $\psi_+$ lies in $L^2((a,b);\varrho)$ near $b$.
Here $m_\pm$ are the Weyl--Titchmarsh--Kodaira functions of the operators $S_\pm$ obtained by restricting $S$
to $(a,x_0)$ and $(x_0,b)$ with a boundary condition 
\begin{align*}
\cos(\vphi_\alpha) f(x_0) + \sin(\vphi_\alpha) f^\qd(x_0)=0,
\end{align*}
 respectively. 
According to Corollary~\ref{spek::manal} the functions $m_\pm$ are Herglotz--Nevanlinna functions.
Now we introduce the $2\times 2$ Weyl--Titchmarsh--Kodaira matrix
\begin{align}\label{eq:wmmat}
M(z) = \begin{pmatrix}
         -\frac{1}{m_-(z) + m_+(z)} & \frac{1}{2} \frac{m_-(z) - m_+(z)}{m_-(z) + m_+(z)} \\
         \frac{1}{2} \frac{m_-(z) - m_+(z)}{m_-(z) + m_+(z)} & \frac{m_-(z)m_+(z)}{m_-(z)+m_+(z)}
       \end{pmatrix}, \quad z\in\C\backslash\R.
\end{align}
In particular, note that we have $\det(M(z)) = -\frac{1}{4}$.
The function $M$ is a matrix Herglotz--Nevanlinna function with representation
\begin{align*}
 M(z) = C_1 + C_2 z + \int_\R \left(\frac{1}{\lambda-z}-\frac{\lambda}{1+\lambda^2}\right) d\Omega(\lambda), \quad z\in\C\backslash\R,
\end{align*}
where $C_1$ is a self-adjoint matrix, $C_2$ a non-negative matrix and $\Omega$ is a symmetric matrix-valued measure given by the Stieltjes inversion formula
\begin{align*}
 \Omega((\lambda_1,\lambda_2]) = \lim_{\delta\downarrow 0} \lim_{\eps\downarrow 0} \frac{1}{\pi} \int_{\lambda_1+\delta}^{\lambda_2+\delta} \im(M(\lambda+\I\eps)) d\lambda, \quad\lambda_1,\lambda_2\in\R, ~\lambda_1<\lambda_2.
\end{align*}
Moreover, the trace $\Omega^{\mathrm{tr}} = \Omega_{1,1}+\Omega_{2,2}$ of $\Omega$ is a non-negative measure and the components of $\Omega$ are absolutely continuous with respect to $\Omega^{\mathrm{tr}}$. The respective densities are denoted by $R_{i,j}$, $i, j\in\lbrace 1,2\rbrace$, and are given by
\begin{align}\label{eq:Rlim}
R_{i,j}(\lambda) = \lim_{\eps\downarrow 0} \frac{\im(M_{i,j}(\lambda+\I\eps))}{\im(M_{1,1}(\lambda+\I\eps) + M_{2,2}(\lambda+\I\eps))},
\end{align}
where the limit exists almost everywhere with respect to $\Omega^{\mathrm{tr}}$. Note that $R$ is non-negative and has trace equal to one. In particular, all entries
of $R$ are bounded; $0 \le R_{1,1},R_{2,2} \le 1$ and $| R_{1,2} | = |R_{2,1}| \le 1/2$.

Furthermore, the corresponding Hilbert space $L^2(\R;\Omega)$ is associated with the inner product
\begin{align*}
 \spr{\hat{f}}{\hat{g}}_\Omega = \int_\R \hat{f}(\lambda)\hat{g}(\lambda)^\ast d\Omega(\lambda) = 
 \int_\R \sum_{i,j=1}^2 \hat{f}_i(\lambda) R_{i,j}(\lambda) \hat{g}_j(\lambda)^\ast d\Omega^{\mathrm{tr}}(\lambda).
\end{align*}
Now for each $f\in L^2_c((a,b);\varrho)$ we define the transform of $f$ as 
\begin{align}\label{eqnSTIItrans}
 \hat{f}(z) = \begin{pmatrix}
                \int_a^b \theta_z(x)f(x) d\varrho(x) \\ \int_a^b \phi_z(x) f(x) d\varrho(x)
              \end{pmatrix}, \quad z\in\C.
\end{align}
In the following lemma we will relate the $2\times 2$ matrix-valued measure $\Omega$ to the operator-valued spectral measure $E$ of $S$. If $F$ is a measurable function on $\R$, we denote with $\M_F$ the maximally defined operator of multiplication with $F$ in the Hilbert space $L^2(\R;\Omega)$.

\begin{lemma}\label{lemSTIIisom}
 If $f$, $g\in L^2_c((a,b);\varrho)$, then we have
  \begin{align}
    \spr{E((\lambda_1,\lambda_2])f}{g} = \spr{\M_{\indik_{(\lambda_1,\lambda_2]}} \hat{f}}{\hat{g}}_\Omega
  \end{align}
  for all $\lambda_1$, $\lambda_2\in\R$ with $\lambda_1<\lambda_2$.
\end{lemma}

\begin{proof}
This follows by evaluating Stone's formula
\[
\spr{E((\lambda_1,\lambda_2])f}{g} =
\lim_{\delta\downarrow 0} \lim_{\eps\downarrow 0} \frac{1}{\pi} \int_{\lambda_1+\delta}^{\lambda_2+\delta} \im\left( \spr{R_{\lam+\I\eps} f}{g}\right) d\lam,
\]
using our formula for the resolvent \eqref{eq:ressbc} together with Stieltjes inversion formula, literally following the proof of \cite[Theorem~2.12]{gz}.
\end{proof}

Lemma~\ref{lemSTIIisom} shows that the transformation defined in~\eqref{eqnSTIItrans} uniquely extends to an
 isometry $\mathcal{F}$ from $\Lr$ into $L^2(\R;\Omega)$.

\begin{theorem}
 The operator $\mathcal{F}$ is unitary with inverse given by
 \begin{align}\label{eq:Finv2}
 \mathcal{F}^{-1} g(x) =
 \lim_{N\rightarrow\infty} \int_{[-N,N)} g(\lam) \begin{pmatrix} \theta_\lam(x)\\ \phi_\lam(x)\end{pmatrix} d\Omega(\lam), \quad g\in L^2(\R;\Omega),
 \end{align}
 where the limit exists in $\Lr$.
 Moreover, $\mathcal{F}$ maps $S$ onto $\M_{\mathrm{id}}$.
\end{theorem}

\begin{proof}
By our previous lemma it remains to show that $\mathcal{F}$ is onto. Since it is straightforward to verify
that the integral operator on the right-hand side of \eqref{eq:Finv2} is the adjoint of $\mathcal{F}$, we can
equivalently show $\ker(\mathcal{F}^*)=\{0\}$. 
To this end let $g\in L^2(\R, \Omega)$, $N\in\N$ and $z\in\rho(S)$, then
\[
(S-z) \int_{-N}^N \frac{1}{\lam-z} g(\lam) \begin{pmatrix} \theta_\lam(x)\\ \phi_\lam(x)\end{pmatrix} d\Omega(\lam)
= \int_{-N}^N g(\lam) \begin{pmatrix} \theta_\lam(x)\\ \phi_\lam(x)\end{pmatrix} d\Omega(\lam),
\]
since interchanging integration with the Radon--Nikod\'{y}m derivatives can be justified using Fubini's theorem. Taking the limit
$N\to\infty$ we conclude
\[
\mathcal{F}^* \frac{1}{\cdot -z} g = R_z \mathcal{F}^* g, \quad g\in L^2(\R, \Omega).
\]
By Stone--Weierstra{\ss} we even conclude $\mathcal{F}^* \M_F g = F(S) \mathcal{F}^* g$ for any
continuous function $F$ vanishing at infinity and by a consequence of the spectral theorem (e.g.\ the last part of
\cite[Theorem~3.1]{tschroe}) we can further extend this to characteristic functions of intervals $I$. Hence, for
$g \in \ker(\mathcal{F}^*)$
we conclude
\[
\int_I g(\lam) \begin{pmatrix} \theta_\lam(x)\\ \phi_\lam(x)\end{pmatrix} d\Omega(\lam) =0
\]
for any compact interval $I$. Moreover, after taking Radon--Nikod\'{y}m derivatives we also have
\[
\int_I g(\lam) \begin{pmatrix} \theta^\qd_\lam(x)\\ \phi^\qd_\lam(x)\end{pmatrix} d\Omega(\lam) =0.
\]
Choosing $x=x_0$ we see
\[
\int_I g(\lam) \begin{pmatrix} \cos(\vphi_\alpha)\\ -\sin(\vphi_\alpha)\end{pmatrix} d\Omega(\lam) =
\int_I g(\lam) \begin{pmatrix} \sin(\vphi_\alpha)\\ \cos(\vphi_\alpha)\end{pmatrix} d\Omega(\lam) =0
\]
for any compact interval $I$ and thus $g=0$ as required.
\end{proof}

Next, there is a measurable unitary matrix $U(\lam)$ which diagonalizes
$R(\lam)$, that is,
\begin{equation} \label{defrulam}
R(\lam) = U(\lam)^* \begin{pmatrix} r_1(\lam) & 0 \\ 0 & r_2(\lam) \end{pmatrix} U(\lam),
\end{equation}
where $0\le r_1(\lam) \le r_2(\lam)\le 1$ are the eigenvalues of $R(\lam)$.
Also note that $r_1(\lam)+r_2(\lam)=1$ by $\tr(R(\lam))=1$. The matrix $U(\lam)$
provides a unitary operator $L^2(\R;\Omega) \to L^2(\R; r_1 d\Omega^{\mathrm{tr}}) \oplus L^2(\R; r_2 d\Omega^{\mathrm{tr}})$
which leaves $\M_{\mathrm{id}}$ invariant. From this observation we immediately obtain
the analog of Corollary~\ref{cor:splimm}.

\begin{corollary}
Introduce the Herglotz--Nevanlinna function
\be
M^\mathrm{tr}(z) = \tr(M(z))= \frac{m_-(z)m_+(z) -1}{m_-(z)+m_+(z)}, \quad z\in\C\backslash\R,
\ee
associated with the measure $\Omega^\mathrm{tr}$.
Then the spectrum of $S$ is given by
\begin{align}
 \sigma(S) & =  \supp(\Omega^\mathrm{tr}) = \overline{\lbrace \lambda\in\R \,|\, 0 < \limsup_{\varepsilon\downarrow 0} \im(M^\mathrm{tr}(\lambda+\I\varepsilon))\rbrace}.
\end{align}
Moreover,
\begin{align*}
\sigma_p(S) & = \{ \lambda\in\R \,|\, 0 < \lim_{\varepsilon\downarrow0} \varepsilon \im(M^\mathrm{tr}(\lambda+\I\varepsilon)) \}, \\
\sigma_{ac}(S) & = \overline{\{ \lambda\in\R \,|\, 0 < \limsup_{\varepsilon\downarrow0} \im(M^\mathrm{tr}(\lambda+\I\varepsilon)) < \infty \}}^{ess},
\end{align*}
and
\begin{align*}
\Sigma_s = \{ \lambda\in\R \,|\, \limsup_{\varepsilon\downarrow0} \im(M^\mathrm{tr}(\lambda+\I\varepsilon)) = \infty \}
\end{align*}
is a minimal support for the singular spectrum (singular continuous plus pure point spectrum) of $S$.
\end{corollary}

Furthermore, this allows us to investigate the spectral multiplicity of $S$.

\begin{lemma} \label{lem:spmul}
If we define
\begin{align*}
\Sigma_1 &= \{ \lam\in \supp(\Omega^{tr}) \,| \det R(\lam)= r_1(\lam) r_2(\lam)=0 \},\\
\Sigma_2 &= \{ \lam\in \supp(\Omega^{tr}) \,| \det R(\lam)= r_1(\lam) r_2(\lam)>0 \},
\end{align*}
then $\M_{\mathrm{id}}=  \M_{\mathrm{id}\cdot \indik_{\Sigma_1}} \oplus  \M_{\mathrm{id}\cdot \indik_{\Sigma_2}}$ and
the spectral multiplicity of $\M_{\mathrm{id}\cdot \indik_{\Sigma_1}}$ is one and
the spectral multiplicity of $\M_{\mathrm{id}\cdot \indik_{\Sigma_2}}$ is two.
\end{lemma}

\begin{proof}
For fixed $\lam\in\Sigma_1$ we have either $r_1(\lam)=1$, $r_2(\lam)=0$ or $r_1(\lam)=0$, $r_2(\lam)=1$. In the latter
case we can modify $U(\lam)$ to also switch components and hence we can assume $r_1(\lam)=1$, $r_2(\lam)=0$
for all $\lam\in\Sigma_1$. Hence $\M_{\mathrm{id}\cdot \indik_{\Sigma_1}}$ is unitarily equivalent to
multiplication with $\lam$ in $L^2(\R;\indik_{\Sigma_1} \Omega^{\mathrm{tr}})$.
Moreover, since $r_j \indik_{\Sigma_2} \Omega^{\mathrm{tr}}$ and $\indik_{\Sigma_2} \Omega^{\mathrm{tr}}$
are mutually absolutely continuous, $\M_{\mathrm{id}\cdot \indik_{\Sigma_2}}$ is unitary equivalent
to $\M_{\mathrm{id}}$ in the Hilbert space $L^2(\R; \indik_{\Sigma_1} \Omega^{\mathrm{tr}} I_2)$.
\end{proof}

Combining \eqref{eq:wmmat} with \eqref{eq:Rlim} we see that 
\begin{align*}
\det R(\lam) = \lim_{\eps\downarrow 0} \frac{\im(m_+(\lam+\I\eps))
\im(m_-(\lam+\I\eps))}{|m_+(\lam+\I\eps) + m_-(\lam+\I\eps)|^2}
\frac{1}{\im(M^\mathrm{tr}(\lam+\I\eps))^2},
\end{align*}
where the first factor is bounded by $\nicefrac{1}{4}$.
Now Lemma~\ref{lem:spmul} immediately gives the following result. 

\begin{lemma}
The singular spectrum of $S$ has spectral multiplicity one. The
absolutely continuous spectrum of $S$ has multiplicity two on the subset 
$\sig_{ac}(S_+) \cap \sig_{ac}(S_-)$ and multiplicity one
on $\sig_{ac}(S) \backslash (\sig_{ac}(S_+) \cap \sig_{ac}(S_-))$.
Here $S_\pm$ are the restrictions of $S$ to $(a,x_0)$ and $(x_0,b)$, respectively.
\end{lemma}

\begin{proof}
Using the fact that $\Sigma_s$ is a minimal support for the singular part of $S$ we
obtain $S_s= S_{pp} \oplus S_{sc} = E(\Sigma_s) S$ and $S_{ac}= (1-E(\Sigma_s)) S$.
So we see that the singular part has multiplicity one by
Lemma~\ref{lem:spmul}.

For the absolutely continuous part use that
\[
\Sigma_{ac,\pm} = \{ \lam\in\R \,|\, 0< \lim_{\eps\downarrow 0}
\im(m_\pm(\lam+\I\eps)) <\infty \}
\]
are minimal supports for the absolutely continuous spectra of $S_\pm$.
Again the remaining result follows from Lemma~\ref{lem:spmul}.
\end{proof}

\appendix

\section{Linear measure differential equations}\label{appMODE}

In this appendix we collect some necessary facts from linear differential equations with measure coefficients.
We refer to  Bennewitz~\cite{bennewitz1989}, Persson~\cite{persson1988}, Volkmer~\cite{volkmer2005}, Atkinson~\cite{at} or
Schwabik, Tvrd\'{y} and Vejvoda~\cite{schwabik} for further information. In order to make our presentation
self-contained we have included proofs for all results.

Let $(a,b)$ be a finite or infinite interval and $\omega$ a positive Borel measures on $(a,b)$.
Furthermore, let $M$ be a $\C^{n\times n}$-valued measurable function on $(a,b)$ and $F$ a $\C^n$-valued 
 measurable function on $(a,b)$, such that $\|M(\cdot)\|$ and $\|F(\cdot)\|$ are locally integrable with respect to $\omega$. 
 Here $\|\cdot\|$ denotes some norm on $\C^n$ as well as the corresponding operator norm on $\C^{n\times n}$.

For $c\in(a,b)$ and $Y_c\in\C^n$, some $\C^n$-valued function $Y$ on $(a,b)$ is a solution of 
 the initial value problem
\begin{align}\label{eqnappLMDE}
 \frac{dY}{d\omega} = MY + F, \qquad Y(c) = Y_c,
\end{align}
if the components of $Y$ are locally absolutely continuous with respect to $\omega$, their Radon--Nikod\'{y}m 
 derivatives satisfy~\eqref{eqnappLMDE} almost everywhere with respect to $\omega$.
An integration shows that some function $Y$ is a solution of the initial value problem~\eqref{eqnappLMDE} if and only if it solves
 the vector integral equation
\begin{align}\label{eqn:MIE}
 Y(x) = Y_c + \int_c^x \left( M(t)Y(t) + F(t) \right) d\omega(t), \quad x\in(a,b).
\end{align}

Before we prove existence and uniqueness for solutions of this initial value problem, we need a
Gronwall lemma. The proof follows \cite[Lemma~1.2 and Lemma~1.3]{bennewitz1989} (see also Atkinson \cite[page 455]{at}).

\begin{lemma}\label{lemappGronwall}
 Let $c\in(a,b)$ and $v\in L^1_{\mathrm{loc}}((a,b);\omega)$ be real-valued such that
 \begin{align}
  0 \leq v(x) \leq K + \int_c^x v(t) d\omega(t), \quad x\in[c,b)
 \end{align}
 for some constant $K\geq 0$, then $v$ can be estimated by
 \begin{align}
  v(x) \leq K \E^{ \int_c^x d\omega}, \quad x\in[c,b). 
 \end{align}
\end{lemma}

\begin{proof}
First of all note that the function $F(x)=\int_c^x d\omega$, $x\in[c,b)$, satisfies
\begin{align}\label{eqnappLMDEGronwallH}
F(x)^{n+1} \ge (n+1)  \int_c^x F(t)^n d\omega(t), \quad x\in[c,b),
\end{align}
by a variant of the substitution rule for Lebesgue--Stieltjes integrals \cite[eq.\ (8)]{ft}.
 Now we will prove that
\begin{align*}
 v(x) \leq K \sum_{k=0}^n \frac{F(x)^k}{k!} + \frac{F(x)^n}{n!} \int_c^x v(t)d\omega(t), \quad x\in[c,b)
\end{align*}
for each $n\in\N_0$.
For $n=0$ this is just the assumption of our lemma. Otherwise we get inductively 
\begin{align*}
 v(x) & \leq K + \int_c^x v(t) d\omega(t) \\
      & \leq K + \int_c^x \left( K \sum_{k=0}^n \frac{F(t)^k}{k!} + \frac{F(t)^n}{n!} \int_c^t v\,d\omega\right) \, d\omega(t) \\
      & \leq K\left(1+\sum_{k=0}^n \int_c^x \frac{F(t)^k}{k!}d\omega(t) \right) + \int_c^x \frac{F(t)^n}{n!}d\omega(t) \int_c^x v\, d\omega \\
      & \leq K \sum_{k=0}^{n+1} \frac{F(x)^k}{k!} + \frac{F(x)^{n+1}}{(n+1)!} \int_c^x v\, d\omega, \quad x\in[c,b),
\end{align*}
where we used~\eqref{eqnappLMDEGronwallH} twice in the last step.
Now taking the limit $n\rightarrow\infty$ yields the claim.
\end{proof}

Because of the definition of our integral the assertion of this lemma is only true to the right of $c$.
However, a simple reflection proves that 
 \begin{align}
  0 \leq v(x) \leq K + \int_{x+}^{c+} v(t) d\omega(t), \quad x\in(a,c]
 \end{align}
 for some constant $K\geq 0$, implies
 \begin{align}\label{eqnappLMDEGronwallLeft}
  v(x) \leq K e^{ \int_{x+}^{c+} d\omega}, \quad x\in(a,c]. 
 \end{align}
We are now ready to prove the basic existence and uniqueness result.

\begin{theorem}\label{thm:appMIEEE}
 The initial value problem~\eqref{eqnappLMDE} has a unique solution for each $c\in(a,b)$ and $Y_c\in\C^n$ 
 if and only if the matrix
 \begin{align}\label{eqn:appImOreg}
  I + \omega(\lbrace x\rbrace) M(x)\quad\text{is invertible}
 \end{align}
 for all $x\in(a,b)$.
 In this case solutions are real if $M$, $F$ and $Y_c$ are real.
\end{theorem}

\begin{proof}
 First assume that the initial value problem~\eqref{eqnappLMDE} has a unique solution for each $c\in(a,b)$ 
  and $Y_c\in\C^n$. Now if the matrix \eqref{eqn:appImOreg} was not invertible for some $x_0\in(a,b)$, we would have two distinct solutions
  $Y_1$, $Y_2$ such that $Y_1(x_0)\not=Y_2(x_0)$ but $Y_1(x_0+)=Y_2(x_0+)$. Indeed, one only had to take
  solutions with different initial conditions at $x_0$ such that
 \begin{align*}
  Y_1(x_0+) + \omega(\lbrace x_0\rbrace) F(x_0) & = \left(  I + \omega(\lbrace x_0\rbrace) M(x_0) \right) Y_1(x_0) \\
                 & = \left(  I + \omega(\lbrace x_0\rbrace) M(x_0) \right) Y_2(x_0) \\
                 & = Y_2(x_0+) + \omega(\lbrace x_0\rbrace) F(x_0).
 \end{align*}
 But then one had
 \begin{align*}
  \left\| Y_1(x) - Y_2(x)\right\| & = \left\| \int_{x_0+}^x M(t)\left(Y_1(t)-Y_2(t)\right)d\omega(t) \right\| \\
         & \leq \int_{x_0+}^x \left\|M(t)\right\| \left\| Y_1(t) - Y_2(t)\right\| d\omega(t), \quad x\in(x_0,b)
 \end{align*}
 and hence by Lemma~\ref{lemappGronwall}, $Y_1(x)=Y_2(x)$ for all $x\in(x_0,b)$. But this is a contradiction since now
 $Y_1$ and $Y_2$ are two different solutions of the initial value problem with $Y_c=Y(c)$ for some $c\in(x_0,b)$.

 Now assume~\eqref{eqn:appImOreg} holds for all $x\in(a,b)$ and let $\alpha$, $\beta\in(a,b)$ with $\alpha<c<\beta$.
 It suffices to prove that there is a unique solution of~\eqref{eqn:MIE} on $(\alpha,\beta)$.
 In order to prove uniqueness, take a solution $Y$ of the homogenous system, i.e., $Y_c=0$ and $F=0$. We get
 \begin{align*}
  \left\|Y(x)\right\| & \leq \int_c^x \left\|M(t)\right\| \left\|Y(t)\right\| d\omega(t), \quad x\in[c,\beta)
 \end{align*}
and hence $Y(x)=0$, $x\in[c,\beta)$, by Lemma~\ref{lemappGronwall}. To the left-hand side of the point $c$ we have 
 \begin{align*}
  Y(x) & = -\int_x^c M(t)Y(t)d\omega(t) = -\int_{x+}^c M(t)Y(t)d\omega(t) - \omega(\lbrace x\rbrace)M(x)Y(x) \\
       & = -\left(I+\omega(\lbrace x\rbrace)M(x)\right)^{-1} \int_{x+}^c M(t)Y(t)d\omega(t), \quad x\in(\alpha,c)
\end{align*}
and hence
\begin{align*}
  \left\|Y(x)\right\| & \leq \left\| \left(I+\omega(\lbrace x\rbrace)M(x)\right)^{-1}\right\| \int_{x+}^{c+} \left\|M(t)\right\| \left\|Y(t)\right\| d\omega(t), \quad x\in(\beta,c].
\end{align*}
Now the function in front of the integral is bounded. Indeed, since $M$ is locally integrable, we have $\omega(\lbrace x\rbrace)\|M(x)\|<\frac{1}{2}$
 for all but finitely many $x\in[\beta,c]$. Moreover, for those $x$ we have
\begin{align*}
 \left\| \left(I+\omega(\lbrace x\rbrace)M(x)\right)^{-1}\right\| & = \left\| \sum_{n=0}^\infty \left(-\omega(\lbrace x\rbrace)M(x)\right)^n\right\| 
    \leq 2.
\end{align*}
Therefore, estimate~\eqref{eqnappLMDEGronwallLeft} applies and yields $Y(x)=0$, $x\in(\beta,c]$.

Next we will construct the solution by successive approximation. To this end we define
\begin{align}\label{eqnappLMDEY0}
 Y_0(x) = Y_c + \int_{c}^x F(t) d\omega(t), \quad x\in[c,\beta)
\end{align}
and inductively for each $n\in\N$
\begin{align}\label{eqnappLMDEYn}
  Y_n(x) = \int_{c}^x M(t)Y_{n-1}(t)d\omega(t), \quad x\in[c,\beta).
 \end{align}
 These functions are bounded by
 \begin{align}\label{eqnappLMDEboundYn}
  \left\| Y_n(x)\right\| \leq \sup_{t\in[c,x)}\|Y_0(t)\| \frac{\left(\int_c^x\left\|M(t)\right\|d\omega(t)\right)^n}{n!}, \quad x\in[c,\beta).
 \end{align}
 Indeed, for $n=0$ this is obvious, for $n>0$ we get inductively, using~\eqref{eqnappLMDEGronwallH},
 \begin{align*}
  \left\| Y_n(x)\right\| & \leq \int_c^x \left\|M(t)\right\| \left\|Y_{n-1}(t)\right\| d\omega(t) \\
         & \leq \sup_{t\in[c,x)}\|Y_0(t)\| \int_c^x \left\|M(t)\right\| \frac{\left(\int_c^t\left\|M(s)\right\|d\omega(s)\right)^n}{n!} d\omega(t) \\
         & \leq \sup_{t\in[c,x)}\|Y_0(t)\| \frac{\left(\int_c^x\left\|M(t)\right\|d\omega(t)\right)^{n+1}}{(n+1)!}.
 \end{align*}
 Hence the sum $Y(x)=\sum_{n=0}^\infty Y_n(x)$, $x\in[c,\beta)$ converges absolutely and uniformly. Moreover, we have
 \begin{align*}
  Y(x) & = Y_0(x) + \sum_{n=1}^\infty \int_{c}^x M(t)Y_{n-1}(t)d\omega(t) \\
       & = Y_c + \int_{c}^x M(t)Y(t) +F(t)d\omega(t), \quad x\in[c,\beta).
 \end{align*}
 In order to extend the solution to the left of $c$, pick some points $x_k\in[\alpha,c]$, $k=-N,\ldots,0$ with 
 \begin{align*}
  \alpha = x_{-N} < x_{-N+1} < \cdots < x_0 = c,
 \end{align*}
 such that
 \begin{align}\label{eqn:appomegabound}
  \int_{(x_k,x_{k+1})}  \|M(t)\| d\omega(t) < \frac{1}{2}, \quad -N \leq k < 0,
 \end{align}
 which is possible since $M$ is locally integrable. More precisely, first take all points $x\in(\alpha,c)$ 
  with $\omega(\lbrace x\rbrace) \|M(x)\| \geq \frac{1}{2}$ 
  (these are at most finitely many because $\|M(\cdot)\|$ is locally integrable). 
  Then divide the remaining subintervals such that~\eqref{eqn:appomegabound} is valid.
  Now let $-N<k\leq 0$ and assume $Y$ is a solution on $[x_k,\beta)$. 
  We will show that $Y$ can be extended to a solution on $[x_{k-1},\beta)$.
  To this end we define
 \begin{align}\label{eqnappLMDEZ0}
  Z_0(x) = Y(x_k) + \int_{x_k}^x F(t) d\omega(t), \quad x\in(x_{k-1},x_k]
 \end{align}
 and inductively for each $n\in\N$
 \begin{align}\label{eqnappLMDEZn}
  Z_n(x) = \int_{x_k}^x M(t)Z_{n-1}(t)d\omega(t), \quad x\in(x_{k-1},x_k].
 \end{align}
 Using~\eqref{eqn:appomegabound} it is not hard to prove inductively that for each $n\in\N$ and $x\in(x_{k-1},x_k]$ these functions are bounded by
 \begin{align}\label{eqnappLMDEboundZn}
  \left\| Z_n(x) \right\| \leq \left(\left\| Y(x_k)\right\| + \int_{[x_{k-1},x_k]} \left\| F(t)\right\| d\omega(t)\right) \frac{1}{2^n}.
 \end{align}
 Hence we may extend $Y$ onto $(x_{k-1},x_k)$ by 
 \begin{align*}
  Y(x) = \sum_{n=0}^\infty Z_n(x), \quad x\in(x_{k-1},x_k),
 \end{align*}
 where the sum converges absolutely and uniformly. 
 As above one shows that $Y$ is a solution of~\eqref{eqn:MIE} on $(x_{k-1},\beta)$.
 Now if we set 
 \begin{align}\label{eqnappLMDEYjump}
   Y(x_{k-1}) = (I-\omega(\lbrace x_{k-1}\rbrace)M(x_{k-1}))^{-1}\left(Y(x_{k-1}+) + \omega(\lbrace x_{k-1}\rbrace) F(x_{k-1})\right)
 \end{align}
 (note that the right-hand limit exists because of~\eqref{eqn:MIE}), then it is easy to show that $Y$ satisfies~\eqref{eqn:MIE} for all $x\in[x_{k-1},\beta)$.
 After finitely many steps we arrive at a solution $Y$, 
  satisfying~\eqref{eqn:MIE} for all $x\in(\alpha,\beta)$.
 
 Finally, if the data $M$, $F$ and $Y_c$ are real, one easily sees that all quantities in the construction stay real.
\end{proof}

The proof of Theorem~\ref{thm:appMIEEE} shows that condition~\eqref{eqn:appImOreg} is actually only needed for 
 all points $x$ to the left of the initial point $c$.
 Indeed, it is always possible to extend solutions uniquely to the right of the initial point but not to the left.
For a counterexample take $n=1$, the interval $(-2,2)$, $y_0\in\C$ and $\omega = -\delta_{-1} - \delta_1$, where $\delta_{\pm 1}$ are the Dirac measures in $\pm 1$.
Then one easily checks that the integral equation
\begin{align*}
 y(x) = y_0 + \int_0^x y(t) d\omega(t), \quad x\in(-2,2)
\end{align*}
has the solutions
\begin{align*}
 y_d(x) = \begin{cases}
           d, & \text{if } x\in(-2,-1], \\
           y_0, & \text{if } x\in(-1,1], \\
           0, & \text{if } x\in(1,2),
          \end{cases}
\end{align*}
for each $d\in\C$. Hence we see that the solutions are not unique to the left of the initial point $c=0$.

\begin{corollary}\label{corappIniValProRC}
 Assume~\eqref{eqn:appImOreg} holds for each $x\in(a,b)$. Then for each $c\in(a,b)$ and $Y_c\in\C^n$, the 
  initial value problem
\begin{align}
 \frac{dY}{d\omega} = MY + F \quad\text{with}\quad Y(c+) = Y_c
\end{align}
has a unique solution. If $M$, $F$ and $Y_c$ are real, then the solution is real.
\end{corollary}

\begin{proof}
 Some function $Y$ is a solution of this initial value problem if and only if it is a solution of 
\begin{align*}
 \frac{dY}{d\omega} = MY+F \quad\text{with}\quad Y(c) = \left(I+\omega(\lbrace c\rbrace)M(c)\right)^{-1}\left(Y_c-\omega(\lbrace c\rbrace) F(c)\right).
\end{align*}
\end{proof}

\begin{theorem}\label{thm:appregep}
 Assume $\|M(\cdot)\|$ and $\|F(\cdot)\|$ are integrable with respect to $\omega$ over $(a,c)$ for some $c\in(a,b)$ and $Y$ is a solution of the
  initial value problem~\eqref{eqnappLMDE}. Then the limit
 \begin{align}
  Y(a) := \lim_{x\rightarrow a} Y(x)
 \end{align}
 exists and is finite. A similar result holds for the endpoint $b$.
\end{theorem}

\begin{proof}
By assumption there is a $c\in(a,b)$ such that 
\begin{align*}
 \int_a^c \left\|M(t)\right\| d\omega(t) \leq \frac{1}{2}.
\end{align*}
We first prove that $\|Y(\cdot)\|$ is bounded near $a$.
Indeed if it was not, we had a monotone sequence $x_n\in(a,c)$ with $x_n\downarrow a$ such that $\|Y(x_n)\|\geq \|Y(x)\|$, $x\in[x_n,c]$. 
From the integral equation which $Y$ satisfies we would get
\begin{align*}
 \|Y(x_n)\| & \leq \|Y(c)\| + \int_{x_n}^c \|M(t)\| \|Y(t)\| d\omega(t) + \int_{x_n}^c \|F(t)\| d\omega(t) \\
        & \leq \|Y(c)\| +\|Y(x_n)\| \int_{x_n}^c \|M(t)\|d\omega(t) + \int_a^c \|F(t)\| d\omega(t) \\
        & \leq \|Y(c)\| + \int_a^c \|F(t)\| d\omega(t) + \frac{1}{2} \|Y(x_n)\|.
\end{align*}
Hence $\|Y(\cdot)\|$ has to be bounded near $a$ by some constant $K$.
Now the claim follows because we have
\begin{align*}
 \left\| Y(x_1) - Y(x_2) \right\| & = \left\| \int_{x_2}^{x_1} M(t) Y(t) + F(t)d\omega(t) \right\| \\
                 & \leq K \int_{x_1}^{x_2} \left\|M(t)\right\| d\omega(t) + \int_{x_1}^{x_2} \|F(t)\| d\omega(t)
\end{align*}
 for each $x_1$, $x_2\in(a,c)$, $x_1<x_2$ i.e., $Y(x)$ is a Cauchy-sequence as $x\rightarrow a$.
\end{proof}

Under the assumption of Theorem~\ref{thm:appregep}
 one can show that there is always a unique solution of the initial value problem
\begin{align*}
 \frac{dY}{d\omega} = MY + F \quad\text{with}\quad Y(a) = Y_a
\end{align*}
 with essentially the same proof as for Theorem~\ref{thm:appMIEEE}.
If $\|M(\cdot)\|$ and $\|F(\cdot)\|$ are integrable near $b$, then one furthermore has to assume that~\eqref{eqn:appImOreg} holds for 
 all $x\in(a,b)$ in order to get unique solutions of the initial value problem
\begin{align*}
 \frac{dY}{d\omega} = MY + F \quad\text{with}\quad Y(b) = Y_b.
\end{align*}

In the following let $M_1$, $M_2$ be $\C^{n\times n}$-valued measurable functions on $(a,b)$ such 
 that $\|M_1(\cdot)\|$, $\|M_2(\cdot)\|\in L^1_{\mathrm{loc}}((a,b);\omega)$. We are interested in the analytic dependence on $z\in\C$ 
 of solutions to the initial value problems
 \begin{align}\label{eqnappLMDEwithz}
  \frac{dY}{d\omega} = \left(M_1 + zM_2\right) Y + F \quad\text{with}\quad Y(c)=Y_c.
 \end{align}

\begin{theorem}\label{thmappLMDEAnaly}
 Assume~\eqref{eqn:appImOreg} holds for each $x\in(a,b)$. If for each $z\in\C$, $Y_z$ is the 
  unique solution of~\eqref{eqnappLMDEwithz}, then $Y_z(x)$ is analytic for each $x\in(a,b)$.
\end{theorem}

\begin{proof}
We show that the construction in the proof of Theorem~\ref{thm:appMIEEE} yields analytic solutions.
 Indeed, let $\alpha$, $\beta\in(a,b)$ with $\alpha<c<\beta$ as in the proof of Theorem~\ref{thm:appMIEEE}.
 Then the now $z$ dependent functions $Y_{z,n}(x)$, $n\in\N_0$ 
 (defined as in~\eqref{eqnappLMDEY0} and~\eqref{eqnappLMDEYn}) are polynomials in $z$ for each fixed point $x\in(c,\beta)$. 
 Furthermore, the sum $\sum_{n=0}^\infty Y_{z,n}(x)$ converges locally uniformly in $z$ 
 by~\eqref{eqnappLMDEboundYn} which proves that $Y_z(x)$ is analytic. 
Now in order to prove analyticity to the left of $c$ fix some $R\in\R^+$. 
 Then there are some points $x_k\in[\alpha,c]$, $k=-N,\ldots,0$ as in the proof of 
  Theorem~\ref{thm:appMIEEE} such that~\eqref{eqn:appomegabound} holds for all $M=M_1+zM_2$, $|z|<R$.
 It suffices to prove that if $Y_z(x_k)$ is analytic for some $-N<k\leq0$ then $Y_z(x)$ is analytic for all 
 $x\in[x_{k-1},x_k)$. Indeed, for each $n\in\N_0$ and $x\in(x_{k-1},x_k)$ the functions $Z_{z,n}(x)$ 
 (defined as in~\eqref{eqnappLMDEZ0} and~\eqref{eqnappLMDEZn}) are 
 analytic and locally bounded in $|z|<R$. From the bound~\eqref{eqnappLMDEboundZn} one sees that 
 $\sum_{n=0}^\infty Z_{z,n}(x)$ converges locally uniformly in $|z|<R$. Hence $Y_{z}(x)$ is analytic in $\C$. 
Furthermore, $Y_z(x_{k-1})$ is analytic by~\eqref{eqnappLMDEYjump} (note that $Y_z(x_{k-1}+)$ is also analytic 
 since $Y_z(x)$ is bounded locally uniformly in $z$ to the right of $x_{k-1}$).
\end{proof}

Under the assumptions of the last theorem we even see that the right-hand limit $Y_z(x+)$ is analytic for each 
 $x\in(a,b)$. In fact, this follows since
 \begin{align*}
  Y_z(x+) = \lim_{\xi \downarrow x} Y_z(\xi), \quad z\in\C
 \end{align*}
 and $Y_z(x)$ is bounded locally uniformly in $x$ and $z$.
Furthermore, one can show (see the proof of Corollary~\ref{corappIniValProRC}) that if for each $z\in\C$, 
 $Y_z$ is the solution of the initial value problem
\begin{align*}
 \frac{dY}{d\omega} = \left(M_1 + zM_2\right) Y + F \quad\text{with}\quad Y(c+)=Y_c,
\end{align*}
 then $Y_z(x)$ as well as $Y_z(x+)$ are analytic in $z\in\C$ for each $x\in(a,b)$.

\section{Linear relations in Hilbert spaces}\label{appLR}

Let $X$ and $Y$ be linear spaces over $\C$. A linear relation of $X$ into $Y$ is a linear subspace of $X\times Y$. 
 The space of all linear relations of $X$ into $Y$ is denoted by $\LRel(X,Y)$.
 Linear relations generalize the notion of linear operators. 
 Indeed, if $D$ is a linear subspace of $X$ and $T:D\rightarrow Y$ is a linear operator, then we may identify 
 $T$ with its graph, which is a linear relation of $X$ into $Y$. In this way any operator can be regarded as a linear relation. 
 Motivated by this point of view, we define the domain, range, kernel and multi-valued part of some linear relation $T\in \LRel(X,Y)$ as
 \begin{align*}
  \dom{T} & = \lbrace x\in X \,|\, \exists y\in Y: (x,y)\in T\rbrace, \\
  \ran(T) & = \lbrace y\in Y \,|\, \exists x\in X: (x,y)\in T\rbrace, \\
  \ker(T) & = \lbrace x\in X \,|\, (x,0)\in T\rbrace, \\
  \mul{T} & = \lbrace y\in Y \,|\, (0,y)\in T\rbrace. 
 \end{align*}
 Note that some relation $T$ is (the graph of) an operator if and only if $\mul{T}=\lbrace 0\rbrace$.
 In this case these definitions are consistent with the usual definitions for operators.
 
 Again motivated by an operator theoretic viewpoint, we define the following algebraic operations.
 For $T$, $S\in\LRel(X,Y)$ and $\lambda\in\C$ we set
 \begin{align*}
  T + S = \lbrace (x,y)\in X\times Y \,|\, \exists y_1, y_2\in Y: (x,y_1)\in T, (x,y_2)\in S, y=y_1+y_2\rbrace
 \end{align*}
and
 \begin{align*}
  \lambda T = \lbrace (x,y)\in X\times Y \,|\, \exists y_0\in Y: (x,y_0)\in T, y=\lambda y_0 \rbrace.
 \end{align*}
 It is simple to check that both, $T+S$ and $\lambda T$ are linear relations of $X$ into $Y$.
 Moreover, we can define the composition of two linear relations. If $T\in\LRel(X,Y)$ and 
  $S\in\LRel(Y,Z)$ for some linear space $Z$, then
 \begin{align*}
  S T = \lbrace (x,z)\in X\times Z \,|\, \exists y\in Y: (x,y)\in T, (y,z)\in S\rbrace
 \end{align*}
 is a linear relation of $X$ into $Z$.
 One may even define an inverse of a linear relation $T\in\LRel(X,Y)$ by
 \begin{align*}
  T^{-1} = \lbrace (y,x)\in Y\times X \,|\, (x,y) \in T\rbrace,
 \end{align*}
 as a linear relation of $Y$ into $X$. For further properties of these algebraic operations of linear relations er refer to~\cite[2.1~Theorem]{arens}, \cite[Chapter~1]{cross} or \cite[Appendix~A]{haase}.

 From now on assume $X$ and $Y$ are Hilbert spaces with inner products $\spr{\,\cdot\,}{\cdot\,}_X$ and $\spr{\,\cdot\,}{\cdot\,}_Y$.    
  The adjoint of a linear relation $T\in\LRel(X,Y)$ is given by
 \begin{align*}
  T^\ast = \lbrace (y,x)\in Y\times X \,|\, \forall (u,v)\in T: \spr{u}{x}_X = \spr{v}{y}_Y \rbrace;
 \end{align*}
  it is a linear relation of $Y$ into $X$. 
  The adjoint of a linear relation is always closed, i.e., $T^\ast$ is closed with respect to the product topology on $Y\times X$. 
  Moreover, one has  
  \begin{align}\label{AppLReqnAdj}
   T^{\ast\ast}=\overline{T}, \quad \ker(T^\ast) = \ran(T)^\bot \quad\text{and}\quad  \mul{T^\ast} = \dom{T}^\bot.
  \end{align}
  If $S\in\LRel(X,Y)$ is another linear relation we also have
  \begin{align}
   T\subseteq S \quad\Rightarrow\quad T^\ast \supseteq S^\ast.
  \end{align} 
 All these properties of adjoints may be found for example in \cite[Section~3]{arens} or in \cite[Proposition~C.2.1]{haase}.

 Now let $T$ be a closed linear relation of $X$ into $X$.
  The resolvent set $\rho(T)$ of $T$ consists of all numbers $z\in\C$ such that $R_z=(T-z)^{-1}$ is an everywhere defined operator.
  Here $T-z$ is short-hand for the relation $T-zI$, where $I$ is the identity operator on $X$.
  The mapping $z\mapsto R_z$ on $\rho(T)$, called the resolvent of $T$, has the following properties (see e.g.\ \cite[Section~VI.1]{cross} or \cite[Proposition~A.2.3]{haase}).

 \begin{theorem}\label{AppLRthmres}
  The resolvent set $\rho(T)$ is open and the resolvent identity
 \begin{align}
  R_z - R_w = (z-w)R_z R_w, \quad z,\,w\in\rho(T)
 \end{align}
 holds.
 Moreover, the resolvent is analytic as a mapping into the space of everywhere defined, bounded linear 
  operators on $X$, equipped with the operator norm.
 \end{theorem}

 The spectrum $\sigma(T)$ of a closed linear relation $T$ is the complement of the resolvent set.
 One may divide the spectrum into three disjoint parts.
 \begin{align*}
  \sigma_p(T) & =\lbrace \lambda\in\sigma(T) \,|\, \ker(T-\lambda) \not=\lbrace0\rbrace\rbrace, \\
  \sigma_c(T) & =\lbrace \lambda\in\sigma(T) \,|\, \ker(T-\lambda) =\lbrace0\rbrace,\, 
             \ran(T-\lambda)\not=X,\, \overline{\ran(T-\lambda)}=X \rbrace, \\
  \sigma_r(T) & = \lbrace \lambda\in\sigma(T) \,|\, \ker(T-\lambda)=\lbrace 0\rbrace,\, 
                \overline{\ran(T-\lambda)}\not=X \rbrace.
 \end{align*}
 The set $\sigma_p(T)$ is called the point spectrum, $\sigma_c(T)$ is the continuous spectrum and $\sigma_r(T)$ is the 
  residual spectrum of $T$.
 Elements of the point spectrum are called eigenvalues. The spaces $\ker(T-\lambda)$ corresponding to some eigenvalue $\lambda$ are called eigenspaces,
  the non zero elements of the eigenspaces are called eigenvectors.
 
 We need a variant of the spectral mapping theorem for the resolvent (see e.g.\ \cite[Section~VI.4]{cross} or \cite[Proposition~A.3.1]{haase}).
 
 \begin{theorem}\label{AppLRthmSMT}
   For each $z\in\rho(T)$ we have
   \begin{align}
    \sigma\left(R_z\right)\backslash\lbrace 0\rbrace = \left\lbrace \left. \frac{1}{\lambda - z} \,\right|\, \lambda\in\sigma\left(T\right) \right\rbrace.
   \end{align}
 \end{theorem}

 A linear relation $T$ is said to be symmetric provided that $T\subseteq T^\ast$.
 If $T$ is a closed symmetric linear relation, we have $\rho(T)\subseteq\reg\left(T\right)$ and $\C\backslash\R\subseteq\reg(T)$, where
 \begin{align*}
  \reg(T) = \lbrace z\in\C \,|\, (T-z)^{-1} \text{ is a bounded operator}\rbrace
 \end{align*}
 denotes the points of regular type of $T$.
 Moreover, a linear relation $S$ is said to be self-adjoint, if $S=S^\ast$.
  In this case $S$ is closed, the spectrum of $S$ is real and from~\eqref{AppLReqnAdj} one sees that 
  \begin{align}
     \mul{S} = \dom{S}^\bot \quad\text{and}\quad \ker(S) = \ran(S)^\bot.
  \end{align}
  In particular, $S$ is an operator if and only if its domain is dense.
   Furthermore, 
   \begin{align}\label{eqnLRSDef}
     S_\D = S \cap \left(\D\times\D\right)
   \end{align}
  is a self-adjoint operator in the Hilbert space $\D$, where $\D$ is the closure of the domain of $S$.
 These properties of symmetric and self-adjoint linear relations may be found in \cite[Theorem~3.13, Theorem~3.20 and Theorem~3.23]{dijksnoo}.
 Moreover, the following result shows that $S$ and $S_\D$ (as an operator in the Hilbert space $\D$) have many spectral properties in common.
 
 \begin{lemma}\label{lemSARelOper}
  $S$ and $S_\D$ have the same spectrum and 
  \begin{align}
   R_z f = (S_\D-z)^{-1} P f, \quad f\in X,~z\in\rho(S),
  \end{align}
  where $P$ is the orthogonal projection onto $\D$. 
  Moreover, the eigenvalues as well as the corresponding eigenspaces coincide.
 \end{lemma}
 
 \begin{proof}
  From the equalities
  \begin{align*}
   \ran(S_\D-z) = \ran(S-z) \cap \D \quad\text{and}\quad \ker(S_\D-z) = \ker(S-z), \quad z\in\C
  \end{align*}
  one sees that $S$ and $S_\D$ have the same spectrum as well as the same point spectrum and corresponding eigenspaces. 
 Now let $z\in\rho(S)$, $f\in X$ and set $g=(S-z)^{-1}f$, i.e., $(g,f)\in S-z$.
   If $f\in\D$, then since $g\in\D$, we have $(g,f)\in S_\D-z$, i.e., $(S_\D-z)^{-1}f=g$.
  Finally note that if $f\in\D^\bot$, then $g=0$ since we have $\mul{S-z}=\mul{S}=\dom{S}^\bot$.
 \end{proof}
 
 Applying a variant of the spectral theorem to $S_\D$, we get the following result for the self-adjoint relation $S$.

\begin{lemma}\label{lemspectransEfg}
 For each $f$, $g\in X$ there is a unique complex Borel measure $E_{f,g}$ on $\R$ such that
 \begin{align}\label{eqnspectransresE}
  \spr{R_z f}{g}_X = \int_\R \frac{1}{\lambda-z} d E_{f,g}(\lambda), \quad z\in\rho(S). 
 \end{align}
 Moreover,
 \begin{align}
  \spr{Pf}{g}_X = \int_\R dE_{f,g}, \quad f,\,g\in X
 \end{align}
 and for each $f\in X$, $E_{f,f}$ is a positive measure with
 \begin{align}
 Pf\in\dom{S} \quad\Leftrightarrow\quad \int_\R |\lambda|^2dE_{f,f}(\lambda)<\infty. 
 \end{align}
 In this case 
 \begin{align} 
 \spr{f_S}{Pg}_X = \int_\R \lambda dE_{f,g}(\lambda), 
 \end{align}
 whenever $(Pf,f_S)\in S$.
\end{lemma}

\begin{proof}
 Since $S_\D$ is a self-adjoint operator in $\D$, there is an operator-valued spectral measure $E$ such that for all $f$, $g\in\D$
 \begin{align*}
  \spr{(S_{\D}-z)^{-1} f}{g}_X = \int_\R \frac{1}{\lambda-z} d E_{f,g}(\lambda), \quad z\in\rho(S_{\D}),
 \end{align*}
 where $E_{f,g}$ is the complex measure given by $E_{f,g}(B) = \spr{E(B) f}{g}_X$, for each Borel set $B$.
 For arbitrary $f$, $g\in X$ we set $E_{f,g}=E_{Pf,Pg}$.
 Because of Lemma~\ref{lemSARelOper} we get for each $z\in\rho(S)$ the claimed equality
 \begin{align*}
  \spr{R_zf}{g}_X & = \spr{R_z Pf}{Pg}_X = \spr{(S_\D-z)^{-1} Pf}{Pg}_X = \int_\R \frac{1}{\lambda-z} dE_{Pf,Pg}(\lambda) \\
                & = \int_\R \frac{1}{\lambda-z} dE_{f,g}(\lambda),
 \end{align*}
 where we used $R_z = PR_zP$ (see~\eqref{eqnSpecRessets}). 
 Uniqueness of these measures follows from the Stieltjes inversion formula.
 The remaining claims follow from the fact that $E$ is the spectral measure of $S_{\D}$.
\end{proof}

 We are interested in self-adjoint extensions of symmetric relations in $X$.
 Therefore, let $T$ be a closed symmetric linear relation in $X\times X$. The linear relations
 \begin{align*}
  N_\pm(T) & = \lbrace (x,y)\in T^\ast \,|\, y=\pm \I x \rbrace
 \end{align*}
 are called deficiency spaces of $T$. Note that $N_\pm(T)$ are operators with 
 \begin{align*}
  \dom{N_\pm(T)} = \ran(T \mp  \I)^\bot = \ker(T^\ast \pm  \I).
 \end{align*}
 Furthermore, one has an analog of the first von Neumann formula (see e.g.\ \cite[Theorem~6.1]{dijksnoo})
 \begin{align}\label{AppLReqnNF}
  T^\ast = T \oplus N_+(T) \oplus N_-(T),
 \end{align}
 where the sums are orthogonal with respect to the usual inner product 
 \begin{align*}
  \spr{(f_1,f_2)}{(g_1,g_2)}_{X\times X} = \spr{f_1}{g_1}_X + \spr{f_2}{g_2}_X, \quad (f_1, f_2),\, (g_1,g_2)\in X\times X
 \end{align*}
 on $X\times X$.
 Now the existence of self-adjoint extensions of $T$ is determined by these subspaces (see e.g.\ \cite[Theorem~15]{coddsymext} or \cite[Corollary~6.4]{dijksnoo}).

 \begin{theorem}\label{AppLRthmSAext}
  The closed symmetric linear relation $T$ has a self-adjoint extension if and only if the dimensions of the deficiency subspaces are equal.
  In this case all self-adjoint extensions $S$ of $T$ are of the form
   \begin{align}\label{AppLReqnSAext}
    S = T \oplus (I-V)N_+(T),
   \end{align}
  where $V$ is an isometry from $N_+(T)$ onto $N_-(T)$. 
  Conversely, for each such isometry $V$ the linear relation $S$ given by~\eqref{AppLReqnSAext} is self-adjoint.
 \end{theorem}

 In particular, we are interested in the case when the deficiency subspaces are finite-dimensional, i.e.,
 \begin{align*}
  n_\pm(T) = \dim N_\pm(T) < \infty.
 \end{align*}
 The numbers $n_\pm(T)$ are called deficiency indices of $T$. 
 
 \begin{corollary}
  If $T$ has equal and finite deficiency indices, i.e.,
   \begin{align*}
    n_-(T)=n_+(T)=n\in\N, 
   \end{align*}
  then the self-adjoint extensions of $T$ are precisely the $n$-dimensional symmetric extensions of $T$.
 \end{corollary}

 \begin{proof}
  By Theorem~\ref{AppLRthmSAext} each self-adjoint extension is an $n$-dimensional symmetric extension of $T$,
   since $\dim (I-V)N_+(T)=n$.
  Conversely, assume that $S$ is an $n$-dimensional symmetric extension of $T$, i.e., $S = T \dot{+} M$ for some
   $n$-dimensional symmetric subspace $N$. Then since 
   \begin{align*}
   \dim\ran(N\mp  \I) = n = \dim X/\ran(T\mp \I),
   \end{align*}
   (note that $(N\mp \I)^{-1}$ is an operator) we get
   \begin{align*}
    \ran(S\mp  \I) = \ran(T\mp \I) \dot{+} \ran(N\mp  \I) = X.
   \end{align*}
  Hence we have $\dim N_\pm(S)=0$ and therefore $S^\ast=S$ in view of~\eqref{AppLReqnNF}.
 \end{proof}

 The essential spectrum $\sigess(S)$ of a self-adjoint linear relation $S$ 
 consists of all eigenvalues of infinite multiplicity and all accumulation points of the spectrum. 
 Moreover, the discrete spectrum $\sigdis(S)$ of $S$ consists of all eigenvalues of $S$ with finite
  multiplicity which are isolated points of the spectrum of $S$.
 From Lemma~\ref{lemSARelOper} one immediately sees that
 \begin{align*}
  \sigess(S) = \sigess(S_\D) \quad\text{and}\quad \sigdis(S) = \sigdis(S_\D).
 \end{align*}
 Using the former equality, we are able to deduce the following two theorems on stability of the essential spectrum from the corresponding results for operators.

 \begin{theorem}\label{AppLRthmSAEessspec}
  Let $T$ be a symmetric relation with equal and finite deficiency indices $n$ and $S_1$, $S_2$ be self-adjoint
   extensions of $T$. Then the operators
  \begin{align}
    \left(S_1\pm\I\right)^{-1} - \left(S_2\pm\I\right)^{-1}
  \end{align}
  are at most $n$-dimensional. In particular, we have
  \begin{align}
    \sigess\left(S_1\right) = \sigess\left(S_2\right).
  \end{align}
 \end{theorem}

 \begin{proof}
  Because of $\dim\ran(T\pm\I)^\bot=n$ and
  \begin{align*}
   \left(S_1\pm\I\right)^{-1} f = \left(T\pm\I\right)^{-1} f = \left(S_2\pm\I\right)^{-1} f, \quad f\in\overline{\ran(T\pm\I)},
  \end{align*}
  the difference of the resolvents is at most $n$-dimensional. Now the remaining claim follows from Lemma~\ref{lemSARelOper} and~\cite[Theorem~6.19]{tschroe}.
 \end{proof}

 \begin{theorem}\label{AppLRthmOSessspec}
  Let $X_1$, $X_2$ be closed subspaces of $X$ such that $X=X_1\oplus X_2$. 
  If $S_1$ is a self-adjoint linear relation in $X_1$ and $S_2$ is a self-adjoint linear relation in $X_2$, then $S_1\oplus S_2$ is a self-adjoint linear relation in $X$ with
   \begin{align}
    \sigess\left(S_1\oplus S_2\right) = \sigess\left(S_1\right)\cup\sigess\left( S_2\right).
   \end{align}
 \end{theorem}
 
 \begin{proof}
  A simple calculation shows that $(S_1\oplus S_2)^\ast = S_1^\ast\oplus S_2^\ast = S_1\oplus S_2$. Since 
  \begin{align*}
   \D = \overline{\dom{S_1\oplus S_2}} = \overline{\dom{S_1}} \oplus \overline{\dom{S_2}} = \D_1 \oplus \D_2
  \end{align*}
  and
  \begin{align*}
   \left(S_1\oplus S_2\right)_\D = S_{1\D_1} \oplus S_{2\D_2},
  \end{align*}
  the claim follows from the corresponding result for operators.
 \end{proof}

\section{One dimensional Sturm--Liouville problems}\label{app:onedim}

For the sake of completeness in this section we consider the case when $\varrho$ is not necessarily supported on more than one point, 
 i.e., we only assume that~\eqref{genhypfirst} to~\eqref{genhypgap} of Hypothesis~\ref{genhyp} hold.
Because of the lack of the identification of Proposition~\ref{prop:identDeftauTloc} in this case, 
 we make the following definition.
Some linear subspace $S\subseteq\Deftau$ is said to give rise to a self-adjoint relation if the map
\begin{align}\label{eqnOneDimMap}
 \begin{matrix}
  S & \rightarrow & \Lr \times\Lr \\
  f & \mapsto     & (f,\tau f)
 \end{matrix}
\end{align}
is well-defined, injective and its range is a self-adjoint relation of $\Lr$ into $\Lr$.
By the identification of Proposition~\ref{prop:identDeftauTloc} one sees that we already determined all 
 linear subspaces of $\Deftau$ which give rise to a self-adjoint relation if $\varrho$ is supported on more than one point.
 Hence we need only consider the case when $\varrho$ is supported on only one point. 
 Indeed, we will do this by proving a version of Theorem~\ref{thm:SRLCLCsepcoup} (note that $\tau$ is in the l.c.~case at both endpoints).
 Therefore, assume in the following $\varrho=\varrho_0\delta_{x_0}$ for some $\varrho_0\in\R^+$ and $x_0\in(a,b)$.
 In this case each function $f\in\Deftau$ is of the form
 \begin{align*}
  f(x) = \begin{cases}
          u_a(x), & \text{if } x\in(a,x_0], \\
          u_b(x), & \text{if } x\in(x_0,b),
         \end{cases}
 \end{align*}
 where $u_a$ and $u_b$ are solutions of $\tau u=0$ with $u_a(x_0-)=u_b(x_0+)$, i.e., $f$ is continuous in $x_0$ 
  but in general the quasi-derivative $f^\qd$ is not.
 In this case $\tau f$ is given by
\begin{align}\label{eqnOneDimTauF}
 \tau f(x_0) = \frac{1}{\varrho_0} \left( -f^\qd(x_0+) + f^\qd(x_0-) + f(x_0)\chi(\lbrace x_0\rbrace)\right).
\end{align}
Furthermore, for two functions $f$, $g\in\Deftau$, the limits
\begin{align*}
 W(f,g)(a) := \lim_{x\rightarrow a} W(f,g)(x) \quad\text{and}\quad W(f,g)(b) := \lim_{x\rightarrow b}W(f,g)(x)
\end{align*}
exist and are finite. In fact, the Wronskian is constant away from $x_0$.
Now as in Section~\ref{secBC} let $w_1$, $w_2\in\Deftau$ with 
\begin{align*}
 W(w_1,w_2^\ast)(a) &= 1 & \text{and} & & W(w_1,w_1^\ast)(a) = W(w_2,w_2^\ast)(a) & = 0, \\
 W(w_1,w_2^\ast)(b) &= 1 & \text{and} & & W(w_1,w_1^\ast)(b) = W(w_2,w_2^\ast)(b) & = 0,
\end{align*}
and define the linear functionals $\BCa^1$, $\BCa^2$, $\BCb^1$ and $\BCb^2$ on $\Deftau$ by
 \begin{align*}
  \BCa^1(f) &= W(f,w_2^\ast)(a) &\text{and} && \BCa^2(f) & =W(w_1^\ast,f)(a) &\text{for }f\in\Deftau, \\
  \BCb^1(f) &= W(f,w_2^\ast)(b) &\text{and} && \BCb^2(f) & =W(w_1^\ast,f)(b) &\text{for }f\in\Deftau.
 \end{align*}
Again one may choose special functions $w_1$, $w_2$ as in Proposition~\ref{prop:BCLefthandlim}.

\begin{theorem}\label{thmOneDimSRSep}
 Let $S\subseteq\Deftau$ be a linear subspace of the form
 \begin{align}\label{eqn:apponedimsep}
 S = \left\lbrace f\in\Deftau \left| \begin{array}{l} \BCa^1(f)\cos\varphi_\alpha-\BCa^2(f)\sin\varphi_\alpha=0 \\
         \BCb^1(f)\cos\varphi_\beta-\BCb^2(f)\sin\varphi_\beta=0 \end{array}\right.\right\rbrace
 \end{align} 
 for some $\varphi_\alpha$, $\varphi_\beta\in[0,\pi)$. 
 Then $S$ gives rise to a self-adjoint relation if and only if one of the following conditions
 \begin{subequations}
 \begin{align}\label{eqn:appsepnotinja}
   w_2(x_0-)\cos\varphi_\alpha + w_1(x_0-)\sin\varphi_\alpha & \not=0,  \\
 \label{eqn:appsepnotinjb}
   w_2(x_0+)\cos\varphi_\beta + w_1(x_0+)\sin\varphi_\beta & \not=0,
 \end{align}
 \end{subequations}
 holds. This relation is an operator if and only if~\eqref{eqn:appsepnotinja} and~\eqref{eqn:appsepnotinjb} hold.
\end{theorem}

\begin{proof}
 The boundary conditions can by written as
 \begin{align*}
  W(f,w_2^\ast \cos\varphi_\alpha + w_1^\ast \sin\varphi_\alpha)(x_0-) & =0, \\
  W(f,w_2^\ast \cos\varphi_\beta + w_1^\ast \sin\varphi_\beta)(x_0+) & =0.
 \end{align*}
 From this one sees that the mapping~\eqref{eqnOneDimMap} is injective if and only if one of the 
 inequalities~\eqref{eqn:appsepnotinja} or~\eqref{eqn:appsepnotinjb} holds.
 Hence for the first part it remains to show that in this case the range of the mapping~\eqref{eqnOneDimMap} 
  is a self-adjoint relation. First consider the case when both inequalities hold.
 Then we get from the boundary conditions
 \begin{align*}
  f^\qd(x_0-) = f(x_0) \frac{\cos\varphi_\alpha w_2^\qd(x_0-)^\ast + \sin\varphi_\alpha w_1^\qd(x_0-)^\ast}
                                {\cos\varphi_\alpha w_2(x_0-)^\ast + \sin\varphi_\alpha w_1(x_0-)^\ast}, \quad f\in S
 \end{align*}
 and similarly for the right-hand limit. A simple calculation shows that the imaginary part of this 
  fraction as well as the imaginary part of the corresponding fraction for the right-hand limit vanish.
 Hence from~\eqref{eqnOneDimTauF} we infer that the range of the mapping~\eqref{eqnOneDimMap} is a 
  self-adjoint operator (multiplication with a real scalar).
 Now in the case when one inequality, say~\eqref{eqn:appsepnotinja} does not hold, we get $f(x_0)=0$ for each
  $f\in S$ from the boundary condition at $a$. Hence it suffices to prove that $\tau f(x_0)$ takes each value
  in $\C$ if $f$ runs through $S$, i.e., $S$ corresponds to the self-adjoint, multi-valued relation 
  $\lbrace 0\rbrace\times\Lr$. But this follows since all functions of the form 
 \begin{align*}
  f(x) = \begin{cases}
          u_a(x), & \text{if }x\in(a,x_0], \\
          0,      & \text{if }x\in(x_0,b),
         \end{cases}
 \end{align*}
 where $u_a$ is a solution of $\tau u=0$ with $u_a(x_0)=0$, lie in $S$.
\end{proof}

The preceding theorem corresponds to separated boundary conditions. Next we discuss the case of coupled boundary conditions.

\begin{theorem}\label{thmOneDimSRCoup}
 Let $S\subseteq\Deftau$ be a linear subspace of the form
 \begin{align}
 S = \left\lbrace f\in\Deftau \left|\, \left(\begin{matrix}\BCb^1(f)\\\BCb^2(f)\end{matrix}\right) 
      = e^{\I\varphi}R\left(\begin{matrix}\BCa^1(f)\\\BCa^2(f)\end{matrix}\right)\right.\right\rbrace
 \end{align} 
 for some $\varphi\in[0,\pi)$ and $R\in\R^{2\times2}$ with $\det{R}=1$, and set
 \begin{align*}
  \tilde{R} = \begin{pmatrix}
                    w_2^\qd(x_0+)^\ast & -w_2(x_0+)^\ast \\ -w_1^\qd(x_0+)^\ast & w_1(x_0+)^\ast
                  \end{pmatrix}^{-1}
     R \begin{pmatrix}
        w_2^\qd(x_0-)^\ast & -w_2(x_0-)^\ast \\ -w_1^\qd(x_0-)^\ast & w_1(x_0-)^\ast
       \end{pmatrix}.
 \end{align*}
 Then $S$ gives rise to a self-adjoint relation if and only if
 \begin{align*}
  \tilde{R}_{12} \not=0
  \qquad\text{or}\qquad
  e^{\I\varphi} \tilde{R}_{11} \not= 1 \not= e^{\I\varphi} \tilde{R}_{22}.
 \end{align*}
 This relation is an operator if and only if $\tilde{R}_{12}\not=0$.
\end{theorem}

\begin{proof}
  The boundary conditions can be written as
 \begin{align*}
  \begin{pmatrix} f(x_0+) \\ f^\qd(x_0+) \end{pmatrix} = \E^{\I\varphi} \tilde{R} \begin{pmatrix} f(x_0-) \\ f^\qd(x_0-) \end{pmatrix}.
 \end{align*}
 First of all note that $\tilde{R}$ is a real matrix. Indeed, since for each $j=1,2$, $w_j$ and $w_j^\ast$ are solutions of $\tau u=0$ 
  on $(a,x_0)$ we see that they must be linearly dependent, hence we get $w_j(x)=w_j(x)^\ast$, $x\in(a,x_0)$.
  Of course the same holds to the right of $x_0$ and since $R$ is real also $\tilde{R}$ is real.
 If $\tilde{R}_{12}\not=0$, then the boundary conditions show that the mapping~\eqref{eqnOneDimMap} is
  injective. Furthermore, using~\eqref{eqnOneDimTauF} one gets
 \begin{align*}
  \tau f(x_0) \varrho_0 & = f(x_0) \frac{1-\E^{\I\varphi}\left(\tilde{R}_{11} + \tilde{R}_{22}\right) + 
      \E^{2i\varphi} \det\tilde{R}}{\E^{\I\varphi}\tilde{R}_{12}} + f(x_0)\chi(\lbrace x_0\rbrace),\quad f\in S.
 \end{align*}
 A simple calculation shows that $\det\tilde{R}=\det R=1$ and that the fraction is real. 
 Hence we see that $S$ gives rise to a self-adjoint, single-valued relation.
 
Now assume $\tilde{R}_{12}=0$ and $\E^{\I\varphi} \tilde{R}_{11} \not= 1 \not= \E^{\I\varphi} \tilde{R}_{22}$, 
  then again the boundary conditions show that the mapping~\eqref{eqnOneDimMap} is injective. Furthermore,
 they show that each function $f\in S$ satisfies $f(x_0)=0$. Hence it suffices to show that $\tau f(x_0)$ takes on 
  every value as $f$ runs through $S$. But this is true since all functions 
 \begin{align}\tag{$*$}\label{eqnOneDimCoupBCfunc}
  f_c(x) = \begin{cases}
            c u_a(x), & \text{if }x\in(a,x_0], \\
            c \E^{\I\varphi} \tilde{R}_{22} u_b(x), & \text{if }x\in(x_0,b),
           \end{cases}
 \end{align}
  where $c\in\C$ and $u_a$, $u_b$ are solutions of $\tau u=0$ with $u_a(x_0-)=u_b(x_0+)=0$ 
   and $u_a^\qd(x_0-)=u_b^\qd(x_0+)=1$, lie in $S$.
 If $\tilde{R}_{12}=0$ but $\E^{\I\varphi} \tilde{R}_{22} = 1$, then the mapping~\eqref{eqnOneDimMap} is not injective.
  Indeed, all functions of the form~\eqref{eqnOneDimCoupBCfunc} are mapped onto zero.
 Finally, if $\tilde{R}_{12}=0$ and $\E^{\I\varphi} \tilde{R}_{11} = 1\not= \E^{\I\varphi}\tilde{R}_{22}$, then 
  since $S$ is two-dimensional it does not give rise to a self-adjoint relation.
\end{proof}

Note that if we choose for $\BCa^1$, $\BCa^2$, $\BCb^1$ and $\BCb^2$ the functionals from Proposition~\ref{prop:BCLefthandlim},
 then we get $\tilde{R}=R$.

The resolvent of the self-adjoint relations given in Theorem~\ref{thmOneDimSRSep} and Theorem~\ref{thmOneDimSRCoup} can be written as in Section~\ref{secSR}.
 In fact, Theorem~\ref{thmSResolLCLC} and Corollary~\ref{corSpecRDis} are obviously valid since the resolvents are simply multiplication by some scalar.
 Moreover, Theorem~\ref{thm:ressep} and Corollary~\ref{corSpecRSimple} for self-adjoint relations as in Theorem~\ref{thmOneDimSRSep} may be proven along the same lines as in the general case.
 The remaining theorems of Section~\ref{secSR} are void of meaning here, since all self-adjoint relations have purely discrete spectrum.
Finally, the results of Sections~\ref{secweyltitchm} and~\ref{secST} are also valid for self-adjoint relations as in 
 Theorem~\ref{thmOneDimSRSep} since all proofs in these sections also apply in this simple case.

\bigskip
\noindent
{\bf Acknowledgments.}
We thank Fritz Gesztesy and Andrei Shkalikov for help with respect to the literature and the anonymous referee for valuable suggestions improving the presentation of the material.

\end{document}